\documentclass[smallextended]{svjour3}       

\usepackage{latexsym, amssymb, enumerate, amsmath}
\usepackage{defs}
\usepackage{fonts}
\usepackage{dsfont}
\usepackage{epsfig}
\usepackage{epstopdf}
\usepackage{caption}
\usepackage{booktabs}
\newcommand{\thickbreve}[1]{\mathbf{\breve{\text{$#1$}}}}
\usepackage{textcomp}

\newtheorem{assumption}{Assumption}
\newtheorem{condition}{Condition}

\usepackage{etoolbox}
\newcommand{\qedsym}{\null\hfill\(\Box\)}
\AtEndEnvironment{proof}{\qedsym}

\usepackage{lineno}

\usepackage{url}
\usepackage{cite}
\usepackage{multirow} 
\usepackage{afterpage}
\usepackage{enumitem}
\usepackage{mathrsfs}
\usepackage{subfig} 

\usepackage{algorithmicx}
\usepackage{xifthen}
\usepackage{needspace}
\usepackage{algorithm}
\usepackage{algpseudocode}

\usepackage{todonotes}
\graphicspath{{./figures_opt/}}

\newcounter{myalg}
\newenvironment{myalg}[1][]%
  {
    \needspace{2\baselineskip}
    \noindent \rule{\linewidth}{1pt} \endgraf
    \refstepcounter{myalg}
    \centering \textsc{Algorithm}~\themyalg%
    \ifthenelse{\isempty{#1}}{}{:\ #1}
  }{
  \noindent \rule{\linewidth}{1pt}
  }%

\title{A Constraint-Reduced MPC Algorithm for Convex Quadratic
Programming, with a Modified Active Set Identification Scheme%
\thanks
	{This manuscript has been authored, in part, by UT-Battelle, LLC, under Contract No. DE-AC0500OR22725 with the U.S. Department of Energy. The United States Government retains and the publisher, by accepting the article for publication, acknowledges that the United States Government retains a non-exclusive, paid-up, irrevocable, world-wide license to publish or reproduce the published form of this manuscript, or allow others to do so, for the United States Government purposes. The Department of Energy will provide public access to these results of federally sponsored research in accordance with the DOE Public Access Plan (\texttt{http://energy.gov/downloads/doe-public-access-plan}).}}
\date{\today}
\author{M. Paul Laiu%
\and
Andr{\'e} L.~Tits%
}

\institute{M. Paul Laiu \at
 Computational and Applied Mathematics Group,
				Computer Science and Mathematics Division,
				Oak Ridge National Laboratory,
				Oak Ridge, TN 37831 USA\\
              \email{laiump@ornl.gov} \\
              This author's research was sponsored by the Office of Advanced Scientific
Computing Research and performed at the Oak Ridge National Laboratory,
which is managed by UT-Battelle, LLC under Contract No. DE-AC05-00OR22725.
           \and
           Andr{\'e} L.~Tits \at
           Department of Electrical and Computer Engineering
  \& Institute for Systems Research,
  University of Maryland
  College Park, MD 20742 USA,\\
				\email{andre@umd.edu}
}

\titlerunning{A Constraint-Reduced MPC Algorithm for Convex Quadratic Programming}        

\begin{document}

\maketitle

\begin{abstract}
A constraint-reduced Mehrotra--Predictor--Corrector 
algorithm for convex quadratic programming
is proposed.
(At each iteration, such algorithms use only a subset of the inequality constraints in constructing 
the search direction, resulting in CPU savings.)  The proposed algorithm
makes use of a regularization scheme to cater to cases where the
reduced constraint matrix is rank deficient.  Global and local
convergence properties are established under arbitrary working-set
selection rules subject to satisfaction of a general condition.
A modified active-set identification scheme that fulfills this
condition is introduced.  Numerical tests show great promise 
for the proposed algorithm, in particular for its active-set 
identification scheme.  
While the focus of the present paper
is on dense systems, application of the main ideas to large sparse
systems is briefly discussed.
\end{abstract}


\section{Introduction}
\label{chap:CR-MPC}

Consider a strictly feasible convex quadratic program (CQP) in 
standard inequality form,%
\footnote{
See end of Sections~\ref{subsec:infeasible} and~\ref{subsec:conv_ana} 
below for a brief discussion of how linear
{\em equality} constraints can be incorporated.
}
\begin{equation}
\minimize_{\bx\in\bbR^n} \: f(\bx):=\frac{1}{2}\bx^TH\bx+\bc^T\bx
\quad \mbox{subject to} \:     A\bx\geq \bb\:,
\tag{P}
\label{eq:primal_cqp}
\end{equation}
where $\bx\in\bbR^n$ is the vector of optimization 
variables, $f\colon\bbR^n\to\bbR$ 
is the objective function with $\bc\in\bbR^n$ and 
$H\in\bbR^{n\times n}$ a symmetric positive semi-definite matrix,
$A\in\bbR^{m\times n}$ and $\bb\in\bbR^m$, with $m>0$, 
define the $m$ linear 
inequality constraints, and (here and elsewhere) all inequalities
($\geq$ or $\leq$) are meant component-wise.  $H$, $A$, and $\bc$ 
are not all zero.
The dual problem associated to \eqref{eq:primal_cqp} is
\begin{equation}
\maximize_{\bx\in\bbR^n,\,\bslambda\in\bbR^m}\:  -\frac{1}{2}\bx^TH\bx+\bb^T\bslambda
\quad \mbox{subject to}\:    H\bx+\bc-A^T\bslambda=\bzero, 
\quad \bslambda\geq\bzero\:,
\tag{D}
\label{eq:dual_cqp}
\end{equation}
where $\bslambda\in\bbR^m$ is the vector of 
multipliers.
Since the objective function $f$ is convex and the constraints are linear, solving \eqref{eq:primal_cqp}--\eqref{eq:dual_cqp} is equivalent to solving the Karush-Kuhn-Tucker (KKT) system
\begin{equation}
H\bx-A^T\bslambda+\bc=\bzero, \quad A\bx-\bb-\bs=\bzero,
\quad S\bslambda=\bzero, \quad \bs,\bslambda\geq\bzero\:,
\label{eq:KKT}
\end{equation}
where $\bs\in\bbR^n$ is a vector of slack variables 
associated to the inequality 
constraints in~\eqref{eq:primal_cqp}, 
and $S=\diag(\bs)$.
%

Primal--dual interior--point methods (PDIPM)
solve \eqref{eq:primal_cqp}--\eqref{eq:dual_cqp} by iteratively applying a     
Newton-like iteration to the three {\em equations} 
in~\eqref{eq:KKT}.
Especially popular for its numerical behavior is
S.~Mehrotra's predictor--corrector (MPC) variant, 
which was introduced in~\cite{Mehrotra-1992} for 
the case of linear optimization (i.e., $H=\bzero$)
with straightforward extension to CQP 
(e.g.,~\cite[Section 16.6]{Nocedal-Wright-2006}).  (Extension
to linear complementarity problems was studied
in~\cite{Zhang1995}.)

A number of authors have paid special attention to ``imbalanced''
problems, in which the number of active constraints 
at the solution is small compared to the total number of
constraints (in particular, cases in which $m\gg n$).
In solving such problems, while most constraints are
in a sense redundant, traditional IPMs devote much effort per
iteration to solving large systems of linear equations
that involve all $m$ constraints.  In the 1990s,
work toward reducing the computational
burden by using only a small portion (``working set'')
of the
constraints in the search direction computation focused
mainly on {\em linear} optimization~\cite{Ye1990,
DantzigYe91,Tone1993,DenHer:94}.  This was also the case
for~\cite{TAW-06},
which may have been the
first to consider such ``constraint-reduction'' schemes in 
the context of PDIPMs (vs.~purely dual interior--point methods), 
and for its extensions~\cite{HT12,WNTO-2012,WTA-2014}.  
Exceptions include the works of 
Jung {\em et al.}~\cite{JOT-08,JOT-12} and of 
Park {\em et al.}~\cite{Park-OLeary-2015,Park2016}; in the
former, an extension to CQP was considered, with an
affine-scaling variant used as the ``base'' algorithm; in
the latter, a constraint-reduced PDIPM for semi-definite
optimization (which includes CQP as a special case) was
proposed, for which polynomial complexity was proved.
Another exception is the ``QP-free'' algorithm for 
inequality-constrained (smooth) nonlinear 
programming of Chen {\em et al.}~\cite{CWH06}.
There, a constraint-reduction approach is used where
working sets are selected by means of the 
Facchinei--Fischer--Kanzow active set identification
technique~\cite{FFK98}.

In the linear-optimization case, the size of the working
set is usually kept above (or no lower than) the number
$n$ of variables in \eqref{eq:primal_cqp}.  It is indeed known
that, in that case, if the set of solutions to \eqref{eq:primal_cqp} is nonempty and bounded,
then a solution exists at
which at least $n$ constraints are active.  Further if fewer
than $n$ constraints are included in the linear system that
defines the search direction, then the default (Newton-KKT)
system matrix is structurally singular.

When $H\not=\bzero$ though, the number of active constraints at
solutions 
may be much less 
than $n$ and, at strictly feasible iterates, the Newton-KKT
matrix is non-singular whenever the subspace spanned by
the columns of $H$ and the working-set columns of $A^T$
has (full) dimension~$n$.  (In particular,
of course, if $H$ is non-singular (i.e., is positive definite)
the Newton-KKT matrix is non-singular even when the working
set is empty---in which case the unconstrained
Newton direction is obtained.)  Hence, when solving a CQP,
forcing the working set to have size at least $n$ is
usually wasteful.  

The present paper proposes a constraint-reduced MPC
algorithm for CQP.
The work
borrows from~\cite{JOT-12} (affine-scaling, convex
quadratic programming)
and is significantly inspired from~\cite{WNTO-2012} (MPC, linear
optimization), but
improves on both in a number of ways---even for the case
of linear optimization (i.e., when $H=\bzero$).  
Specifically,
\begin{itemize}
\item in contrast with~\cite{JOT-12,WNTO-2012} (and~\cite{CWH06}), 
it does not involve a (CPU-expensive) rank estimation combined 
with an increase of the size of the working set when a rank 
condition fails; 
rather, it makes use of a regularization scheme adapted 
from~\cite{WTA-2014};
\item a general condition (Condition~\ref{cond:working_set_GC}) to be satisfied by the constraint-selection rule is proposed
which, when satisfied, guarantees global and local quadratic convergence of the overall algorithm (under appropriate assumptions);
\item a specific constraint-selection rule is introduced which,
like in~\cite{CWH06} (but unlike in~\cite{JOT-12}), does not 
impose any {\sl a priori} lower bound on the size of 
the working set; this rule involves 
a modified active set identification  
scheme that builds on results from~\cite{FFK98};
numerical comparison shows that the new rule outperforms previously
used rules.
\end{itemize}
Other improvements over~\cite{JOT-12,WNTO-2012}
include (i)
a modified stopping criterion and a proof of termination of
the algorithm (in~\cite{JOT-12,WNTO-2012}, 
termination is only guaranteed under 
uniqueness of primal-dual solution and strict complementarity), 
(ii) a potentially larger value of the
``mixing parameter'' in the definition of the primal search
direction (see footnote~\ref{foonote:gamma}(iii)),
and
(iii) 
an improved update formula (compared to that used in~\cite{WTA-2014})
for the regularization parameter, which fosters a smooth evolution of 
the regularized Hessian $W$ from an initial (matrix) value $H+R$,
where $R$ is
specified by the user,
at a rate no faster than that required for local q-quadratic convergence.

In Section~\ref{sec:CR_MPC} below, we introduce the proposed algorithm
(Algorithm~\ref{algo:CR-MPC})
and a general condition (Condition~\ref{cond:working_set_GC}) to 
be satisfied by the constraint-selection rule, 
and we state global and local quadratic convergence results for 
Algorithm~\ref{algo:CR-MPC}, subject to Condition~\ref{cond:working_set_GC}.
We conclude the section by proposing a specific rule
({\rulename}), and proving that it
satisfies Condition~\ref{cond:working_set_GC}. 
Numerical results are reported in Section~\ref{sec:opt_num_results}.
While the focus of the present paper is on dense systems, 
application of the main ideas to large sparse systems is briefly discussed
in the Conclusion (Section~\ref{sec:conclusion}), which also includes other
concluding remarks.
Convergence proofs are given in two appendices.  

The following notation is used throughout the paper.  To the number $m$
of inequality constraints, we associate the index set 
$\bm:= \{1,2,\ldots,m\}$.
The primal feasible and primal strictly feasible sets are
\[
\cF_P:=\{\bx\in\bbR^n:A\bx\geq\bb\} 
\quand \cF_P^o:=\{\bx\in\bbR^n:A\bx>\bb\},
\]
and the primal and primal-dual solution sets are
\[
\cF_P^*:=
\{\bx\in\cF_P:f(\bx)\leq f(\tilde\bx),\,\forall \tilde\bx\in\cF_P\}
\quand \cF^*:=\{(\bx,\bslambda): \eqref{eq:KKT} {\rm~holds} \}\:.
\]
Of course, $\bx^*\in\cF_P^*$ if and only if, for 
some $\bslambda^*\in\bbR^m$,
$(\bx^*,\bslambda^*)\in\cF^*$  
Also,
we term {\em stationary} a point $\hat\bx\in\cF_P$ for which
there exists $\hat\bslambda\in\bbR^m$ such that $(\hat\bx,\hat\bslambda)$
satisfies~\eqref{eq:KKT} except possibly for non-negativity of
the components of $\hat\bslambda$.
Next, for $\bx\in\cF_P$, the (primal) {\em active-constraint} set 
at $\bx$ is 
\[
\cA(\bx):=\{i\in\bm:\ba_i^T\bx=b_i\}\:,
\]
where $\ba_i$ is the transpose of the $i$-th row of $A$.
Given a subset $Q\subseteq\bm$, $\Qc$ indicates its complement 
in $\bm$ and
$|Q|$ its cardinality; for a vector $\bv\in\bbR^m$, $\bv_Q$
is a sub-vector consisting of those entries with index in $Q$,
and for an $m\times n$ matrix $L$, $L_Q$ is a $|Q|\times n$
sub-matrix of $L$ 
consisting of those rows with index in $Q$.  An exception to this
rule, which should not create any confusion, is that for an 
$m\times m$ {\em diagonal} matrix $V=\diag(\bv)$, $V_Q$ is $\diag(\bv_Q)$,
a $|Q|\times|Q|$ diagonal sub-matrix of $V$.
For symmetric matrices $W$ and $H$, $W\succeq H$ (resp.~$W\succ H$)
means that $W-H$ is positive semi-definite (resp.~positive definite).
Finally, given a vector $\bv$, $[\bv]_+$ and $[\bv]_-$ denote the
positive and negative parts of $\bv$, i.e., vectors with components 
respectively given by $\max\{v_i,0\}$ and $\min\{v_i,0\}$,
$\bone$ is a vector of all ones, and given a Euclidean space $\bbR^p$,
the ball centered at $\bv^*\in\bbR^p$ with radius $\rho>0$ 
is denoted by
$B(\bv^*,\rho):=\{\bv\in \bbR^p:\|\bv-\bv^*\|\leq\rho\}$.

\section{A Regularized, Constraint-Reduced MPC Iteration}
\label{sec:CR_MPC}

\subsection{A Modified MPC Algorithm}
\label{subsec:MPC}

In~\cite{WNTO-2012}, a constraint-reduced MPC algorithm was
proposed for {\em linear} optimization problems, as a constraint-reduced
extension of a globally and locally superlinearly convergent variant 
of Mehrotra's original algorithm~\cite{Mehrotra-1992, Wright-1997}.
Transposed to the CQP context, that variant proceeds as follows.

In a first step (following the basic MPC paradigm),
given ($\bx$, $\bslambda$, $\bs$) with $\bslambda>\bzero$, $\bs>\bzero$, 
it computes the primal--dual \emph{affine-scaling} direction 
$(\Delta \bx^{\aff},\, \Delta\bslambda^{\aff},\, \Delta \bs^{\aff})$ 
at $(\bx,\bslambda,\bs)$, viz., the Newton direction for the solution 
of the equations portion of~\eqref{eq:KKT}.  Thus, it solves
\begin{equation}
J(H,A,\bs,\bslambda)
  \left[\begin{array}{c}   \Delta \bx^{\aff}\\ \Delta\bslambda^{\aff}\\  \Delta \bs^{\aff} \end{array}\right]=
  \left[\begin{array}{c}   -\nabla f(\bx)+A^T\bslambda \\ \bzero \\  -S\bslambda \end{array}\right]\:,
  \label{eq:KKT_LS_P}
\end{equation}
where, given a symmetric matrix $W\succeq\bzero$, we define
\begin{equation*}
J(W,A,\bs,\bslambda):=\left[\begin{array}{ccc}    W & -A^T & \bzero\\    A & \bzero   & -I\\    \bzero & S & \Lambda  \end{array}\right],
\end{equation*}
with $\Lambda=\diag(\bslambda)$.
Conditions for unique solvability of system~\eqref{eq:KKT_LS_P}
are given in the following standard result
(invoked in its full form later in this paper);
see, e.g.,~\cite[Lemma B.1]{JungThesis}.

\begin{lemma}
Suppose $s_i,\lambda_i\geq0$ for all $i$ and $W\succeq\bzero$.
Then $J(W,A,\bs,\bslambda)$ is invertible if and only if
the following three conditions hold:
\begin{enumerate}[label=(\roman*)]
\item $s_i+\lambda_i>0$ for all $i$;
\item $A_{\{i:s_i=0\}}$ has full row rank; and
\item $\left[W\,\,\left(A_{\{i:\lambda_i\neq0\}}\right)^T\right]$ has full row rank.
\end{enumerate}
\label{lemma:nonsingular_J}
\end{lemma}

\noindent
In particular, with $\bslambda>\bzero$ and $\bs>\bzero$,
$J(H,A,\bs,\bslambda)$ is invertible if and only if
$[H\,\,A^T]$ has full row rank.
By means of 
two steps of block Gaussian elimination, system~\eqref{eq:KKT_LS_P}
reduces to the {\em normal} system
\begin{equation}
\begin{alignedat}{2}
  M\Delta \bx^{\aff}&=-\nabla f(\bx),\\
  \Delta \bs^{\aff}&=A \Delta \bx^{\aff},\\
  \Delta\bslambda^{\aff}&=-\bslambda-S^{-1}\Lambda\Delta \bs^{\aff}\:,
\end{alignedat}
\label{eq:normal_sys}
\end{equation}
where $M$ is given by
\begin{equation}
  M:=H+A^TS^{-1}\Lambda A=H+\sum_{i=1}^m \frac{\lambda_i}{s_i}\ba_i\ba_i^T\:.
  \label{eq:normal_mat}
\end{equation}
Given positive definite $S$ and $\Lambda$,
$M$ is 
invertible whenever $J(H,A,\bs,\bslambda)$ is.

In a second step, MPC algorithms construct
a \emph{centering/corrector} direction, which in the CQP case 
(e.g.,~\cite[Section~16.6]{Nocedal-Wright-2006}) is the solution
$(\Delta \bx^{\cor},\, \Delta\bslambda^{\cor},\, \Delta \bs^{\cor})$
to (same coefficient matrix as in~\eqref{eq:KKT_LS_P})
\begin{equation}
J(H,A,\bs,\bslambda)
  \left[\begin{array}{c}   \Delta \bx^{\cor}\\ \Delta\bslambda^{\cor}\\  \Delta \bs^{\cor} \end{array}\right]=
  \left[\begin{array}{c}  \bzero\\ \bzero\\  \sigma\mu\bone-\Delta S^{\aff}\Delta\bslambda^{\aff} \end{array}\right]\:,
  \label{eq:KKT_LS_C}
\end{equation}
where 
$\mu:=\bs^T\bslambda/m$ is the ``duality measure''
and $\sigma := (1-\alpha^{\aff})^3$ is the centering parameter, with
\begin{equation*}
  \alpha^{\aff} := \text{argmax}\{\alpha\in[0,\,1]\,|\,\bs+\alpha\Delta\bs^{\aff}\geq \bzero,\,\,\bslambda+\alpha\Delta\bslambda^{\aff}\geq \bzero \}\:.
\end{equation*}
While most MPC algorithms use as search direction the sum of the 
affine-scaling and centering/corrector directions, to force global
convergence, we borrow from~\cite{WNTO-2012}%
\footnote{\label{foonote:gamma}We however do not
fully follow~\cite{WNTO-2012}: (i) Equation~\eqref{eq:mixing_para_1} 
generalizes (22) of~\cite{WNTO-2012} to CQP; (ii) In~\eqref{eq:update_lambda}
we explicitly
bound $\bslambda^+$ ($x^+$ in~\cite{WNTO-2012}), by $\lambda^{\max}$;
in the linear case, such boundedness is guaranteed (Lemma~3.3
in~\cite{WNTO-2012}); as a side-effect, in~\eqref{eq:mixing_para},
we could drop the penultimate
term in~(24) of~\cite{WNTO-2012} (invoked in proving convergence
of the $x$ sequence in the proof of Lemma~3.4
of~\cite{WNTO-2012}); (iii)
We do not restrict the primal step size as done 
in (25)
of~\cite{WNTO-2012}
(dual step size in the context of~\cite{WNTO-2012}), at the expense
of a slightly more involved convergence proof: see our 
Proposition~\ref{prop:dx_to_zero} below, to be compared 
to~\cite[Lemma~3.7]{WNTO-2012}.
}
and define
\begin{equation}
(\Delta \bx,\, \Delta\bslambda,\, \Delta \bs)=
(\Delta \bx^{\aff},\, \Delta\bslambda^{\aff},\, \Delta \bs^{\aff})
+\gamma(\Delta \bx^{\cor},\, \Delta\bslambda^{\cor},\, \Delta \bs^{\cor}),
\label{eqn:overall direction}
\end{equation}
where the ``mixing'' parameter $\gamma\in[0,\,1]$ is one 
when $\Delta\bx^{\cor}=\bzero$ and otherwise
\begin{equation}
\gamma := \min\left\{\gamma_1,\, 
\tau \frac{\|\Delta\bx^{\aff}\|}{\|\Delta\bx^{\cor}\|},\,
\tau \frac{\|\Delta\bx^{\aff}\|}{\sigma\mu}
\right\},
\label{eq:mixing_para}
\end{equation}
where $\tau\in[0,1)$ and
\begin{equation}
\gamma_1 := \text{argmax}\left\{\tilde\gamma\in[0,\,1]~|~f(\bx)-f(\bx+\Delta\bx^{\aff}+\tilde{\gamma}\Delta\bx^{\cor})\geq
\omega(f(\bx)-f(\bx+\Delta\bx^{\aff}))\right\},
\label{eq:mixing_para_1}
\end{equation}
with $\omega\in(0,\,1)$.
The first term in~\eqref{eq:mixing_para} guarantees that the search 
direction is a direction of significant descent for the objective
function (which in our context is central to forcing global convergence)
while the other two terms ensures that 
the magnitude of the centering/corrector direction is not too 
large compared to the magnitude of the affine-scaling direction.

As for the line search, we again borrow from~\cite{WNTO-2012}, where
specific safeguards are imposed to guarantee global and local q-quadratic
convergence.  We set
  \begin{equation*}
  \begin{alignedat}{2}
  \bar{\alpha}_\prim &:= \text{argmax}\{\alpha:\bs+\alpha\Delta\bs\geq \bzero\},\quad
  &&\alpha_\prim := \min\{1,\,\max\{\varkappa\bar{\alpha}_\prim,\,\bar{\alpha}_\prim-\|\Delta\bx\|\}\},\\
  \bar{\alpha}_\dual &:= \text{argmax}\{\alpha:\bslambda+\alpha\Dbl\geq \bzero\},\quad
   &&\alpha_\dual := \min\{1,\,\max\{\varkappa\bar{\alpha}_\dual,\,\bar{\alpha}_\dual-\|\Delta\bx\|\}\}\:,
  \end{alignedat}
  \end{equation*}
with $\varkappa\in(0,1)$, then
\begin{equation*}
  (\bx^{+}, \bs^{+}):=
  (\bx,  \bs) +
 \alpha_\prim (\Dbx, \Dbs)\:.
  \end{equation*}
and finally 
  \begin{equation}
\lambda^{+}_i := \min\{\lambda^{\text{max}},\,\max\{{\lambda}_i + \alpha_\dual\Delta\lambda_i,\,\min\{\underline{\lambda},\,\chi\}\}\},\,i=1,\ldots,m\:,
  \label{eq:update_lambda}
 \end{equation}
where $\lambda^{\text{max}}>0$ and $\underline{\lambda}\in(0,\lambda^{\text{max}})$ are algorithm parameters, and
\begin{equation*}
\chi := \|\Delta\bx^{\aff}\|^\nu+\|[{\bslambda}+\Delta\bslambda^{\aff}]_-\|^\nu\:,
\end{equation*}
with $\nu\geq2$.

We verified via numerical tests that for the problems considered in Section~\ref{sec:opt_num_results}, the modified MPC algorithm outlined in this section is at least as efficient as the MPC algorithm for CQPs given in \cite[Algorithm 16.4]{Nocedal-Wright-2006}.

\subsection{A Regularized Constraint-Reduced MPC Algorithm}
\label{subsec:CR_MPC}

In the modified MPC algorithm described in Section \ref{subsec:MPC}, 
the main computational cost is incurred in forming the normal 
matrix $M$ (see~\eqref{eq:normal_mat}), which requires approximately 
$mn^2/2$ multiplications (at each iteration) if $A$ is dense, 
regardless of how many of the $m$ inequality constraints 
in~\eqref{eq:primal_cqp} are active 
at the solution.  This may be wasteful when 
few of these constraints are active at the solution,
in particular (generically) when $m\gg n$ (imbalanced problems).
The constraint-reduction mechanism introduced in~\cite{TAW-06}
and used in \cite{JOT-12} in the context
of an affine-scaling algorithm for the solution of CQPs modifies $M$ 
by limiting the sum in \eqref{eq:normal_mat} to a wisely selected 
small subset of terms,
indexed by an index set 
$Q\subseteq\bm$ referred to as the {\em working set}.

Given a working set $Q$, the constraint-reduction 
technique produces an \emph{approximate} affine-scaling direction by 
solving a ``reduced'' version of the Newton system \eqref{eq:KKT_LS_P}, 
viz.
\begin{equation}
J(H,A_Q,\bs_Q,\bslambda_Q)
  \left[\begin{array}{c}   \Delta \bx^{\aff}\\ \Delta\bslambda_Q^{\aff}\\  \Delta \bs_Q^{\aff} \end{array}\right]=
  \left[\begin{array}{c}   -\nabla f(\bx)+(A_Q)^T\bslambda_Q \\ \bzero\\  -S_Q\bslambda_Q \end{array}\right]\:.
  \label{eq:CR_KKT}
\end{equation}
Just like the full system, when $\bs_Q>\bzero$, the reduced 
system~\eqref{eq:CR_KKT} is equivalent to the reduced normal system
\begin{equation}
\begin{alignedat}{2}
  \tilde M_{(Q)}\Delta \bx^{\aff} &= -\nabla f(\bx)\:,\\
  \Delta \bs_Q^{\aff} &= A_Q \Delta \bx^{\aff}\:,\\
  \Delta\bslambda^{\aff}_Q&=-\bslambda_Q-S_Q^{-1}\Lambda_Q\Delta\bs_Q^{\aff}\:,
  \end{alignedat}
\end{equation}
where the ``reduced'' $\tilde M_{(Q)}$ (still of size $n\times n$) is given by
\begin{equation*}
  \tilde M_{(Q)}:=H+(A_Q)^TS_Q^{-1}\Lambda_Q A_Q=H+\sum_{i\in Q} \frac{\lambda_i}{s_i}\ba_i\ba_i^T\:.
\end{equation*}
When $A$ is dense, the cost of forming $\tilde M_{(Q)}$ is approximately
$qn^2/2$, 
where $q:=|Q|$, 
leading to significant savings when $q\ll m$.  

One difficulty that may arise, when substituting $A_Q$ for $A$ in
the Newton-KKT matrix, is that the resulting linear system might
no longer be uniquely solvable.  Indeed, even 
when $[H\,\, A^T]$ has full row rank,
$[H\,\,(A_Q)^T]$ may be rank-deficient, so the third condition 
in Lemma~\ref{lemma:nonsingular_J} would not hold.
A possible remedy is to regularize the linear system. In the 
context of linear optimization, such regularization was implemented 
in \cite{Gill-1994} and explored in \cite{Saunders-Tomlin-1996} by 
effectively adding a fixed scalar multiple of identity matrix into
the normal matrix to improve numerical stability of the Cholesky factorization.
A more general regularization was proposed in \cite{Altman-Gondzio-1999} where
diagonal matrices that are adjusted dynamically based on the pivot values
in the Cholesky factorization were used for regularization.
On the other hand, quadratic regularization was applied to obtain 
better preconditioners in~\cite{Castro-Cuesta-2011}, where 
a hybrid scheme of the Cholesky factorization and a preconditioned 
conjugate gradient method is used to solve linear systems arising
in primal block-angular problems.
In~\cite{Castro-Cuesta-2011}, 
the regularization dies out when optimality is approached.

Applying regularization to address rank-deficiency of the normal
matrix due to constraint reduction was first considered in~\cite{WTA-2014},
in the context of linear optimization.
There a similar regularization as in \cite{Gill-1994,Saunders-Tomlin-1996}
is applied, while the scheme lets the regularization die out 
as a solution to the optimization problem is approached, to 
preserve fast local convergence.
Adapting such approach to the present context,
we replace $J(H,A_Q,\bs_Q,\bslambda_Q)$ by $J(W,A_Q,\bs_Q,\bslambda_Q)$
and $\tilde M_{(Q)}$ by
\begin{equation}
  M_{(Q)}:=W+(A_Q)^TS_Q^{-1}\Lambda_Q A_Q 
\:,
  \label{eq:CR_normat}
\end{equation}
with $W:= H+\varrho R$, where  
$\varrho\in(0,1]$ is a
regularization parameter that is updated at each iteration and 
$R\succeq \bzero$
a constant symmetric matrix such that $H+R \succ \bzero$.
Thus the inequality $W\succeq H$ is enforced, ensuring $f(\bx+\Dbxa)<f(\bx)$ 
(see Proposition~\ref{proposition:desc_mpc} below),
which in our context is critical for global convergence.
In the proposed algorithm, the modified coefficient matrix
is used in the computation of both a modified 
affine-scaling direction and a modified centering/corrector direction, which
thus solves
\begin{equation}
J(W,A_Q,\bs_Q,\bslambda_Q)
  \left[\begin{array}{c}   \Delta \bx^{\cor}\\ \Delta\bslambda_Q^{\cor}\\  \Delta \bs_Q^{\cor} \end{array}\right]=
  \left[\begin{array}{c}   \bzero\\ \bzero\\  \sigma\mu_{(Q)}\mathbf{1}-\Delta S_Q^{\aff}\Delta\bslambda_Q^{\aff} \end{array}\right]\:.
  \label{eq:CR_KKT_C}
\end{equation}
In the bottom block of the right-hand side of~\eqref{eq:CR_KKT_C}
(compared to~\eqref{eq:KKT_LS_C})
we have substituted
$\Delta S_Q^{\aff}$ and $\Delta\bslambda_Q^{\aff}$ for
$\Delta S^{\aff}$ and $\Delta\bslambda^{\aff}$, and replaced $\mu$ with
\begin{equation}
\mu_{(Q)}:=
\begin{cases}
\bs_Q^T\bslambda_Q/q\:, &\text{if } q\neq0\\
0\:, & \text{otherwise}
\end{cases}\:,
\label{eq:mu_Q}
\end{equation}
the duality measure for the ``reduced'' problem.
The corresponding normal equation system is given by 
\begin{equation}
\begin{alignedat}{2}
  M_{(Q)}\Delta \bx^{\cor}&=(A_Q)^TS_Q^{-1}(\sigma\mu_{(Q)}\mathbf{1}-\Delta S_Q^{\aff}\Delta\bslambda_Q^{\aff}),\\
  \Delta \bs_Q^{\cor}&=A_Q \Delta \bx^{\cor},\\
  \Delta\bslambda_Q^{\cor}&=S_Q^{-1}(-\Lambda_Q\Delta\bs_Q^{\cor}+\sigma\mu_{(Q)}\mathbf{1}-\Delta S_Q^{\aff}\Delta\bslambda_Q^{\aff})\:.
\end{alignedat}
\label{eq:normal_sys_CR_C}
\end{equation}
A partial search direction for the 
constraint-reduced MPC algorithm at
$(\bx,\, \bslambda,\, \bs)$ is then given by 
(see~\eqref{eqn:overall direction})
\begin{equation}
(\Delta \bx,\, \Delta\bslambda_Q,\, \Delta \bs_Q)=
(\Delta \bx^{\aff},\, \Delta\bslambda^{\aff}_Q,\, \Delta \bs^{\aff}_Q)+\gamma(\Delta \bx^{\cor},\, \Delta\bslambda^{\cor}_Q,\, \Delta \bs^{\cor}_Q),
\label{eq:search_dir_CR}
\end{equation}
where $\gamma$ is given by \eqref{eq:mixing_para}--\eqref{eq:mixing_para_1}, 
with $\mu_{(Q)}$ replacing $\mu$.%
\footnote{In the case that $q=0$ ($Q$ is empty), $\gamma$ is chosen to be zero. 
  Note that, in such case, there is no corrector direction, as the right-hand side of \eqref{eq:CR_KKT_C} vanishes.}

Algorithm~\ref{algo:CR-MPC},
including a stopping criterion, 
a simple update rule for 
$\varrho$,
and update rules (adapted from~\cite{WNTO-2012})
for the components $\lambda_i$ of the dual variable with $i\in\Qc$, 
but with the constraint-selection rule (in Step~2) left unspecified,
is formally stated below.%
\footnote{The ``modified MPC algorithm'' 
outlined in Section~\ref{subsec:MPC} is recovered as a special case 
by setting $\varrho^+=0$ and $Q=\{1,\ldots,m\}$ in Step 2 of
Algorithm~\ref{algo:CR-MPC}.}
Its core,
Iteration~\ref{algo:CR-MPC}, takes as input the current iterates 
$\bx$, $\bs>0$, $\bslambda>0$, $\tbl$, and produces
the next iterates $\bx^+$, $\bs^+>0$, $\bslambda^+>0$, $\tilde\bslambda^+$,
used as input to the next iteration. 
Here $\tbl$, with possibly $\tbl\not\geq0$, is asymptotically slightly
closer to optimality than $\bslambda$, and is used in the stopping
criterion.
While dual feasibility of $(\bx,\bslambda)$ is not enforced along the
sequence of iterates, a primal strictly feasible starting point
$\bx\in\cF_{P}^o$ is required, and primal feasibility of subsequent
iterates is enforced, as it allows for monotone descent of $f$,
which in the present context is key to global convergence.
(An extension of Algorithm~\ref{algo:CR-MPC} that allows for 
infeasible starting points is discussed in Section~\ref{subsec:infeasible}
below.)
Algorithm~\ref{algo:CR-MPC} makes use (in its stopping criterion and $\varrho$ update) 
of an ``error'' function $E\colon\bbR^n\times\bbR^m\to\bbR$ 
(also used in the constraint-selection {\rulename} in 
Section~\ref{subsec:rule1} below) given by
\begin{equation}
E(\bx,\bslambda):=\left\|(\|\bv(\bx,\bslambda)\|,\|\bw(\bx,\bslambda)\|)\right\|\:,
\label{eq:Exlambda}
\end{equation}
where 
\begin{equation}
\bv(\bx,\bslambda):=H\bx+\bc-A^T\bslambda, 
\quad w_i(\bx,\bslambda):=\min\{|s_i|,|\lambda_i|\}, ~i=1,\ldots,m,
\label{eq:v,w}
\end{equation}
with $\bs:=A\bx-\bb$, and where the norms are arbitrary.
Here $E$ measures both dual feasibility (via $\bv$) and complementary slackness (via $\bw$).
Note that, for $\bx\in\cF_P$ and $\bslambda\geq\bzero$, $E(\bx,\bslambda)=0$
if and only if $(\bx,\bslambda)$ solves (P)--(D).

\renewcommand{\themyalg}{CR-MPC}

\begin{myalg}[\small A Constraint-Reduced variant of MPC Algorithm for CQP]
\label{algo:CR-MPC}
\begin{algorithmic}[1]

 \Statex{\bf Parameters: }{$\varepsilon\geq0$, $\tau\in[0,1)$, 
$\omega\in(0,1)$, $\varkappa\in(0,1)$, $\nu\geq 2$,
$\lambda^{\max}>0$, 
$\underline{\lambda}\in(0,\lambda^{\max})$, 
and $\bar E >0$.\footnote{\label{footnote:E}For scaling reasons, it
may be advisable to set the value of $\bar E$
to the initial value of $E(\bx,\bslambda)$ (so that, in Step~2
of the initial iteration, $\varrho$ is set to 1, and $W$ to $H+R$).
This was done in the numerical tests reported in 
Section~\ref{sec:opt_num_results}.}
A symmetric $n\times n$ matrix $R\succeq\bzero$ such 
that $H+R\succ\bzero$.}
 
 \Statex{\bf Initialization: }{
$\bx\in\cF_P^o$,\footnote{Here 
it is implicitly assumed that $\cF_P^o$ is nonempty.  This assumption
is subsumed by Assumption~\ref{assum:non_empty_bdd_sol} below.}
$\bslambda>\bzero$, $\bs:= A \bx-\bb>\bzero$, 
$\tilde{\bslambda}:=\bslambda$.
}
\smallskip
\Statex{\bf Iteration CR-MPC:}
{
  \Statex{\bf Step 1.} 
Terminate if either (i) $\nabla f(\bx)=\bzero$, in which 
case $(\bx,\bzero)$ is optimal for (P)--(D), or (ii)
  \begin{equation}
\min\left\{ E(\bx,\bslambda), E(\bx,[\tilde{\bslambda}]_+)\right\}
   < \varepsilon,
   \label{eq:error_term}
  \end{equation}
in which case $(\bx,[\tilde\bslambda]_+)$ is declared 
$\varepsilon$-optimal for (P)--(D) if $E(\bx,\bslambda)\geq 
E(\bx,[\tilde{\bslambda}]_+)$, and $(\bx,\bslambda)$ is otherwise.
  \Statex{\bf Step 2.} Select a working set $Q$. Set $q:=|Q|$. 
  Set $\varrho:=\min\{1,\frac{E(\bx,\bslambda)}{\bar E}\}$.  Set $W:=H+\varrho R$.
  \Statex{\bf Step 3.} Compute approximate normal matrix $ M_{(Q)}:=W+\sum_{i\in Q} \frac{\lambda_i}{s_i}\ba_i\ba_i^T$. 
  \Statex{\bf Step 4.} Solve 
\begin{equation}
  M_{(Q)}\Delta \bx^{\aff} = -\nabla f(\bx)\:,
  \label{eq:normal_sys_CR_P_x}
\end{equation}
and set 
\begin{equation}
  \Delta \bs^{\aff} := A \Delta \bx^{\aff}\:,\quad
  \Delta\bslambda^{\aff}_Q:=-\bslambda_Q-S_Q^{-1}\Lambda_Q\Delta\bs_Q^{\aff}\:.
  \label{eq:normal_sys_CR_P}
\end{equation}

  \Statex{\bf Step 5.} Compute the affine-scaling step 
  \begin{equation}
  \alpha^{\aff} := \text{argmax}\{\alpha\in[0,\,1]\,|\,\bs+\alpha\Delta\bs^{\aff}\geq \bzero,\,\,\bslambda_Q+\alpha\Delta\bslambda^{\aff}_Q\geq \bzero \}\:.
  \label{eq:alpha_aff}
  \end{equation}
  \Statex{\bf Step 6.} Set $\mu_{(Q)}$ as in \eqref{eq:mu_Q}. Compute centering parameter $\sigma:=(1-\alpha^{\aff})^3$.
  \Statex{\bf Step 7.} Solve \eqref{eq:CR_KKT_C} for the corrector direction $(\Delta \bx^{\cor},\, \Delta\bslambda_Q^{\cor},\, \Delta \bs_Q^{\cor})$, and set $\Delta \bs^{\cor} := A \Delta \bx^{\cor}$.
  \Statex{\bf Step 8.} If $q=0$, set $\gamma:=0$. Otherwise, compute $\gamma$ 
  as in \eqref{eq:mixing_para}--\eqref{eq:mixing_para_1}, 
  with $\mu_{(Q)}$ replacing $\mu$.
Compute the search direction 
\begin{equation}
(\Delta \bx,\, \Delta\bslambda_Q,\, \Delta \bs):=
(\Delta \bx^{\aff},\, \Delta\bslambda_Q^{\aff},\, \Delta \bs^{\aff})+\gamma(\Delta \bx^{\cor},\, \Delta\bslambda_Q^{\cor},\, \Delta \bs^{\cor})\:.
\label{eq:MPC_direction_CR}
\end{equation}  
Set $\tilde{\lambda}^+_i = \lambda_i+\Delta\lambda_i$, $\forall i\in Q$, and $\tilde{\lambda}^+_i = 0$, $\forall i\in\Qc$.
 \Statex{\bf Step 9.} Compute the primal and dual steps
  $\alpha_\prim$ and $\alpha_\dual$ by 
  \begin{equation}
  \begin{alignedat}{2}
  \bar{\alpha}_\prim &:= \text{argmax}\{\alpha:\bs+\alpha\Delta\bs\geq \bzero\},\quad
   &&\alpha_\prim := \min\{1,\,\max\{\varkappa\bar{\alpha}_\prim,\,\bar{\alpha}_\prim-\|\Delta\bx\|\}\},\\
     \bar{\alpha}_\dual &:= \text{argmax}\{\alpha:\bslambda_Q+\alpha\Dbl_Q\geq \bzero\},\quad
   &&\alpha_\dual := \min\{1,\,\max\{\varkappa\bar{\alpha}_\dual,\,\bar{\alpha}_\dual-\|\Delta\bx\|\}\}\:.
  \end{alignedat}
  \label{eq:step_size}
  \end{equation}
  \Statex{\bf Step 10.} Updates:
  \begin{equation}
  (\bx^{+}, \bs^{+}):= 
  (\bx, \bs) + 
  (\alpha_\prim\Delta\bx, \alpha_\prim\Delta\bs)\:.
  \label{eq:update}
  \end{equation}
  Set 
$\chi := \|\Delta\bx^{\aff}\|^\nu
+\|[\bslambda_Q+\Delta\bslambda_Q^{\aff}]_-\|^\nu$.
Set
  \begin{equation}
\lambda^{+}_i := 
\max\{
  \min\{\lambda_i+\alpha_\dual\Delta\lambda_i,\,\lambda^{\text{max}}\},\,
  \min\{\chi,\,\underline{\lambda}\}
\},\,
\forall i\in Q\:.
  \label{eq:update_lambda_Q} 
 \end{equation}
  Set $\mu_{(Q)}^{+} := {(\bs_Q^{+})^T(\bslambda_Q^{+})}/{q}$ if $q\neq0$,
  otherwise set $\mu_{(Q)}^{+} :=0$.
  Set
  \begin{equation}
 \lambda^{+}_i := 
\max\{
  \min\{{\mu_{(Q)}^{+}}/{s_i^{+}},\,\lambda^{\text{max}}\},\,
  \min\{\chi,\,\underline{\lambda}\} 
\},\,
\forall i\in\Qc\:.
  \label{eq:update_lambda_notQ} 
  \end{equation}
  }
\end{algorithmic}
\end{myalg}

A few more comments are in order concerning Algorithm~\ref{algo:CR-MPC}.
First, the stopping criterion is a variation on that 
of~\cite{JOT-12,WNTO-2012},
involving both $\bslambda$ and $[\tilde\bslambda]_+$ instead of only 
$\bslambda$; in fact the 
latter will fail when the parameter $\lambda^{\max}$ 
(see~\eqref{eq:update_lambda_Q}--\eqref{eq:update_lambda_notQ}) is not large enough
and may fail when second order sufficient conditions are
not satisfied, while we prove below
(Theorem~\ref{thm:convergence}(iv)) that the new criterion is 
eventually satisfied indeed, in that the iterate $(\bx,\bslambda)$ converges
to a solution (even if it is not unique), be it on a mere subsequence. 
Second, our update formula for the regularization 
parameter $\varrho$ in Step~2 improves on that in~\cite{WTA-2014} 
($\varrho^+=\min\{\chi,\chi_{\max}\}$ in the notation of this paper,
where $\chi_{\max}$ is a user-defined constant) as it fosters a ``smooth''
evolution of $W$ from the initial value of $H+R$, with $R$ specified by the user,
at a rate no faster than that required for local q-quadratic convergence.
And third, $R\succeq0$ should be selected to compensate for 
possible ill-conditioning of $H$---so as to mitigate possible 
early ill-conditioning of $M_{(Q)}$.  (Note that
a nonzero $R$ may be beneficial even when $H$ is non-singular.)

It is readily established that, 
starting from a strictly feasible point,
regardless of the choice made for $Q$ in
Step~2, 
Algorithm~\ref{algo:CR-MPC}
either stops at Step 1 after finitely 
many iterations, or generates
infinite sequences $\{E_k\}_{k=0}^\infty$, $\{\bx^k\}_{k=0}^\infty$, 
$\{\bslambda^k\}_{k=0}^\infty$, 
$\{\tilde\bslambda^k\}_{k=0}^\infty$,
$\{\bs^k\}_{k=0}^\infty$, $\{\chi_k\}_{k=0}^\infty$, 
$\{Q_k\}_{k=0}^\infty$,
$\{\varrho_k\}_{k=0}^\infty$, and $\{W_k\}_{k=0}^\infty$, 
with $\bs^k=A\bx^k-\bb>\bzero$ and $\bslambda^k>\bzero$ for all $k$.
($E_0$, $\chi_0$, $\varrho_0$, and $W_0$ correspond to the
values of $E_k$, $\chi_k$, $\varrho_k$, and $W_k$ computed
in the initial iteration, while the
other initial values are provided in the ``Initialization'' step.)  Indeed, if 
the algorithm does not
terminate at Step~1, then $\nabla f(\bx)\not=\bzero$, i.e., $\Dbxa\not=\bzero$
(from~\eqref{eq:normal_sys_CR_P_x}, since $M_{(Q)}$ is invertible); it 
follows that $\Dbx\not=\bzero$ 
(if $\Dbxc\not=\bzero$, since $\tau\in[0,1)$, \eqref{eq:mixing_para} 
yields $\gamma<\|\Dbxa\|/\|\Dbxc\|$) and, since $\Dbxa\not=\bzero$ implies $\chi>0$,
\eqref{eq:step_size}, \eqref{eq:MPC_direction_CR}, \eqref{eq:update}, \eqref{eq:update_lambda_Q},
and~\eqref{eq:update_lambda_notQ} imply that $\bs^+=A\bx^+-\bb>\bzero$
and $\bslambda^+>\bzero$.
From now on, we assume that infinite sequences are generated.

\subsection{Extensions: Infeasible Starting Point, Equality Constraints}
\label{subsec:infeasible}

Because, in our constraint-reduction context, convergence is
achieved by enforcing descent of the objective
function at every iteration, infeasible starts cannot be
accommodated as, e.g., in~S.~Mehrotra's original paper \cite{Mehrotra-1992}.
The penalty-function approach proposed and analyzed in
\cite[Chapter 3]{He-Thesis}
in the context of constraint-reduced affine scaling 
for CQP (adapted from a scheme introduced in \cite{TWBUL-2003}
for a nonlinear optimization context) fits right in however.  
(Also see \cite{HT12} for the 
linear optimization case.)
Translated to the notation of the present paper, it substitutes 
for (P)--(D) the primal-dual 
pair\footnote{An $\ell_\infty$
penalty function can be substituted for this $\ell_1$ penalty function
with minor adjustments: see \cite{He-Thesis,HT12} for details.}
\begin{equation}
\minimize_{\bx\in\bbR^n,\,\bz\in\bbR^m} \: f(\bx)+\varphi\bone^T\bz
\quad \mbox{s.t.~} \:  A\bx +\bz \geq \bb\:, \bz\geq\bzero\:,
\tag{P$_\varphi$}
\label{eq:primal_cqp_rho}
\end{equation}
\begin{equation}
\maximize_{\bx\in\bbR^n,\,\bslambda\in\bbR^m,\bu\in\bbR^m}\:  
-\frac{1}{2}\bx^TH\bx+\bb^T\bslambda
\mbox{~s.t.~} \:    H\bx+\bc-A^T\bslambda=\bzero\:,
\bslambda+\bu = \varphi\bone, (\bslambda,\bu)\geq\bzero, 
\tag{D$_\varphi$}
\label{eq:dual_cqp_rho}
\end{equation}
with $\varphi>0$ a scalar penalty parameter, for which primal-strictly-feasible
points $(\bx,\bz)$ are readily available.  
Hence, given $\varphi$, this problem can be handled by Algorithm~\ref{algo:CR-MPC}.%
\footnote{
It is readily checked that, given the simple form in which $\bz$ enters 
the constraints, for dense problems, 
the cost of forming $M^{(Q)}$ still dominates and
is still approximately $|Q|n^2/2$, with still, typically, $|Q|\ll m$.}
Such $\ell_1$ penalty function is known to be exact, i.e., for some unknown,
sufficiently large (but still moderate) value of $\varphi$, solutions
$(\bx^*,\bz^*)$ to \eqref{eq:primal_cqp_rho} are such that $\bx^*$ 
solves \eqref{eq:primal_cqp};
further (\cite{He-Thesis,HT12}), $\bz^*=\bzero$.
In \cite{He-Thesis,HT12}, an adaptive scheme is proposed for increasing 
$\varphi$ to such value.
Applying this scheme on \eqref{eq:primal_cqp_rho}--\eqref{eq:dual_cqp_rho}
allows Algorithm~\ref{algo:CR-MPC} to handle infeasible starting
points for \eqref{eq:primal_cqp}--\eqref{eq:dual_cqp}.
We refer the reader to \cite[Chapter 3]{He-Thesis} for details.

Linear equality constraints of course can be handled by first projecting
the problem on the associated affine space, and then run 
Algorithm~\ref{algo:CR-MPC} on that affine space.
A weakness of this approach though is that it does not adequately
extend to the case of sparse problems (discussed in the Conclusion
section (Section 4) of this paper), as projection may destroy sparsity.
An alternative approach is, again, via augmentation:  Given the
constraints $C\bx=\bd$, with $\bd\in\bbR^p$, solve the problem
\begin{equation}
\minimize_{\bx\in\bbR^n,\,\by\in\bbR^p} \: f(\bx)+\varphi\bone^T\by
\quad \mbox{s.t.~} \:  C\bx +\by \geq \bd\:, ~C\bx - \by \leq \bd\:.
\label{eq:primal_cqp_rho_equals}
\end{equation}
(Note that, taken together, the two constraints imply $\by\geq\bzero$.)
Again, given $\varphi>0$ and using the same adaptive scheme 
from~\cite{He-Thesis,HT12}, this problem can be tackled by 
Algorithm~\ref{algo:CR-MPC}.

\subsection{A Class of Constraint-Selection Rules}
\label{subsec:working_set_guidelines}
Of course, the quality of the search directions is highly dependent 
on the choice of the working set $Q$.
Several constraint-selection rules have been proposed for 
constraint-reduced algorithms on various classes of optimization 
problems, such as linear
optimization~\cite{WNTO-2012,WTA-2014,TAW-06}, convex quadratic
optimization~\cite{JOT-12,JOT-08}, semi-definite 
optimization~\cite{Park-OLeary-2015,Park2016},
and nonlinear optimization~\cite{CWH06}.
In~\cite{WNTO-2012, TAW-06, WTA-2014}, the cardinality $q$ of $Q$ is constant
and decided at the outset.  Because in the non-degenerate case the set of
active constraints at the solution of~\eqref{eq:primal_cqp} 
with $H=\bzero$
is at least equal to the number $n$ of primal variables, $q\geq n$ is usually 
enforced in that context.  In~\cite{JOT-12}, which like this paper deals 
with quadratic problems, $q$ was allowed to vary from iteration
from iteration, but $q\geq n$ was still enforced throughout (owing to
the fact that, in the regular case, there are no more than $n$ active
constraints at the solution).  Here we propose to again allow $q$ to
vary, but in addition to not {\sl a priori} impose a positive lower 
bound on $q$.

The convergence results stated in Section~\ref{subsec:conv_ana} 
below are valid with any constraint-selection rule 
that satisfies the following condition.

\renewcommand*{\thecondition}{CSR}

\begin{condition}
Let $\{(\bx^k,\bslambda^k)\}$ be the sequence constructed by Algorithm~\ref{algo:CR-MPC} 
with the constraint-selection rule under consideration, and let $Q_k$ be the working set 
generated by the constraint-selection rule at iteration $k$.
Then the following holds: (i) if $\{(\bx^k,\bslambda^k)\}$ is bounded away from $\cF^*$,
then, for all limit points $\bx^\prime$ such that $\{\bx^k\}\to\bx^\prime$ on some infinite
index set $K$, $\cA(\bx^\prime)\subseteq Q_k$ for all large enough $k\in K$;
and (ii) 
if~\eqref{eq:primal_cqp}--\eqref{eq:dual_cqp} has a unique
solution $(\bx^*,\bslambda^*)$ and strict complementarity holds at $\bx^*$, 
and if $\{\bx^k\}\to \bx^*$, then $\cA(\bx^*)\subseteq Q_k$ 
for all $k$ large enough.  
\label{cond:working_set_GC}
\end{condition}

\smallskip\noindent
Condition~\ref{cond:working_set_GC}(i) aims at preventing convergence 
to non-optimal primal point, and hence (given a bounded sequence
of iterates) forcing convergence to solution points.  
Condition~\ref{cond:working_set_GC}(ii) is important for fast local 
convergence to set in.
A specific rule that satisfies Condition~\ref{cond:working_set_GC}, 
{\rulename},
used in our numerical experiments, is presented in Section~\ref{subsec:rule1}
below.

\subsection{Convergence Properties}
\label{subsec:conv_ana}
The following standard assumptions are used in portions of the analysis.
\noindent
\begin{assumption}
$\cF_P^o$ is nonempty and $\cF_P^*$ is 
nonempty and bounded\,%
\footnote{Nonemptiness and boundedness of $\cF_P^*$
are equivalent to dual strict feasibility (e.g.,~\cite{DrummondSvaiter99}).}
.
\label{assum:non_empty_bdd_sol}
\end{assumption}
\noindent
\begin{assumption}\footnote{\label{footnote}Equivalently
(under the sole assumption that $\cF_P^*$ is nonempty)
$A_{\cA(\bx)}$ has full row rank at all $\bx\in\cF_P$.
In fact, while we were not able to carry out the analysis 
without such (strong) assumption (the difficulty being to rule out 
convergence to non-optimal stationary points), numerical experimentation 
suggests that the assumption is immaterial.}
At every stationary point $\bx$, 
$A_{\cA(\bx)}$ has full row rank.
\label{assum:lin_ind_act}
\end{assumption}
\begin{assumption}
There exists a (unique) $\bx^*$ where the second-order sufficient condition of optimality with strict complementarity holds, with (unique) $\bslambda^*$.
\label{assum:singleton_sol+strict_complementary}
\end{assumption}
Assumption~\ref{assum:singleton_sol+strict_complementary}, mostly
used in the local analysis, 
subsumes Assumption~\ref{assum:non_empty_bdd_sol}.

Theorem~\ref{thm:convergence}, 
proved in Appendix~\ref{appendix:GC},
addresses global convergence.
\begin{theorem}
Suppose that the constraint-selection rule invoked in Step~2 of
Algorithm~\ref{algo:CR-MPC} satisfies Condition~\ref{cond:working_set_GC}.
First suppose that $\varepsilon=0$, that the iteration never stops,
and that  
Assumptions~\ref{assum:non_empty_bdd_sol}
and~\ref{assum:lin_ind_act} hold.
Then (i) the infinite sequence $\{\bx^k\}$ it constructs converges to the
primal solution set $\cF_P^*$; if in addition,
Assumption~\ref{assum:singleton_sol+strict_complementary}
holds, then
(ii)
$\{(\bx^k, \tblk)\}$
converges to the unique primal--dual solution
$(\bx^*, \bslambda^*)$ and
$\{(\bx^k, \bslambda^k)\}$ converges to $(\bx^*, \bsxi^*)$,
with $\xi^*_i:=\min\{\lambda_i,\lambda^{\max}\}$ for all $i\in\bm$,
and (iii) for sufficiently large $k$,
the working set $Q_k$ contains $\cA(\bx^*)$.

Finally, suppose again that Assumptions~\ref{assum:non_empty_bdd_sol}
and~\ref{assum:lin_ind_act} hold.
Then (iv) if $\varepsilon>0$, Algorithm~\ref{algo:CR-MPC} 
stops (in Step~1) after finitely many iterations.
\label{thm:convergence}
\end{theorem}


\noindent
Fast local convergence is addressed next; Theorem~\ref{thm:q-quad}
is proved in Appendix~\ref{appendix:LQC}. 
\begin{theorem}
\label{thm:q-quad}
Suppose that Assumption~\ref{assum:singleton_sol+strict_complementary}
holds, that $\varepsilon=0$, that the iteration never stops,
and that
$\lambda^*_i<\lambda^{\rm{max}}$ for all $i\in\bm$.
Then there exist
$\rho>0$ and $C>0$
such that, if $\|(\bx-\bx^*,{\bslambda}-{\bslambda}^*)\|<\rho$
and $Q\supseteq\cA(\bx^*)$, then
\begin{equation}
\label{eq:r-quadtildelambda}
\|(\bx^+-\bx^*,\tilde{\bslambda}^+-{\bslambda}^*)\|
\leq C \|(\bx-\bx^*,{\bslambda}-{\bslambda}^*)\|^2.
\end{equation}
\end{theorem}
When the constraint-selection rule satisfies
Condition~\ref{cond:working_set_GC}(ii), local q-quadratic convergence
is an immediate consequence of Theorems~\ref{thm:convergence} and \ref{thm:q-quad}.

\begin{corollary}
\label{cor:q-quadratic}
Suppose that Assumptions~\ref{assum:non_empty_bdd_sol}--\ref{assum:singleton_sol+strict_complementary} hold, 
that $\varepsilon=0$, that the iteration never stops, and that
$\lambda^*_i<\lambda^{\rm{max}}$ for all $i\in\bm$.
Further suppose that the constraint-selection rule invoked in Step~2
satisfies Condition~\ref{cond:working_set_GC}.
Then Algorithm~\ref{algo:CR-MPC} is locally q-quadratically convergent.  
Specifically, there exists $C>0$ such that, given any initial point $(\bx^0,\bslambda^0)$, 
for some $k^\prime>0$,
\begin{equation*}
\cA(\bx^*)\subseteq Q_k \quand
\|(\bx^{k+1}-\bx^*,{\bslambda}^{k+1}-{\bslambda}^*)\|
\leq C \|(\bx^k-\bx^*,{\bslambda}^k-{\bslambda}^*)\|^2\:, \quad\forall k>k^\prime\:.
\end{equation*}
\end{corollary}

The foregoing  theorems and corollary (essentially) extend to the case
of infeasible starting point discussed in Section~\ref{subsec:infeasible}.
The proof follows the lines of that in~\cite[Theorem~3.2]{He-Thesis}.
While Assumptions~\ref{assum:non_empty_bdd_sol} 
and~\ref{assum:singleton_sol+strict_complementary}
remain unchanged, 
Assumption~\ref{assum:lin_ind_act} must be tightened to: For
every $\bx\in\bbR^n$, $\{\ba_i : \ba_i^T\bx \leq b_i\}$ is a linearly
independent set.\footnote{\label{equality} In fact, given any known
upper bound $\overline z$ to $\{\bz^k\}$, this assumption can 
be relaxed to merely requiring linear independence of the set
$\{\ba_i : b_i-{\overline z} \leq \ba_i^T\bx \leq b_i\}$,
which tends to the set of active constraints when $\overline z$
goes to zero.  This can be done, e.g., with $\overline z=c\|\bz^0\|_\infty$,
with any $c>1$, if the constraint $\bz\leq c\bz^0$ is added to the augmented
problem.}
(While this assumption appears to be rather 
restrictive---a milder condition is used in~\cite[Theorem~3.2]{He-Thesis} 
and~\cite{HT12}, but we believe it is 
insufficient---footnote~\ref{footnote} applies here 
as well.)  

Subject to such tightening of Assumption~\ref{assum:lin_ind_act},
Theorem~\ref{thm:convergence} still holds.
Further, Theorem~\ref{thm:q-quad} and Corollary~\ref{cor:q-quadratic}
(local quadratic convergence) also hold, but for the {\em augmented} set 
of primal--dual variables, $(\bx,\bz,\bslambda,\bu)$.    While proving
the results for $(\bx,\bslambda)$ might turn out to be possible, an 
immediate consequence of q-quadratic convergence for 
$(\bx,\bz,\bslambda,\bu)$ is {\em r-quadratic}
convergence for $(\bx,\bslambda)$.

Under the same assumptions, 
Theorems~\ref{thm:convergence} and~\ref{thm:q-quad} and 
Corollary~\ref{cor:q-quadratic} still hold in the presence of
equality constraints $C\bx=\bd$ via transforming the problem
to~\eqref{eq:primal_cqp_rho_equals}, provided 
$\{\ba_i:\ba_i^T\bx\leq b_i\}\cup\{\bc_i:i=1,\ldots,p\}$ is
a linearly independent set for every $\bx\in\bbR^n$, 
with $\bc_i$ the $i$th row of $C$.
Note that it may be beneficial to choose $\bx^0$ to lie on the affine space
defined by $C\bx=\bd$, in which case 
the components of $\by^0$ can be chosen quite small, 
and to include the 
constraint $\by\leq c \by^0$ for some $c>1$ as suggested in footnote~\ref{equality}.

\subsection{A New Constraint-Selection Rule}
\label{subsec:rule1}

The proposed {\rulename}, stated below,
first computes a threshold value 
based on the amount of decrease of the error $E_k:=E(\bx^k,\bslambda^k)$, 
and then selects the working set
by including all constraints with slack values less than the 
computed threshold.

\renewcommand{\thealgorithm}{}

\begin{algorithm}[ht]
 \floatname{algorithm}{\rulename}
\begin{algorithmic}[1]

 \Statex{\bf Parameters: }{$\bar{\delta}>0$, $0<\beta<\theta<1$.}
 
 \Statex{\bf Input: }{Iteration: $k$, Slack variable: $\bs^k$, 
    Error: $E_{\min}$ (when $k>0)$, $E_k:=E(\bx^k,\bslambda^k)$, 
    Threshold: $\delta_{k-1}$.}
 \Statex{\bf Output: }{Working set: $Q_k$, 
    Threshold: $\delta_k$,
    Error: $E_{\min}$.}
 
  \If{$k=0$}
  	\State $\delta_k := \bar{\delta}$
 	\State $E_{\min} := E_k$
  \ElsIf{$E_k\leq\beta E_{\min}$}
 	\State $\delta_k := \theta \delta_{k-1}$
 	\State $E_{\min} := E_k$
  \Else
  	\State $\delta_k := \delta_{k-1}$

  	\EndIf
 \State Select $Q_k:=\{i\in\bm \,|\, s^k_i\leq\delta_k\}$.

 \caption{Proposed Constraint-Selection Rule}
\end{algorithmic}
\end{algorithm}



\noindent
A property of 
$E$ 
that plays a key role in proving that {\rulename} satisfies 
Condition~\ref{cond:working_set_GC} is stated next; it does not
require strict complementarity.  It was established in~\cite{FFK98},
within the proof of Theorem~3.12 (equation (3.13)); also 
see~\cite[Theorem~1]{Hager-Gowda-99}, \cite[Theorem~A.1]{Wright-02}, 
as well as~\cite[Lemma~2, with the ``vector of perturbations'' set
to zero]{Cartis2016} for an 
equivalent, yet global inequality in the case of linear optimization 
($H=\bzero$), under an additional dual (primal in the context of~\cite{Cartis2016}) 
feasibility assumption ($(H\bx)-A^T\bslambda+\bc=\bzero$).
A self-contained
proof in the case of CQP is given here for the sake of
completeness and ease of reference.

\begin{lemma}
\label{lem:E>c||z-z*||}
Suppose $(\bx^*,\bslambda^*)$ solves (P)--(D), let $\cI:=\{i\in\bm:\lambda_i^*>0\}$,
and suppose that (i) $A_{\cA(\bx^*)}$ and (ii) $[H~(A_\cI)^T]$ have
full row rank.
Then there exists $c>0$ and some neighborhood $V$ of the
origin 
such that
\[
E(\bx,\bslambda) \geq c \|(\bx-\bx^*,\bslambda-\bslambda^*)\|
{\rm ~~whenever~~} (\bx-\bx^*,\bslambda-\bslambda^*)\in V.
\]
\end{lemma}
\begin{proof}
Let $\bz^*:=(\bx^*,\bslambda^*)\in\bbR^{n+m}$,
let $\bs:=A\bx-\bb$, $\bs^*:=A\bx^*-\bb$, and
let $\Psi\colon\bbR^{n+m}\to\bbR$ be given by
$\Psi(\bszeta):=E(\bz^*+\bszeta)$.
We show that, restricted to an appropriate punctured convex neighborhood
of the origin, $\Psi$ is strictly positive and absolutely homogeneous,
so that the convex hull $\hat\Psi$ of such restriction generates a norm 
on $\bbR^{n+m}$,
proving the claim (with $c=1$ for the norm generated by $\hat\Psi$).  
To proceed, let 
$\bszeta^x\in\bbR^n$ and $\bszeta^\lambda\in\bbR^m$ denote
respectively the first $n$ and last $m$ components of $\bszeta$,
and let $\bszeta^s:=A\bszeta^x$.

First, since $\bv(\bz^*)=\bzero$,
$\bv(\bz^*+\bszeta) = H\bszeta^x-A^T\bszeta^\lambda$ is linear in $\bszeta$, 
making its norm absolutely homogeneous in $\bszeta$; and since 
$s_i=s_i^*+\zeta^s_i$ and $\lambda_i=\lambda_i^*+\zeta^\lambda_i$,
complementarity slackness ($s_i^*\lambda_i^*=0$) implies that
$\|\bw(\bx^*+\bszeta^x,\bslambda^*+\bszeta^\lambda)\|$ is absolutely homogeneous in 
$\bszeta$ as well, 
in some neighborhood $V_1$ of the origin.  Hence $\Psi$ is indeed absolutely 
homogeneous in $V_1$.  

Next, turning to strict positiveness and proceeding by contradiction, 
suppose that for every $\delta>0$ there exists $\bszeta\not=\bzero$, 
with $\|\bszeta\|<\delta$, such that $\Psi(\bszeta)=0$, i.e.,
$\bv(\bz^*+\bszeta)=\bzero$ and $\bw(\bz^*+\bszeta)=\bzero$.  
In view of (i), which implies uniqueness (over all of $\bbR^m$) of 
the KKT multiplier vector associated to $\bx^*$, 
and given that $\bszeta\not=\bzero$,
we must have $\bszeta^x\not=\bzero$.
In view of (ii), this implies that 
$H\bszeta^x$ and $\bszeta^s_\cI=A_\cI\bszeta^x$ cannot
vanish concurrently.  On the other hand, for $i\in \cI$ and
for small enough $\delta$, $w_i(\bz^*+\bszeta)=0$ 
implies $\zeta^s_i=0$.  Hence, $H\bszeta^x$ cannot vanish, and it
must hold that $(\bszeta^x)^T H\bszeta^x>0$.  Since
$\bv(\bz^*+\bszeta)=H\bszeta^x-A^T\bszeta^\lambda$, we conclude
from $\bv(\bz^*+\bszeta)=\bzero$
that $(\bszeta^x)^T A^T\bszeta^\lambda>0$, i.e.,
$(\bszeta^s)^T\bszeta^\lambda>0$.
Now, the argument that shows that $\zeta^s_i=0$ when $\lambda^*_i>0$
also shows that $\zeta^\lambda_i=0$ when $s^*_i>0$.  Hence our
inequality reduces to 
$\sum_{\{i:s^*_i=\lambda^*_i=0\}}\zeta^s_i\zeta^\lambda_i>0$,
in contradiction with $\bw(\bz^*+\bszeta)=\bzero$.
Taking $V$ to be a convex neighborhood of the origin contained in 
$V_1\cap \{\bszeta:\|\bszeta\|<\delta\}$ completes the proof.
\end{proof}

\begin{proposition}
Algorithm~\ref{algo:CR-MPC} with {\rulename} satisfies Condition~\ref{cond:working_set_GC}.
\label{prop:rule_cond_GC}
\end{proposition}

\begin{proof}
To prove that Condition~\ref{cond:working_set_GC}(i) holds, 
let $\{(\bx^k,\bslambda^k)\}$ be bounded away from $\cF^*$,
let $\bx^\prime$ be a limit point of $\{\bx^k\}$, and let
$K$ be an infinite subsequence such that $\{\bx^k\}\to\bx^\prime$ on $K$.
By \eqref{eq:Exlambda}--\eqref{eq:v,w}, $\{E_k\}$ is bounded away from zero
so that, under \rulename, 
there exists $\delta^\prime>0$ such that 
$\delta_k>\delta^\prime$ for all $k$.
Now, with $\bs^\prime:=A\bx^\prime-\bb$ and
$\bs^k:=A\bx^k-\bb$ for all $k$,
since $\bs^\prime_{\cA(\bx^\prime)}=\bzero$,
we have that, for all $i\in\cA(\bx^\prime)$, 
$s_i^k<\delta^\prime$ for all large enough $k\in K$.
Hence, for all large enough $k\in K$,
\begin{equation*}
s_i^k<\delta^\prime<\delta_k\:,\quad \forall i\in\cA(\bx^\prime)\:.
\end{equation*}
Since {\rulename} chooses the working set 
$Q_k:=\{i\in\bm \,|\, s^k_i\leq\delta_k\}$ for all $k$, we
conclude that $\cA(\bx^\prime)\subseteq Q_k$ for all large enough $k\in K$,
which proves Claim~(i).

Turning now to Condition~\ref{cond:working_set_GC}~(ii), 
suppose that~\eqref{eq:primal_cqp}--\eqref{eq:dual_cqp} has a unique
solution $(\bx^*,\bslambda^*)$, that strict complementarity holds
at $\bx^*$, and that $\{\bx^k\}\to\bx^*$.  If $\delta_k$ is reduced no more
than finitely many times, then of course it is bounded away from zero,
and the proof concludes as for Condition~\ref{cond:working_set_GC}(i); 
thus suppose $\{\delta_k\}\to0$.
Let $K:=\{k\geq1:\delta_k=\theta\delta_{k-1}\}$ (an infinite index set)
and, for given $k$,
let $\ell(k)$ be the cardinality of $\{k^\prime\leq k:k^\prime\in K\}$.
Then we have $\delta_k = \bar\delta\theta^{\ell(k)}$ for all $k$ and
$E_k \leq\beta^{\ell(k)}E_0$ for all $k\in K$.
Since $\beta<\theta$ (see~{\rulename}), this implies that~
$\{\frac{E_k}{\delta_k}\}_{k\in K}\to0$.
And from the definition of $E_k$
and uniqueness of the solution to~\eqref{eq:primal_cqp}--\eqref{eq:dual_cqp},
it follows that $\{\bslambda^k\}\to\bslambda^*$ as $k\to\infty$, $k\in K$.
We use these two facts to prove that, for all $i\in\cA(\bx^*)$ and some $k_0$,
\begin{equation}
s_i^k\leq\delta_k \quad\forall k\geq k_0;
\label{eq:s<delta}
\end{equation}
in view of {\rulename}, this will complete the proof of Claim (ii).
From Lemma~\ref{lem:E>c||z-z*||}, there exist
$C>0$ and $\eta>0$ such that
\[
\|(\bx-\bx^*,\bspsi-\bslambda^*)\| \leq C E(\bx,\bspsi)
\]
for all $(\bx,\bspsi)$ satisfying $\|(\bx-\bx^*,\bspsi-\bslambda^*)\|<\eta$.
Since $\{\frac{E_k}{\delta_k}\}_{k\in K}\to0$ and 
$\{\bslambda^k-\bslambda^*\}_{k\in K}\to\bzero$, and
since $\|\bs^k-\bs^*\|\leq\|A\|\|\bx^k-\bx^*\|$
(since $\bs^k-\bs^*=A(\bx^k-\bx^*)$), there exists $k_0$
such that, for all $i\in\cA(\bx^*)$,
\begin{equation}
\label{eq:s<deltaOnK}
s_i^k \leq \|A\|\|(\bx^k-\bx^*,\bslambda^k-\bslambda^*)\| 
\leq \|A\| C E_k \leq \delta_k, \quad\forall k\in K,~k\geq k_0,
\end{equation}
establishing~\eqref{eq:s<delta} for $k\in K$.  It remains to show 
that~\eqref{eq:s<delta} does hold for {\em all} $k$ large enough.
Let $\rho$ and $C$ be as in Theorem~\ref{thm:q-quad} and without loss of
generality suppose $C\rho\leq\theta$.
Since $\{(\bx^{k}-\bx^*,\bslambda^{k}-\bslambda^*)\}_{k\in K}\to\bzero$,
there exists $k\in K$ (w.l.o.g. $k\geq k_0$), such that 
$\|(\bx^{k}-\bx^*,\bslambda^{k}-\bslambda^*)\|<\rho$.
Theorem~\ref{thm:q-quad} together with~\eqref{eq:s<deltaOnK} then 
imply that
\[
\|(\bx^{k+1}-\bx^*,{\bslambda}^{k+1}-{\bslambda}^*)\|
\leq C  \|(\bx^{k}-\bx^*,{\bslambda}^{k}-{\bslambda}^*)\|^2
<\theta \|(\bx^{k}-\bx^*,{\bslambda}^{k}-{\bslambda}^*)\|
\leq\theta\delta_k/\|A\|.
\]
(When $A=\bzero$, Proposition~\ref{prop:rule_cond_GC} holds trivially.)
Hence, $\|(\bx^{k+1}-\bx^*,{\bslambda}^{k+1}-{\bslambda}^*)\|<\rho$ 
and since (in view of {\rulename}) $\delta_{k+1}$ is equal 
to either $\delta_k$ or $\theta\delta_k$ and $\theta\in(0,1)$, 
we get, for all $i\in\cA(\bx^*)$,
\[
s_i^{k+1} \leq \|A\| \|(\bx^{k+1}-\bx^*,\bslambda^{k+1}-\bslambda^*)\|
\leq \delta_{k+1},   
\]
so that $\cA(\bx^*)\subseteq Q_{k+1}$.  Theorem~\ref{thm:q-quad} can be
applied recursively, yielding $\cA(\bx^*)\subseteq Q_k$ for all $k$
large enough, concluding the proof of Claim (ii).
\end{proof}

Note that if a constraint-selection rule satisfies
Condition~\ref{cond:working_set_GC}, rules derived from it
by replacing $Q_k$ by a superset of it also satisfy 
Condition~\ref{cond:working_set_GC}
so that our convergence results still hold.  Such augmentation
of $Q_k$ is often helpful; e.g., see Section~5.3 in~\cite{WNTO-2012}.
Note however that the
following corollary to Theorems~\ref{thm:convergence}--\ref{thm:q-quad}
and Proposition~\ref{prop:rule_cond_GC}, proved in 
Appendix~\ref{appendix:GC}, of course does {\em not} apply 
when {\rulename} is thus augmented.

\begin{corollary}
Suppose that {\rulename} is used in Step~2 of Algorithm~\ref{algo:CR-MPC},
$\varepsilon=0$,
and that Assumptions~\ref{assum:non_empty_bdd_sol}--\ref{assum:singleton_sol+strict_complementary} hold.
Let $(\bx^*, \bslambda^*)$ be the unique primal--dual solution.
Further suppose that  $\lambda_i^*<\lambda^{\max}$ for all $i\in\bm$.
Then, for sufficiently large $k$, {\rulename} gives $Q_k=\cA(\bx^*)$.%
\footnote{In particular, if $\bx^*$ is an unconstrained minimizer,
the working set $Q$ is eventually empty, and Algorithm~\ref{algo:CR-MPC}
reverts to a simple regularized Newton method (and terminates in
one additional iteration if $H\succ\bzero$ and $R=\bzero$).}
\label{cor:exact_working_set}
\end{corollary}

\section{Numerical Experiments}
\label{sec:opt_num_results}

We report computational results obtained with
Algorithm~\ref{algo:CR-MPC} on randomly generated problems and 
on data-fitting problems of various sizes.%
\footnote{In addition, a preliminary version of the proposed 
algorithm (with a modified version of {\ruleJOT}, see \cite{Laiu-Thesis} for details) 
was successfully tested in \cite{LHMOT-2015} on CQPs arising from a 
positivity-preserving numerical scheme for solving linear kinetic transport equations.}
Comparisons are made across different constraint-selection 
rules, including the unreduced case ($Q=\bm$).%
\footnote{We also ran 
comparison tests with the constraint-reduced
algorithm of~\cite{Park2016}, for which polynomial complexity 
was established (as was superlinear convergence)
for general semi-definite optimization problems.  
As was expected, that algorithm could not
compete (orders of magnitude slower) with algorithms specifically
targeting CQP.}

\subsection{Other Constraint-Selection Rules}
As noted, the convergence properties of Algorithm~\ref{algo:CR-MPC}
that are given in Section~\ref{subsec:conv_ana} hold with any working-set
selection rule that satisfies Condition~\ref{cond:working_set_GC}.
The rules used in the numerical tests are our {\rulename},
{\ruleJOT} of~\cite{JOT-12}, {\ruleFFK} of~{\cite{FFK98,CWH06},
and {\ruleALL} ($Q=\bm$, i.e., no reduction). 
The details of {\ruleJOT} and {\ruleFFK} are stated below.\\
\renewcommand{\thealgorithm}{}
\begin{minipage}[t]{1\textwidth}
  \vspace{-0.6cm}  
\begin{algorithm}[H]
 \floatname{algorithm}{\ruleJOT}
\begin{algorithmic}[1]
 \Statex{\bf Parameters: }{$\kappa>0$, $q_U\in[n, m]$ (integer).}
 \Statex{\bf Input: }{Iteration: $k$, Slack variable: $\bs^k$,  
Duality measure: $\mu:=(\bslambda^k)^T \bs^k/m$.
}
 \Statex{\bf Output: }{Working set: $Q_k$.}
 \Statex{Set $q:=\min\{\max\{n, \lceil{\mu^\kappa m}\rceil\},q_U\}$, and let $\eta$ 
be the $q$-th smallest slack value.}
 \Statex Select $Q_k:=\{i\in\bm \,|\, s^k_i\leq \eta\}$.
 \caption{ }
\end{algorithmic}
\end{algorithm}
\end{minipage}\\~~
\renewcommand{\thealgorithm}{}
\begin{minipage}[t]{1\textwidth}
  \vspace{0pt} 
\begin{algorithm}[H]
 \floatname{algorithm}{\ruleFFK}
\begin{algorithmic}[1]
 \Statex{\bf Parameter: }{$0<r<1$.}
 \Statex{\bf Input: }{Iteration: $k$, Slack variable: $\bs^k$,  Error: $E_k:=E(\bx^k,\bslambda^k)$ (see \eqref{eq:Exlambda}).}
 \Statex{\bf Output: }{Working set: $Q_k$.}
 
 \Statex Select $Q_k:=\{i\in\bm \,|\, s^k_i\leq (E_k)^r\}$.
 \caption{ }
\end{algorithmic}
\end{algorithm}
\end{minipage}\\~~~~

\noindent\smallskip
Note that the thresholds in {\rulename} and {\ruleFFK} depend on both the duality measure $\mu$ and dual feasibility (see \eqref{eq:Exlambda}) and these two rules impose no restriction on $|Q|$.
On the other hand, the threshold in {\ruleJOT} involves only $\mu$, while it is required that $|Q|\geq n$ .
In addition, it is readily verified that {\ruleFFK} satisfies 
Condition~\ref{cond:working_set_GC}, and that so does {\ruleJOT} under 
Assumption~\ref{assum:lin_ind_act}.

It is worth noting that {\rulename}, {\ruleFFK}, and {\ruleJOT} all select 
constraints by comparing the values of primal slack variables $s_i$ to 
some threshold values (independent of $i$), while the associated dual 
variables $\lambda_i$ 
are not taken into account individually.  Of course, variations with
respect to such choice are possible.
In fact, it was shown in \cite{FFK98} (also see an implementation 
in~\cite{CWH06}) that the 
{\em strongly} active constraint set $\{i\in\bm\colon \lambda_i^* > 0\}$ can be 
(asymptotically) identified at iteration~$k$ by the set $\{i\in\bm\colon \lambda_i^k \geq \delta_k\}$ with a properly chosen threshold~$\delta_k$.
Modifying the constraint selection rules considered in the present
paper to include such information 
might improve the efficiency of the rules, especially when 
the constraints are poorly scaled. 
(Such modification does not affect the convergence properties of 
Algorithm~\ref{algo:CR-MPC} as long as the modified rules still satisfy 
Condition~\ref{cond:working_set_GC}.)  
Numerical tests were carried out with an ``augmented''
{\rulename} that also includes $\{i\in\bm\colon \lambda_i^k \geq \delta_k\}$
(with the same $\delta_k$ as in the original {\rulename}).
The results suggest that,
on the class of imbalanced problems considered in this section, while 
introducing some overhead, such augmentation (with the same $\delta_k$)
brings no benefit.

\subsection{Implementation Details}
\label{subsec:implementation}
All numerical tests were run with a Matlab implementation of 
Algorithm~\ref{algo:CR-MPC} on a machine with Intel(R) Core(TM) i5-4200 CPU(3.1GHz), 4GB RAM, Windows 7 Enterprise, and Matlab 7.12.0(R2011a).
In the implementation, $E(\bx,\bslambda)$ (see 
\eqref{eq:Exlambda}--\eqref{eq:v,w}) is normalized via division by 
the factor of $\max\{\|{A}\|_\infty,\,\|H\|_\infty,\,\|\bc\|_\infty\}$,
and 2-norms are used in~\eqref{eq:Exlambda} 
and~Steps~9 and~10.
In addition, for scaling purposes 
(see, for example, \cite{JOT-12}), we used the normalized 
constraints $(DA)\bx\geq D\bb$, where $D=\diag{(1/\|\ba_i\|_2)}$.

To highlight the significance of constraint-selection rules, a dense direct 
Cholesky solver was used to solve normal equations \eqref{eq:normal_sys_CR_P_x} and \eqref{eq:normal_sys_CR_C}.
Following \cite{TAW-06} and \cite{JOT-12}, we 
set $s_i:=\max\{s_i,10^{-14}\}$ for all $i$ when computing 
$M_{(Q)}$ in \eqref{eq:CR_normat}.
Such safeguard prevents $M_{(Q)}$ from being too ill-conditioned and 
mitigates numerical difficulties in 
solving \eqref{eq:normal_sys_CR_P_x} and \eqref{eq:normal_sys_CR_C}.
When the Cholesky factorization 
of the modified $M_{(Q)}$ failed, we then doubled 
the regularization parameter $\varrho$ and recomputed $M_{(Q)}$ 
in~\eqref{eq:CR_normat}, and repeated this process
until $M_{(Q)}$ was successfully 
factored.%
\footnote{An alternative approach to take care of ill-conditioned 
$M_{(Q)}$ is to apply a variant of the Cholesky factorization that handles 
positive semi-definite matrices, such as the Cholesky-infinity factorization 
(i.e., \texttt{cholinc(X,{\textquotesingle}inf{\textquotesingle})} in Matlab)
or the diagonal pivoting strategy discussed in \cite[Chapter 11]{Wright-1997}.
Either implementation does not make notable difference in the numerical 
results reported in this paper, since the Cholesky factorization fails 
in fewer than $1\%$ of the tested problems.}

In the implementation, Algorithm~\ref{algo:CR-MPC} is set to terminate as
soon as either the stopping criterion \eqref{eq:error_term} is satisfied or 
the iteration count reaches $200$.
The algorithm parameter values used in the tests were
 $\varepsilon=10^{-8}$, $\tau=0.5$, $\omega=0.9$, $\varkappa=0.98$, 
$\nu=3$, $\lambda^{\max} =10^{30}$, $\underline{\lambda}=10^{-6}$, 
$R=I_{n\times n}$ (the $n\times n$ identity matrix), and 
$\bar{E}=E(\bx^0,\bslambda^0)$, 
as suggested in footnote~\ref{footnote:E}.
The parameters in {\rulename} were given values $\beta=0.4$, 
$\theta=0.5$, and $\bar{\delta}=$ 
the $2n$-th smallest initial slack value.
In {\ruleJOT}, $\kappa=0.25$ is used as in \cite{JOT-12},
and $q_U=m$ was selected (although $q_U=3n$ is suggested as a ``good
heuristic'' in~\cite{JOT-12}) to protect against a possible very
large number of active constraints at the solution;
the numerical results in~\cite{JOT-12} suggest 
that there is no significant downside in using $q_U=m$.
In {\ruleFFK}, $r=0.5$ is used as in \cite{FFK98,CWH06}.

\subsection{Randomly Generated Problems}
\label{subsec:random_problems}

We first applied Algorithm~\ref{algo:CR-MPC} on imbalanced ($m\gg n$)
randomly generated problems;
we used $m:=10\,000$ and $n$ ranging from $10$ to $500$.
Problems of the form~\eqref{eq:primal_cqp}
were generated in a similar way as 
those used in \cite{TAW-06,WNTO-2012,JOT-12}.
The entries of $A$ and $\bc$ were taken from a standard normal 
distribution $\cN(0,1)$,
those of $\bx^0$ and $\bs^0$ from uniform distributions $\cU(0,1)$ 
and $\cU(1,2)$, respectively, and we set $\bb:=A\bx^0-\bs^0$, which 
guarantees that 
$\bx^0$ is strictly feasible.
We considered two sub-classes of problems: (i) strongly convex, with
$H$ diagonal and positive, with random diagonal entries from $\cU(0,1)$,
and (ii) linear, with $H=\bzero$.
We solved $50$ randomly generated problems for each 
sub-class of $H$
and for each problem size, and report the results averaged over the $50$ 
problems. 
There was no instance of failure on these problems.
Figure~\ref{fig:Random} shows the results.  (We also ran tests with
$H$ rank-deficient but nonzero, with similar results.)

%
\begin{figure}[ht]
\centering
\captionsetup[subfigure]{labelformat=parens}

\subfloat[Strongly convex QP: $H\succ\bzero$]{
\captionsetup[subfigure]{labelformat=empty, position=top}
\subfloat[Iteration count]{\includegraphics[width=0.33\linewidth]{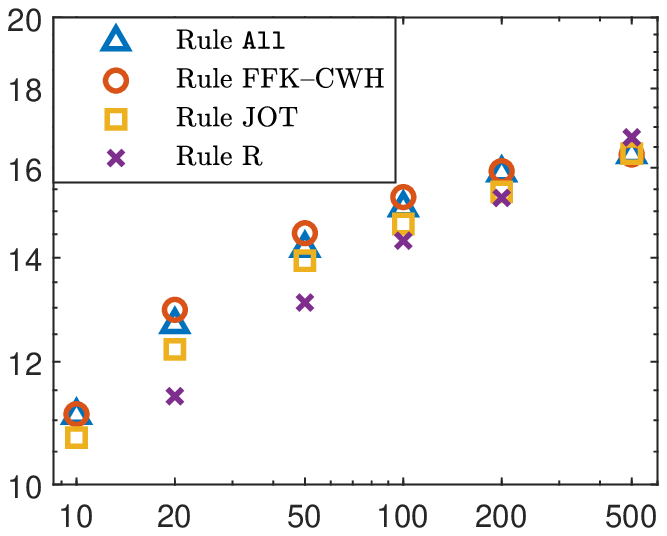}}~~~
\subfloat[Average $|Q|$]{\includegraphics[width=0.33\linewidth]{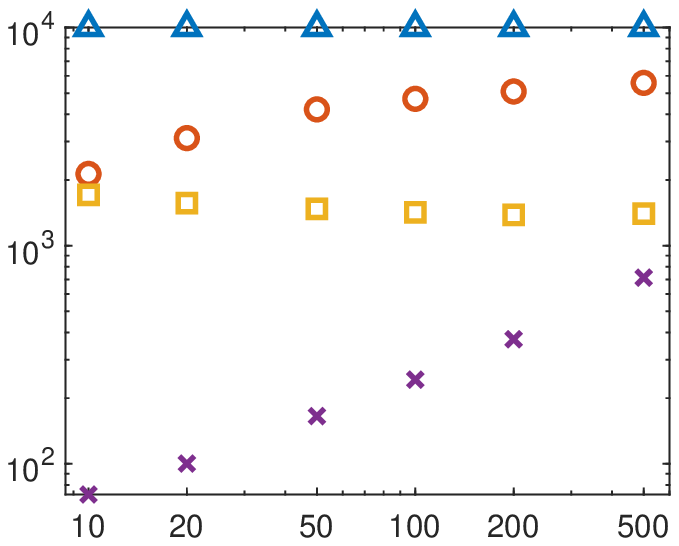}}~
\subfloat[Computation time (sec)]{\includegraphics[width=0.33\linewidth]{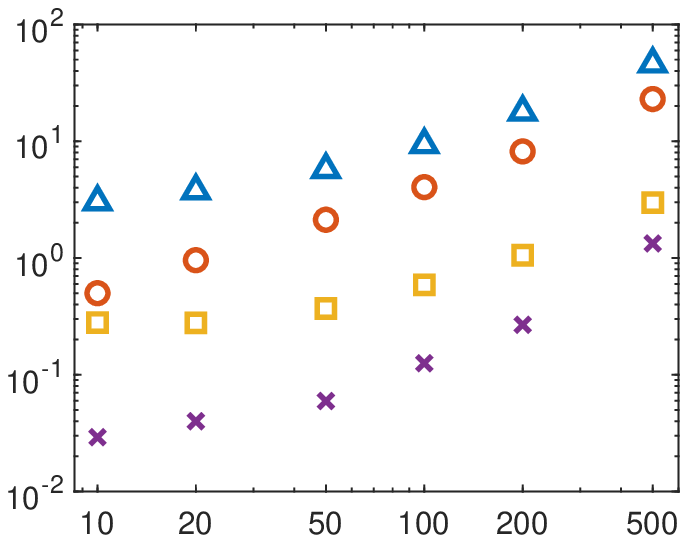}}
\setcounter{subfigure}{1}
\label{fig:PD}}


\subfloat[linear optimization: $H=\bzero$]{
\captionsetup[subfigure]{labelformat=empty, position=top}
\subfloat[Iteration count]{\includegraphics[width=0.33\linewidth]{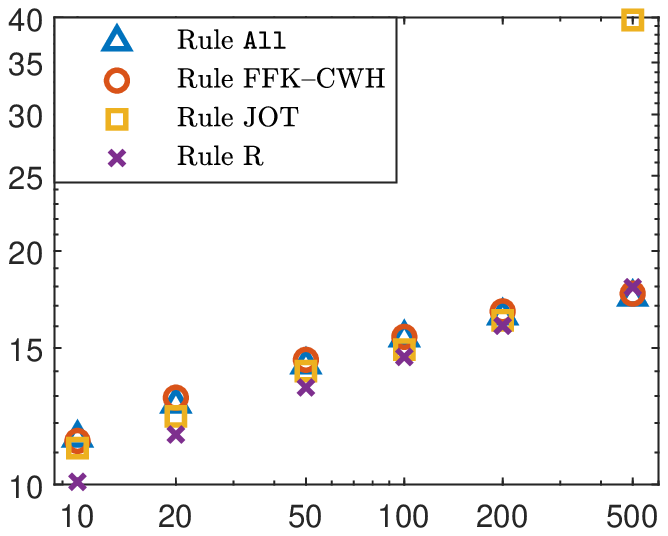}}~~~
\subfloat[Average $|Q|$]{\includegraphics[width=0.33\linewidth]{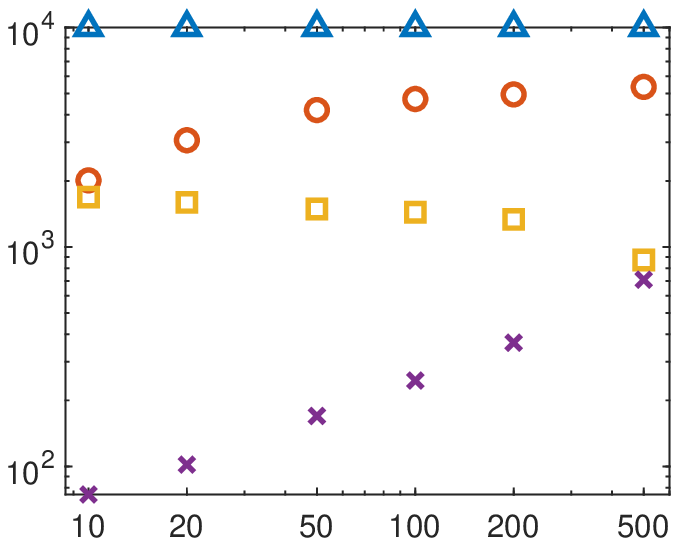}}~
\subfloat[Computation time (sec)]{\includegraphics[width=0.33\linewidth]{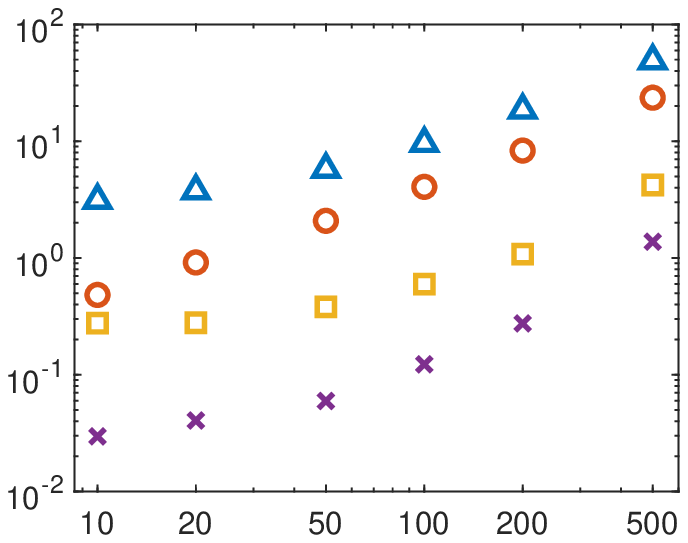}}
\setcounter{subfigure}{2}
\label{fig:LP}}
\caption{\small Randomly generated problems 
with $m=10\,000$ constraints-- Numerical results on 
two types of randomly generated problems. In each figure, the $x$-axis is the 
number of variables ($n$) and the $y$-axis is iteration count, average 
size of working set, or computation time,
all averaged over the 50 problem instances and plotted in logarithmic scale.
The results of {\ruleALL}, {\ruleFFK}, {\ruleJOT}, and {\rulename} are plotted 
as blue triangles, red circles, yellow squares, and purple crosses, respectively.}
\label{fig:Random}
\end{figure}

It is clear from the plots that, in terms of computation time, 
{\rulename} outperforms other constraint-selection rules for 
the randomly generated problems we tested.\footnote{Interestingly,
on strongly convex problems, most rules (and especially {\rulename})
need a smaller number of iterations than {\ruleALL} (except for
$n=500$)!}
When the number of variables ($n$) is $1\%\sim 5\%$ of the 
number of constraints ($m$) (i.e, $n$=100 to 500), 
Algorithm~\ref{algo:CR-MPC} with {\rulename} 
is two to five times faster than with the second best rule,
or 20 to 50 times faster than the unreduced 
algorithm ({\ruleALL}). 
When $n$ is lowered to less than $1\%$ of $m$ (i.e., $n<100$), 
the time advantage of using {\rulename} further doubles.
Note that, as $n$ decreases, {\rulename} is 
asymptotically more restrictive than {\ruleFFK} and {\ruleJOT}.
We believe this may be the key reason that {\rulename} outperforms
other rules, especially on problems with small $n$.
 
\subsection{Data-Fitting Problems}
\label{subsec:data_fitting}
We also applied Algorithm~\ref{algo:CR-MPC} on CQPs arising from two
instances of a data-fitting problem:
trigonometric curve fitting to noisy observed data points. 
This problem was formulated in \cite{TAW-06} as a linear
optimization problem, 
and then in~\cite{JOT-12} reformulated as a CQP by imposing a 
regularization term.
The CQP formulation of this problem, taken from~\cite{JOT-12}, is as follows.
Let $g:[0,1]\to\bbR$ be a given function of time, 
and let $\bar{\bb}:=[\bar{b}_1,\dots,\bar{b}_{\bar{m}}]^T
\in\bbR^{\bar m}$ 
be a vector that collects noisy observations of $g$ at 
sample time $t=t_1,\dots,t_{\bar{m}}$.
The problem aims at finding a trigonometric expansion $u(t)$ from 
the noisy data $\bar{\bb}$ that best approximates $g$.
Here $u(t):=\sum_{j=1}^{\bar{n}} \bar{x}_j\psi_j(t)$, with the trigonometric basis
$$
\psi_j(t):=
\begin{cases}
\cos(2(j-1)\pi t)\:, & j=1,\dots,\lceil\frac{\bar{n}}{2}\rceil\\
\sin(2(j-\lceil\frac{\bar{n}}{2}\rceil)\pi t)& j=\lceil\frac{\bar{n}}{2}\rceil+1,\dots,\bar{n}
\end{cases}\:.
$$
Equivalently, 
$u(t)=\bar{A}\bar{\bx}$, where $\bar{\bx}:=[\bar{x}_1,\dots,\bar{x}_{\bar{n}}]^T$ 
and $\bar{A}$ is a $\bar{m}\times\bar{n}$ matrix with entries $\bar{a}_{ij}=\psi_j(t_i)$.
Based on a regularized minimax approach, 
the problem is then formulated as
$$
\minimize_{\bar{\bx}\in\bbR^{\bar{n}}} \|\bar{A}\bar{\bx}-\bar{\bb}\|_{\infty} + \frac{1}{2}\bar\alpha\bar{\bx}^T \bar{H} \bar{\bx}\:,
$$
where $\bar{H}\succeq\bzero$ is a symmetric $\bar{n}\times\bar{n}$ matrix, $\bar{\alpha}$ is a regularization parameter, 
and $\bar{\bx}^T \bar{H} \bar{\bx}$ is a regularization term that 
helps resist
over-fitting. 
This problem can be rewritten as
$$
\begin{alignedat}{2}
   \minimize_{\bar{\bx}\in\bbR^{\bar{n}} v\in \bbR} \:& v+\frac{1}{2}\bar\alpha\bar{\bx}^T \bar{H} \bar{\bx}\\
   \mbox{subject to} \:    & \bar{A}\bar{\bx}-\bar{\bb}\geq -v\bone\:,\\
   						  -& \bar{A}\bar{\bx}+\bar{\bb}\geq -v\bone\:,
\end{alignedat}
$$
which is a CQP in the form of \eqref{eq:primal_cqp} with number of variables $n=\bar{n}+1$ and number of constraints $m=2\bar{m}$.

Following~\cite{JOT-12}, 
we tested Algorithm~\ref{algo:CR-MPC} on this problem with two target functions
$$
g(t) = \sin(10t)\cos(25t^2) \quand g(t) = \sin(5t^3)\cos^2(10t)\:.
$$
In each case, as in~\cite{JOT-12}, we sampled the data uniformly 
in time and set $\bar{b}_i:=g(\frac{i-1}{\bar{m}})+\epsilon_i $, 
where $\epsilon_i$ is an independent and identically distributed 
noise that takes values from $\cN(0,0.09)$,%
\footnote{We also ran the tests without noise and with noise of 
variance between $0$ and $1$, and the results were very similar 
to the ones reported here.}
for $i=1,\dots,\bar{m}$ and, as in~\cite{JOT-12},
the regularization parameters were chosen as $\bar\alpha:=10^{-6}$ and $\bar{H}=\diag(\bar{\bh})$, with $\bar{h}_1:=0$ and $\bar{h}_j=\bar{h}_{j+\lceil\frac{\bar{n}}{2}\rceil-1}:=2(j-1)\pi$, for $j=2,\dots,\lceil\frac{\bar{n}}{2}\rceil$, and $\bar{h}_{\bar{n}}:=2(\lfloor\frac{\bar{n}}{2}\rfloor)\pi$.

Figure~\ref{fig:Data} reports our numerical results. 
Since these problems involve noise, we solved the problem $50$ times for each target function and report the average results.
(The average is not reported---the corresponding symbol is
not plotted---for problems on which one or more of the 50 runs
failed, i.e., did not converge when iteration count reaches 200.)
The sizes of the tested problems are $m:=10\,000$ and $n$ 
ranging from 10 to 500.

\begin{figure}[ht]
\centering
\captionsetup[subfigure]{labelformat=parens}

\subfloat[$g(t) = \sin(10t)\cos(25t^2)$]{
\captionsetup[subfigure]{labelformat=empty, position=top}
\subfloat[Iteration count]{\includegraphics[width=0.33\linewidth]{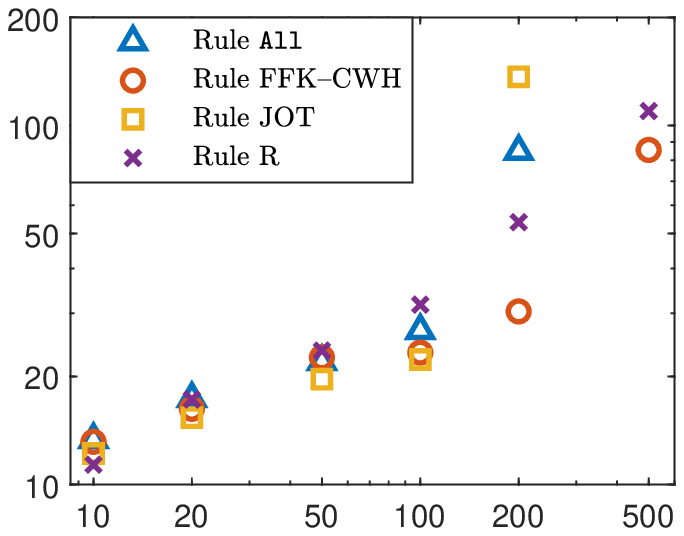}}~~~
\subfloat[Average $|Q|$]{\includegraphics[width=0.33\linewidth]{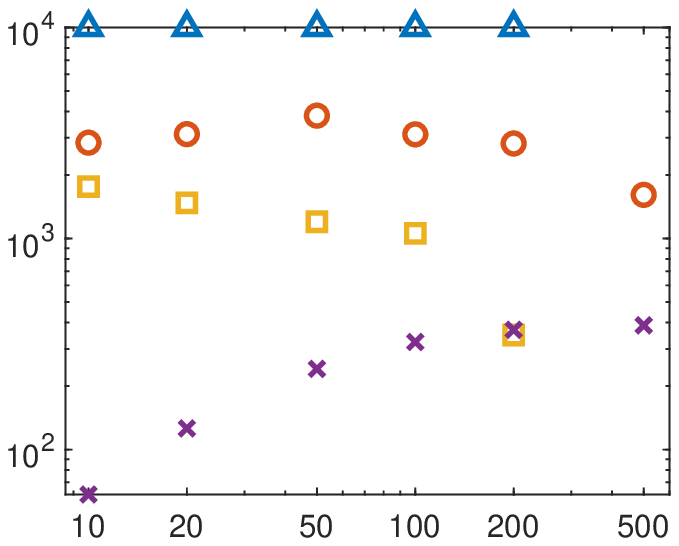}}~
\subfloat[Computation time (sec)]{\includegraphics[width=0.33\linewidth]{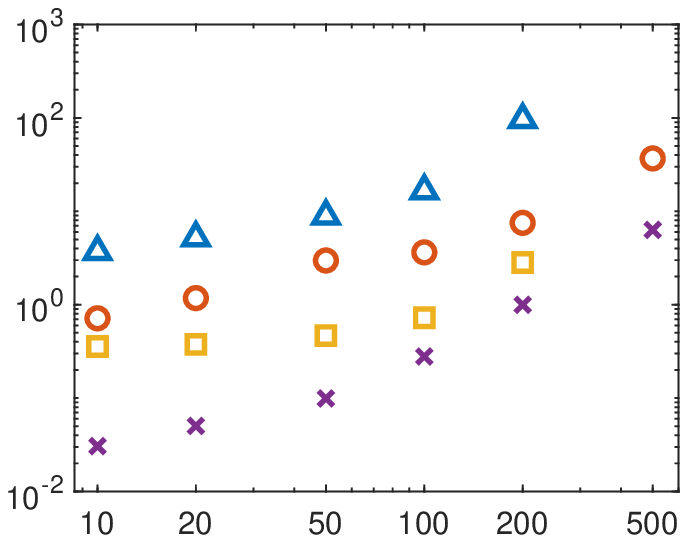}}
\setcounter{subfigure}{1}
\label{fig:Data1}}

\subfloat[$g(t) = \sin(5t^3)\cos^2(10t)$]{
\captionsetup[subfigure]{labelformat=empty, position=top}
\subfloat[Iteration count]{\includegraphics[width=0.33\linewidth]{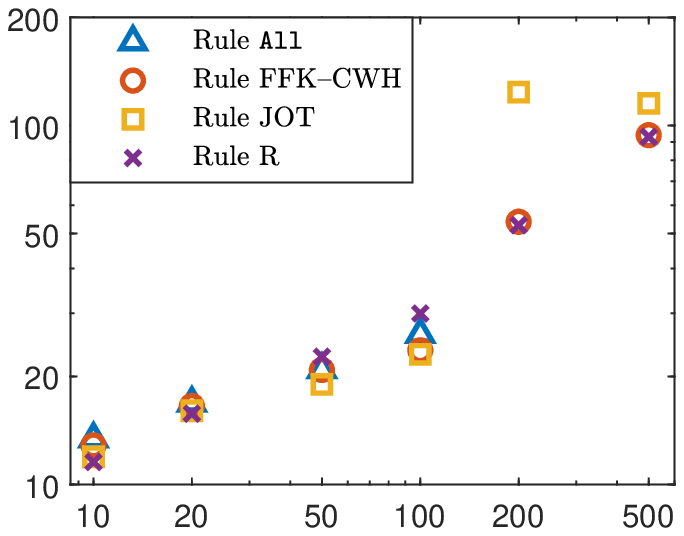}}~~~
\subfloat[Average $|Q|$]{\includegraphics[width=0.33\linewidth]{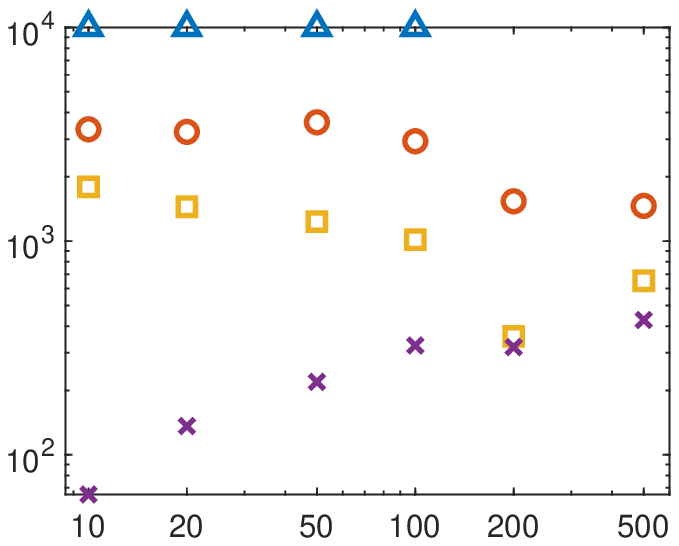}}~
\subfloat[Computation time (sec)]{\includegraphics[width=0.33\linewidth]{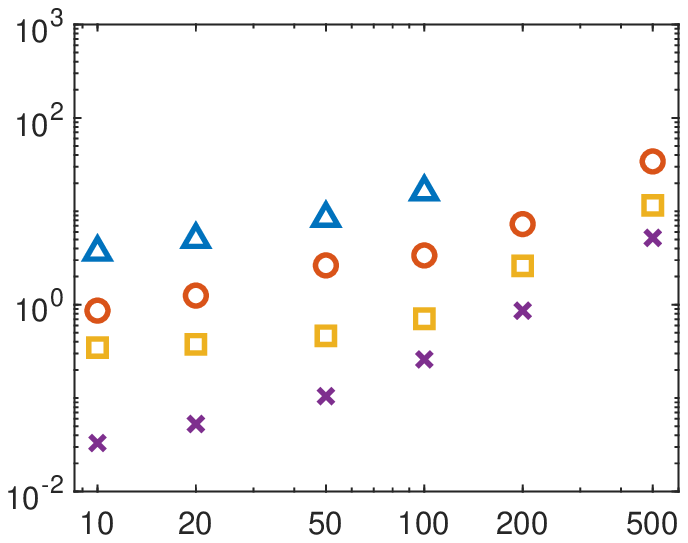}}
\setcounter{subfigure}{2}
\label{fig:Data2}}

\caption{\small Data-fitting problems 
with $m=10\,000$ constraints -- Numerical results on 
two data-fitting problems. In each figure, the $x$-axis is the number of variables ($n$) and the $y$-axis is iteration count, average size of working set, or computation time, all averaged over the 50 problem instances and plotted in logarithmic scale.
The results of {\ruleALL}, {\ruleFFK}, {\ruleJOT}, and {\rulename} are plotted 
as blue triangles, red circles, yellow squares, and purple crosses, respectively.}
\label{fig:Data}
\end{figure}

The results show that {\rulename} still outperforms other constraint-selection rules in terms of computation time, especially on problems with relatively small $n$.
In general, {\rulename} is two to ten times faster than the second best rule.
We observe in Figure~\ref{fig:Data} that {\ruleJOT} and {\ruleALL} failed to converge within $200$ iterations in a few instances on problems with relatively large $n$.
Numerical results suggest that these failures are due to ill-conditioning of $M_{Q}$ 
apparently producing poor search directions. 
Thus, we conjecture that accurate identification of active constraints not only reduces computation time, but also alleviates the ill-conditioning issue 
of $M_{Q}$ near optimal points.

\subsection{Comparison with Broadly Used, Mature Code%
\protect\footnote{
It may also be worth pointing out that a short decade ago, 
in~\cite{Winternitz-Thesis}, the performance of an early version of 
a constraint-reduced MPC algorithm (with a more elementary 
constraint-selection rule than {\ruleJOT}) was compared, 
on imbalanced filter-design applications (linear optimization), to 
the ``revised primal simplex with partial pricing'' algorithm discussed
in~\cite{BertsimasTsitsiklis97}, with encouraging results: on the 
tested problems, the constraint-reduced code proved competitive
with the simplex code on some such problems and superior on
others.}
}

\label{subsec:cvx comparison}

With a view toward calibrating the performance of Algorithm~\ref{algo:CR-MPC}
reported in Sections~\ref{subsec:random_problems} and \ref{subsec:data_fitting}, 
we carried out a numerical comparison with
two widely used solvers, SDPT3~\cite{Toh-Todd-Tutuncu-1999, Tutuncu-Toh-Todd-2003} 
and SeDuMi~\cite{Sturm-1999}.%
\footnote{While these two solvers have a broader scope (second-order cone optimization,
semidefinite optimization) than Algorithm~\ref{algo:CR-MPC}, they
allow a close comparison with our code, as Matlab implementations
are freely available within the CVX Matlab package~\cite{cvx,Grant-Boyd-2008}.}
The tests were performed
on the problems considered in Sections~\ref{subsec:random_problems} and \ref{subsec:data_fitting} with sizes $m=10\,000$ and $n=10,\,20,\,50,\,100,\,200,\,500$.
For~\ref{algo:CR-MPC}, the exact same implementation, including starting points 
and stopping criterion, as outlined in 
Section~\ref{subsec:implementation} was used.
As for the SDPT3 and SeDuMi solvers, we set the solver precision to $10^{-8}$ 
and let the solvers decide the starting points.

Table~\ref{table:cvx_comparison} reports the iteration counts and computation time of SDPT3, SeDuMi, and Algorithm~\ref{algo:CR-MPC} with {\ruleALL} and {\rulename}, on each type of tested problems.
The numbers in Table~\ref{table:cvx_comparison} are average values 
over 50 runs and over all six tested values of $n$.
These results show that, for such significantly 
imbalanced problems, 
constraint reduction brings a clear edge.
In particular, for such problems, Algorithm~\ref{algo:CR-MPC} 
with {\rulename} shows a 
significantly better time-performance than two mature solvers.

\begin{table}[h]
\begin{center}
\begin{tabular}{lrrrrrrrr}
 & \multicolumn{4}{c}{Randomly generated problems} 
 & \multicolumn{4}{c}{Data fitting problems}\\
     \cmidrule(r){2-5} \cmidrule(r){6-9} 
Algorithm & \multicolumn{2}{c}{  $H\succ \bzero$  } 
 & \multicolumn{2}{c}{  $H= \bzero$  } 
 & \multicolumn{2}{c}{{\footnotesize$\sin(10t)\cos(25t^2)$}}
 & \multicolumn{2}{c}{{\footnotesize$\sin(5t^3)\cos^2(10t)$}}\\
    \cmidrule(r){2-3}    \cmidrule(r){4-5} \cmidrule(r){6-7} \cmidrule(r){8-9} 
 & iteration & time & iteration & time & iteration & time & iteration & time\\
 SDPT3 & 23.6 & 35.8 & 21.2 & 22.3	 & 26.2 & 46.1 & 	26.7 & 	48.7 	\\
 SeDuMi & 22.0	& 4.0 & 16.3 & 4.4	& 26.9 & 5.1	& 26.1	& 5.5 \\
 {\ruleALL} & 14.1 & 16.5 & 14.7 &	18.7 & 48.8 & 94.3 & 54.3 & 	119.4 	\\
{\rulename} & 13.2	& 0.3 & 14.3 & 0.4 & 38.7 & 1.4 & 43.8	& 1.6 	
\end{tabular}
\end{center}
\caption{\small 
Comparison of Algorithm~\ref{algo:CR-MPC} with 
popular codes -- This table reports the iteration count and computation time (sec) for each of the compared algorithms on each type of tested problems, averaged over 50 runs. Every reported number is also averaged over various problems sizes: $m=10\,000$ and $n=10,\,20,\,50,\,100,\,200,\,500$.}
\label{table:cvx_comparison}
\end{table}

\section{Conclusion}
\label{sec:conclusion}

Convergence properties of the constraint-reduced algorithm proposed 
in this paper, which includes a number of novelties, were proved
independently of the choice of the working-set selection rule, provided
the rule satisfies Condition~\ref{cond:working_set_GC}.  
Under a specific such rule, based
on a modified active-set identification scheme, the algorithm performs 
remarkably well in practice, both on randomly generated problems (CQPs
as well as linear optimization problems) as well as data-fitting
problems.

Of course, while the focus of the present paper was on dense problems,
the concept of constraint reduction also applies to imbalanced large,
sparse problems. Indeed, whichever technique is used for solving the
Newton-KKT system, solving instead a reduced Newton-KKT system,
of like sparsity but of drastically reduced size, is bound to bring
in major computational savings when the total number of inequality
constraints is much larger than the number of inequality constraints
that are active at the solution---at least when the number of variables
is reasonably small compared to the number of inequality constraints.
In the case of sparse problems, the main computation cost in an IPM
iteration would be that of a (sparse) Cholesky factorization or, in 
an iterative approach to solving the linear system, would be linked to the
number of necessary iterations for reaching needed accuracy. In both 
cases, a major reduction in the dimension of the Newton-KKT system 
is bound to reduce the computation time, and like savings as in the 
dense case should be expected.

\bigskip\bigskip
\appendix
{\noindent \bf \normalsize Appendix\\}

\noindent
The following results are used in the proofs in 
Appendices~\ref{appendix:GC} and~\ref{appendix:LQC}.
Here we assume that $Q\subseteq\bm$
and $W$ is symmetric, with $W\succeq H\succ\bzero$.
First, from \eqref{eq:CR_KKT} and \eqref{eq:CR_KKT_C}, the approximate 
MPC search direction $(\Dbx, \Dbl_Q, \Dbs_Q)$ defined 
in \eqref{eq:search_dir_CR} solves 
\begin{equation}
J(W,A_Q,\bs_Q,\bslambda_Q)
  \left[\begin{array}{c}   
  \Dbx \\ \Dbl_Q\\  \Dbs_Q \end{array}\right]=
  \left[\begin{array}{c}   
  -\nabla f(\bx)+(A_Q)^T\bslambda_Q \\ \bzero\\  
  -S_Q\bslambda_Q + \gamma\sigma\mu_{(Q)}\mathbf{1}
  - \gamma\Delta S_Q^{\aff}\Delta\bslambda_Q^{\aff} 
  \end{array}\right]\:,
  \label{eq:CR_KKT_dx}
\end{equation}
and equivalently, when $\bs_Q>\bzero$, from~\eqref{eq:normal_sys_CR_P_x}, \eqref{eq:normal_sys_CR_P} and~\eqref{eq:normal_sys_CR_C},
\begin{equation}
\begin{alignedat}{2}
  M_{(Q)}&\Delta \bx=-\nabla f(\bx) + (A_Q)^TS_Q^{-1}(\gamma\sigma\mu_{(Q)}\mathbf{1}-\gamma\Delta S_Q^{\aff}\Delta\bslambda_Q^{\aff}),\\
  &\Dbs_Q=A_Q \Dbx,\\
  &\Dbl_Q=-\bslambda_Q + S_Q^{-1}(-\Lambda_Q\Dbs_Q+\gamma\sigma\mu_{(Q)}\mathbf{1}-\gamma\Delta S_Q^{\aff}\Delta\bslambda_Q^{\aff})\:.
\end{alignedat}
\label{eq:normal_sys_CR}
\end{equation}
Next, with  
$\tbl^+$ and $\tilde\bslambda^{\aff,+}$ given by
\begin{equation}
\tilde{\lambda}^+_i := \begin{cases}
\lambda_i + \Delta\lambda_i & i\in Q,\\
0 & i\in\Qc,
\end{cases}
\quad\text{ and }\quad
\tilde{\lambda}_i^{\aff,+} := \begin{cases}
\lambda_i + \Delta\lambda_i^\aff & i\in Q,\\
0 & i\in\Qc,
\end{cases}
\label{eq:tblq_tblaq_def}
\end{equation}
from the last equation
of~\eqref{eq:normal_sys_CR} and 
from~\eqref{eq:normal_sys_CR_P}, we have 
\begin{equation}
\tbl_Q^+ = S_Q^{-1}(-\Lambda_Q \Dbs_Q +
\gamma\sigma\mu_{(Q)}\mathbf{1} - 
\gamma\Delta S^{\aff}_Q \Delta \bslambda_Q^{\aff})\:,
\label{eq:tblq}
\end{equation}
\begin{equation}
\tilde\bslambda^{\aff,+}_Q = - S_Q^{-1} \Lambda_Q \Dbs^{\aff}_Q
= - S_Q^{-1} \Lambda_Q A_Q \Dbx^{\aff}_Q \:,
\label{eq:tblaq}
\end{equation}
and hence
\begin{equation}
(\Dbxa)^T(A_Q)^T\tilde\bslambda_Q^{\aff,+} = -(\Dbxa)^T(A_Q)^T S_Q^{-1}\Lambda_Q A_Q \Dbx^\aff \leq 0\:,
\label{eq:lemma_desc_pf2}
\end{equation}
so that, when in addition $\bslambda>\bzero$,
$(\Dbxa)^T(A_Q)^T\tilde\bslambda^{\aff,+} =0$ if 
and only if $A_Q\Dbx^\aff=\bzero$. 
Also, \eqref{eq:normal_sys_CR_P_x} yields
\begin{equation}
\nabla f(\bx)^T\Dbxa = -{(\Dbxa)}^T M_{(Q)}\Dbxa = - {(\Dbxa)}^T W\Dbxa
-{(\Dbxa)}^T (A_Q)^T S_Q^{-1}\Lambda_Q A_Q\Dbxa\:.
\label{eq:nablafTDx}
\end{equation}
Since $W\succeq H$, it follows from~\eqref{eq:lemma_desc_pf2} that,
\begin{equation}
\nabla f(\bx)^T\Dbxa + {(\Dbxa)}^T H \Dbxa \leq 0\:.
\label{eq:f'd+d'Hd<=0}
\end{equation}
In addition, when $S_Q\succ\bzero$, $\Lambda_Q\succ\bzero$ and since $W\succeq \bzero$, the right-hand side of \eqref{eq:nablafTDx} is strictly negative as long as $W\Dbxa$ 
and $A_Q\Dbxa$ are not both zero. 
In particular, when $[W~(A_Q)^T]$ has full row rank, 
\begin{equation}
\nabla f(\bx)^T\Dbxa<0\quad{\rm if}~\Dbxa\not=\bzero\:.
\label{eq:f'dx<0}
\end{equation}
Finally, we state and prove two technical lemmas.
\begin{lemma}
\label{lemma:dx_cvge}
Given an infinite index set $K$, $\{\Dbxk\}\to\bzero$ as $k\to\infty$, $k\in K$ if and only if $\{\Dbxak\}\to\bzero$ as $k\to\infty$, $k\in K$.
\end{lemma}
\begin{proof}
We show that $\|\Dbxk\|$ is sandwiched between constant multiples of $\|\Dbxak\|$.
We have from the search direction given in \eqref{eq:search_dir_CR} that, for all $k$,
$\|\Dbxk -\Dbxak\| = \|\gamma\Dbxck\|\leq \tau \|\Dbxak\|$,
where $\tau\in(0,1)$ and the inequality follows from \eqref{eq:mixing_para}.
Apply triangle inequality leads to
$(1-\tau)\|\Dbxak\|\leq\|\Dbxk\|\leq(1+\tau)\|\Dbxak\|$ for all $k$, proving the claim.
\end{proof}

\begin{lemma}
\label{lemma:step_size_bound}
Suppose Assumption~\ref{assum:non_empty_bdd_sol} holds.
Let $Q\subset\bm$, $\cA\subseteq Q$, $\bx\in \cF_P^o$, 
$\bs:=A\bx-\bb~(>\bzero)$, and
$\bslambda>\bzero$ enjoy
the following property:
With $\Dbl_Q$, $\Dbla_Q$, $\Dbs$, and $\Dbsa$ 
produced by Iteration~\ref{algo:CR-MPC},
$\lambda_i+\Delta\lambda_i>0$ 
for all $i\in\cA$ and $s_i+\Delta s_i>0$ for all 
$i\in Q\setminus\cA$. 
Then
\begin{equation}
\bar{\alpha}_\dual \geq \min\left\{ 1 , \,
\min_{i\in Q\setminus\cA}
\left\{\frac{s_i}{|s_i+\Delta s_i^{\aff}|}\right\} , \,
\min_{i\in Q\setminus\cA}
\left\{\frac{s_i-|\Delta s_i^\aff|}{|s_i+\Delta s_i|} \right\}
\right\}
\label{eq:step_size_bound_dual}
\end{equation}
and 
\begin{equation}
\bar{\alpha}_\prim \geq \min\left\{ 1 , \,
\min_{i\in\cA}
\left\{\frac{\lambda_i}{|\lambda_i+\Delta\lambda_i^\aff|}\right\} , \,
\min_{i\in\cA}
\left\{
\frac{\lambda_i-|\Delta\lambda_i^\aff|}{|\lambda_i+\Delta\lambda_i|}\right\}  
\right\}\:.
\label{eq:step_size_bound_primal}
\end{equation}
\end{lemma}

\begin{proof}
If $\bar{\alpha}_\dual \geq 1$, \eqref{eq:step_size_bound_dual} holds 
trivially, hence suppose $\bar{\alpha}_\dual < 1$.
Then, from the definition of $\bar{\alpha}_\dual$ in~\eqref{eq:step_size}, 
we know that there exists some 
index $i_0\in Q$ such that 
\begin{equation}
\Delta \lambda_{i_0}<-\lambda_{i_0}<0 \quad\text{  and  }\quad 
\bar{\alpha}_\dual = \frac{\lambda_{i_0}}{|\Delta \lambda_{i_0}|}\:.
\label{eq:i_0}
\end{equation}
Since $\lambda_i+\Delta\lambda_i>0$ for all $i\in\cA$,
we have $i_0\in Q\setminus\cA$.
Now we consider two cases: 
$|\Delta \lambda_{i_0}^\aff|\geq|\Delta \lambda_{i_0}|$ 
and $|\Delta \lambda_{i_0}^\aff|<|\Delta \lambda_{i_0}|$.
If $|\Delta \lambda_{i_0}^\aff|\geq|\Delta \lambda_{i_0}|$, 
then, since the second equation in~\eqref{eq:normal_sys_CR_P} 
is equivalently written as 
$\lambda_i s_i+s_i\Delta\lambda_i^\aff+\lambda_i\Delta s_i^\aff=0$
for all $i\in Q$ and since $\lambda_s s_i>0$ for all $i\in\bm$,
it follows from~\eqref{eq:i_0} that
\begin{equation*}
\bar{\alpha}_\dual = \frac{\lambda_{i_0}}{|\Delta \lambda_{i_0}|} 
\geq \frac{\lambda_{i_0}}{|\Delta \lambda_{i_0}^\aff|}
= \frac{s_{i_0}}{|s_{i_0}+\Delta s_{i_0}^{\aff}|}\:,
\end{equation*}
proving~\eqref{eq:step_size_bound_dual}.
To conclude, suppose now that
$|\Delta \lambda_{i_0}^\aff|<|\Delta \lambda_{i_0}|$.
Since (i) $s_i+\Delta s_i>0$ for $i\in Q\setminus\cA$;
(ii) $\gamma$, $\sigma$, and $\mu_{(Q)}$ in \eqref{eq:tblq} 
are non-negative; and (iii) $\Delta\lambda_{i_0}<0$ (from~\eqref{eq:i_0}), 
\eqref{eq:tblq_tblaq_def}--\eqref{eq:tblq}
yield
\begin{equation*}
{\lambda}_{i_0} (s_{i_0}+\Delta s_{i_0}) \geq s_{i_0}|\Delta \lambda_{i_0}|
- \gamma |\Delta s_{i_0}^{\aff}|
|\Delta\lambda_{i_0}^{\aff}|\:.
\end{equation*}
Applying this inequality to~\eqref{eq:i_0} leads to
\begin{equation*}
\bar{\alpha}_\dual = \frac{\lambda_{i_0}}{|\Delta \lambda_{i_0}|}
\geq \frac{s_{i_0}}{|s_{i_0}+\Delta s_{i_0}|}
- \frac{\gamma|\Delta \lambda_{i_0}^\aff| |\Delta s_{i_0}^\aff|}
       {|s_{i_0}+\Delta s_{i_0}||\Delta \lambda_{i_0}|}
\geq \frac{s_{i_0}-|\Delta s_{i_0}^\aff|}{|s_{i_0}+\Delta s_{i_0}|},
\end{equation*}
where the last inequality holds since $\gamma\leq 1$
and $|\Delta \lambda_{i_0}^\aff|<|\Delta \lambda_{i_0}|$. 
Following a very similar argument that flips the roles of $\bs$ 
and $\bslambda$, one can prove that \eqref{eq:step_size_bound_primal} 
also holds.
\end{proof}

\section{Proof of Theorem~\ref{thm:convergence} 
and Corollary~\ref{cor:exact_working_set}}
\label{appendix:GC}
Parts of this proof are inspired from~\cite{TZ:94}, \cite{JOT-12}, 
\cite{JungThesis}, \cite{WNTO-2012}, and~\cite{WTA-2014}.  Throughout, 
we assume that the constraint-selection rule used by the algorithm 
is such that Condition~\ref{cond:working_set_GC} is satisfied
and (except in the proof of Lemma~\ref{lemma:termination})
we let $\varepsilon=0$ and assume that the iteration never stops.

A central feature of Algorithm~\ref{algo:CR-MPC},
which plays a key role in the convergence proofs,
is that it forces
descent with respect of the primal objective function.
The next proposition establishes 
some related facts.
\begin{proposition} 
\label{proposition:desc_mpc}
Suppose $\bslambda>\bzero$ and $\bs>\bzero$, and $W$ satisfies $W\succ\bzero$
and $W\succeq H$.
If $\Dbxa \neq \bzero$, then the following inequalities hold:
\begin{align}
f(\bx+\alpha\Dbx^\aff)<f(\bx)\:,\quad \forall\alpha\in(0,2)\:,
\label{eq:desc_aff} \\
\frac{\partial}{\partial \alpha} f(\bx+\alpha\Dbx^\aff)< 0\:,\quad 
\forall\alpha\in[0,1]\:,
\label{eq:desc_aff_deri} \\
f(\bx)-f(\bx+\alpha\Dbx)\geq \frac{\omega}{2}(f(\bx)-f(\bx+\alpha\Dbx^\aff))\:,\quad\forall\alpha\in[0,1]\:,
\label{eq:prop_desc_mpc} \\
f(\bx+\alpha\Dbx)<f(\bx)\:,\quad\forall\alpha\in(0,1]\:.
\label{eq:prop_desc_mpc2}
\end{align}
\end{proposition}

\begin{proof}
When $f(\bx+\alpha\Dbxa)$ is linear in $\alpha$, i.e., when
$(\Dbxa)^T H\Dbxa=0$, 
then in view of~\eqref{eq:f'dx<0},
\eqref{eq:desc_aff}--\eqref{eq:desc_aff_deri} hold
trivially.  When, on the other hand, 
$(\Dbxa)^T H\Dbxa>0$, $f(\bx+\alpha\Dbxa)$ is
quadratic and strictly convex in $\alpha$ and is minimized at
\[
\hat\alpha 
= -\frac{\nabla f(\bx)^T \Dbxa} {{(\Dbxa)}^T H\Dbxa}
= 1+\frac{{(\Dbxa)}^T\left(W-H + (A_Q)^T S_Q^{-1}\Lambda_Q A_Q\right)\Dbxa}
{(\Dbxa)^T H\Dbxa}
\geq1,
\]
where we have used~\eqref{eq:nablafTDx}, \eqref{eq:lemma_desc_pf2}, 
and the fact that $W\succeq H$,
and~\eqref{eq:desc_aff}--\eqref{eq:desc_aff_deri} again follow.
Next, note that, since $\omega>0$,
\begin{equation*}
\psi(\theta):= \omega(f(\bx)-f(\bx+\Dbxa)) - (f(\bx)-f(\bx+\Dbxa+\theta\Dbxc))\:
\end{equation*}
is quadratic and convex.
Now, since $\gamma_1$ satisfies the constraints in its 
definition~\eqref{eq:mixing_para_1}, we see that $\psi(\gamma_1)\leq 0$, and
since $\omega\leq1$, it follows from~\eqref{eq:desc_aff}
that $\psi(0)= (\omega-1)(f(\bx)-f(\bx+\Dbxa))\leq 0$.
Since $\gamma\in[0,\gamma_1]$ (see \eqref{eq:mixing_para}), it follows that 
$\psi(\gamma)\leq0$, i.e., 
since from~\eqref{eq:search_dir_CR} $\Dbx=\Dbxa+\gamma\Dbxc$,
\begin{equation*}
f(\bx)-f(\bx+\Dbx)\geq \omega(f(\bx)-f(\bx+\Dbx^{\aff}))\:,
\end{equation*}
i.e.,
\begin{equation}
-\nabla f(\bx)^T \Dbx-\frac{1}{2}{\Dbx}^T H\Dbx
\geq \omega \left( -\nabla f(\bx)^T \Dbxa -\frac{1}{2}{(\Dbxa)}^T H\Dbxa \right).
\label{eq:La1iii}
\end{equation}
Now, for all $\alpha\in[0,1]$, invoking~\eqref{eq:La1iii},
\eqref{eq:f'd+d'Hd<=0}, and 
the fact that $H\succeq\bzero$, we can write
\begin{eqnarray*}
f(\bx)-f(\bx+\alpha\Dbx) 
= -\alpha\nabla f(\bx)^T\Dbx - \frac{\alpha^2}{2}{\Dbx}^T H \Dbx
& \geq \alpha  \left(-\nabla f(\bx)^T\Dbx - \frac{1}{2}{\Dbx}^T H \Dbx\right) \\
\geq \omega\alpha \left(-\nabla f(\bx)^T\Dbxa
- \frac{1}{2}{(\Dbxa)}^T H \Dbxa \right) \\
= \frac{\omega\alpha}{2}\left(-\nabla f(\bx)^T\Dbxa - \left(\nabla f(\bx)^T\Dbxa
+ {(\Dbxa)}^T H \Dbxa\right) \right)
& \geq  \frac{\alpha\omega}{2} \left(-\nabla f(\bx)^T\Dbxa\right) \\
\geq \frac{\alpha\omega}{2}
\left(-\nabla f(\bx)^T\Dbxa -\frac{\alpha}{2}{(\Dbxa)}^T H\Dbxa \right)
& = \frac{\omega}{2}(f(\bx)-f(\bx+\alpha\Dbx^\aff))\:,
\end{eqnarray*}
proving~\eqref{eq:prop_desc_mpc}.
Finally, since $\omega>0$, \eqref{eq:prop_desc_mpc2} is a direct consequence 
of~\eqref{eq:prop_desc_mpc} and~\eqref{eq:desc_aff}.
\end{proof}
Given that the iterates are primal-feasible, an immediate 
consequence of Proposition~\ref{proposition:desc_mpc} is that
the primal sequence is bounded.
\begin{lemma} 
Suppose Assumption~\ref{assum:non_empty_bdd_sol} holds. 
Then $\{\bx^k\}$ is bounded.
\label{lemma:bdd_seq_x}
\end{lemma}
We are now ready to prove a key result, relating two successive
iterates, that plays a central role in the remainder of the proof of
Theorem~\ref{thm:convergence}.
\begin{proposition} 
Suppose Assumptions~\ref{assum:non_empty_bdd_sol} and \ref{assum:lin_ind_act} hold,
and either $\{(\bx^k,\bslambda^k)\}$ is bounded away from $\cF^*$, or
Assumption \ref{assum:singleton_sol+strict_complementary} also holds and
$\{\bx^k\}$ converges to the unique primal solution $\bx^*$.
Let $K$ be an infinite index set such that 
\begin{equation}
\left(\inf_{k\in K} \{\chi_{k-1}\}=\right)~
\inf\left\{\|\Dbx^{\aff,k-1}\|^\nu 
+ \|[\tbl^{\aff,k}_{Q_{k-1}}]_{-}\|^\nu \colon k\in K\right\}>0\:.
\label{eq:lemma_dx_to_zero_assumption}
\end{equation}
Then $\{\Dbxk\}\to\bzero$ as $k\to\infty$, $k\in K$.
\label{prop:dx_to_zero}
\end{proposition}

\begin{proof}
From Lemma~\ref{lemma:bdd_seq_x}, $\{\bx^k\}$ is bounded,
and hence so is $\{\bs^k\}$;
by construction, $\bs^k$ and $\bslambda^k$
have positive components for all $k$, and $\{\bslambda^k\}$ 
(\eqref{eq:update_lambda_Q}--\eqref{eq:update_lambda_notQ}) 
and $\{W_k\}$ are bounded.
Further, for any infinite index set $K^\prime$ such 
that~\eqref{eq:lemma_dx_to_zero_assumption} holds,
\eqref{eq:update_lambda_Q} and \eqref{eq:update_lambda_notQ}
imply that all
components of $\{\bslambda^k\}$ are bounded away from zero on $K^\prime$.
Since, in addition, $Q_k$ can take no more than finitely many
different (set) values, it follows that
there exist $\hat\bx$, $\hat\bslambda>\bzero$, $\hat W\succeq\bzero$, an index
set $\hat Q\subseteq\bm$, and some infinite index set $\hat K\subseteq K^\prime$
such that
\begin{align}
 \{\bx^k\}&\to \hat\bx \text{ as } k\to\infty\:,\: k\in \hat K\:,
\nonumber\\
 \{\bs^k\} &\to  \hat\bs:= \{A\hat\bx-\bb\}\geq\bzero \text{ as } k\to\infty\:,\: k\in \hat K\:.
\label{eq:lemma_dx_to_zero_s_cvge}\\
\{\bslambda^k\}&\to \hat\bslambda>\bzero \text{ as } k\to\infty\:,\: k\in \hat K\:,
\label{eq:lemma_dx_to_zero_lambda_cvge}\\
\{W_k\}&\to \hat W \text{ as } k\to\infty\:,\: k\in \hat K\:,\nonumber\\
 Q_k &= \hat Q\:,\: \forall k\in \hat K\:.
\label{eq:lemma_dx_to_zero_Q_star}
\end{align}
Next, under the stated assumptions, 
$J(\hat W,A_{\hat Q}, \hat\bs_{\hat Q},\hat\bslambda_{\hat Q})$ is 
non-singular.  
Indeed, if $\{(\bx^k,\bslambda^k)\}$ is bounded away from $\cF^*$, then
$E(\bx^k,\bslambda^k)$ is bounded away from zero and since 
$H+R\succ\bzero$, 
$W_k=H+\varrho_k R = H + 
\min\left\{1,\frac{E(\bx^k,\bslambda^k)}{\bar E}\right\} R$ 
is bounded away from singularity 
and the claim follows from Assumption~\ref{assum:lin_ind_act} and Lemma~\ref{lemma:nonsingular_J}.
On the other hand, if Assumption \ref{assum:singleton_sol+strict_complementary} also 
holds and $\{\bx^k\}\to\bx^*$, then the claim follows from Condition~\ref{cond:working_set_GC}(ii) and Lemma~\ref{lemma:nonsingular_J}.
As a consequence of this claim, and by continuity of $J$, it follows
from Newton-KKT systems~\eqref{eq:CR_KKT} and \eqref{eq:CR_KKT_dx} that 
there exist $\Delta\hat\bx^\aff$, $\Delta\hat\bx$, $\bar{\bslambda}^{\aff}_{\hat Q}$, $\bar\bslambda_{\hat Q}$ such that 
\begin{align}
 \{\Dbxak\}&\to \Delta\hat\bx^\aff \text{ as } k\to\infty\:,\: k\in \hat K\:,
\label{eq:lemma_dx_to_zero_Dbxak_cvge}\\
 \{\Dbx^k\}&\to \Delta\hat\bx \text{ as } k\to\infty\:,\: k\in \hat K\:,
\nonumber\\
 \{\Dbs^k\}&\to \Delta\hat\bs:=A\Delta\hat\bx \text{ as } k\to\infty\:,\: k\in \hat K\:,
\label{eq:lemma_dx_to_zero_Dbsk_cvge}\\
 \{\tilde\bslambda^{\aff,k+1}_{\hat Q}\}
&\to \bar{\bslambda}^{\aff}_{\hat Q}\text{ as } k\to\infty\:,\: k\in \hat K\:,
\label{eq:lemma_dx_to_zero_tblak_cvge}\\
 \{\tilde\bslambda^{k+1}_{\hat Q}\}
&\to \bar{\bslambda}_{\hat Q} \text{ as } k\to\infty\:,\: k\in \hat K\:,
\label{eq:lemma_dx_to_zero_tblk_cvge}
\end{align}

The remainder of the proof proceeds by contradiction.
Thus suppose that, for the infinite index set $K$ in the
statement of this lemma,
$\{\Dbxk\}\not\to\bzero$ as $k\to\infty$, $k\in K$, i.e.,
for some $K^{\prime\prime}\subseteq K$, $\|\Dbxk\|$ is bounded away
from zero on $K^{\prime\prime}$.  Use $K^{\prime\prime}$ as our 
$K^\prime$ above, so that
(since $\hat K\subseteq K^\prime$), $\|\Dbxk\|$ is
bounded away from zero on $\hat K$.
Then, in view of Lemma~\ref{lemma:dx_cvge} (w.l.o.g.),
\begin{equation}
  \inf_{k\in \hat K}\|\Dbxak\|>0\:.
\label{eq:lemma_dx_to_zero_Dbxak_notzero}
\end{equation}
In addition, we have $\cA(\hat\bx)\subseteq \hat Q$, an implication of
Condition~\ref{cond:working_set_GC}(i) 
when $\{(\bx^k,\bslambda^k)\}$ is bounded away from $\cF^*$ and of
Condition~\ref{cond:working_set_GC}(ii) when Assumption~\ref{assum:singleton_sol+strict_complementary} holds and $\{\bx^k\}$ converges to $\bx^*$.
With these facts in hand, we next show that the sequence of primal
step sizes $\{\alpha_\prim^k\}$ is bounded away from zero for $k\in \hat K$.
To this end, let us define
\begin{equation}
\tilde\bslambda^{\prime,k+1} :=-(S^{k})^{-1}\Lambda^k\Dbs^k\:,\quad \forall k\:,
\label{eq:lemma_dx_to_zero_hblk_def}
\end{equation}
so that, for all $i\in\bm$ and all $k$, 
$\tilde\lambda_i^{\prime,k+1}>0$ if and only if $\Delta s_i^k<0$,
and the primal portion of~\eqref{eq:step_size} can be 
written as
\begin{equation*}
  \begin{alignedat}{2}
  \bar{\alpha}_\prim^k &:= 
  \begin{cases}
  \infty & \text{if } \tilde\bslambda^{\prime,k+1}\leq \bzero\:,\\
    \min_i\left\{ \frac{\lambda_i^k}{\tilde\lambda_i^{\prime,k+1}} : 
     \tilde{\lambda}_i^{\prime,k+1}>0\right\} & \text{otherwise.}
  \end{cases}\\
  \alpha_\prim^k &:= \min\left\{1,\,\max\{\varkappa\bar{\alpha}_\prim,\,\bar{\alpha}_\prim-\|\Delta\bx^k\|\}\right\}\:.
  \end{alignedat}
\end{equation*}
Clearly, it is now sufficient to
show that, for all $i$, $\{\tilde{\lambda}_i^{\prime,k+1}\}$ is bounded above 
on $\hat K$.  On the one hand, 
this is clearly so for $i\not\in \hat Q$
(whence $i\not\in\cA(\hat\bx)$),
in view of~\eqref{eq:lemma_dx_to_zero_hblk_def} 
and~\eqref{eq:lemma_dx_to_zero_Dbsk_cvge},
since $\{\bslambda^k\}$ is bounded and
$\{s_i^k\}$ is bounded away from zero on $\hat K$ for $i\not\in\cA(\hat\bx)$
(from~\eqref{eq:lemma_dx_to_zero_s_cvge}).
On the other hand, 
in view of~\eqref{eq:lemma_dx_to_zero_Q_star},
subtracting~\eqref{eq:lemma_dx_to_zero_hblk_def} 
from~\eqref{eq:tblq} yields, for all $k\in \hat K$,
\begin{equation*}
\tilde\bslambda'^{k+1}_{\hat Q} = \tilde\bslambda^{k+1}_{\hat Q} - 
\gamma_k\sigma_k\mu_{({\hat Q})}^k(S_{\hat Q}^{k})^{-1}\mathbf{1} 
+ \gamma_k(S_{\hat Q}^{k})^{-1}\Delta S_{\hat Q}^{\aff,k}\Dbl_{\hat Q}^{\aff,k}\:.
\end{equation*}
From \eqref{eq:lemma_dx_to_zero_tblk_cvge}, 
$\{\tilde\bslambda^{k+1}_{\hat Q}\}$ is bounded 
on $\hat K$, and clearly the second term in the right-hand side 
of the above equation is non-positive
component-wise.  
As for the third term, the second equation in \eqref{eq:normal_sys_CR_P} gives 
$(S_{Q_k}^{k})^{-1}\Delta S_{Q_k}^{\aff,k} = (\Lambda_{Q_k}^{k})^{-1} \tilde{\Lambda}_{Q_k}^{\aff,k+1}$, 
so that we have
\begin{equation*}
\gamma_k(S_{\hat Q}^{k})^{-1}\Delta S_{\hat Q}^{\aff,k}\Dbl_{\hat Q}^{\aff,k} 
= \gamma_k(\Lambda_{\hat Q}^{k})^{-1} 
\tilde{\Lambda}_{\hat Q}^{\aff,k+1}\Dbl_{\hat Q}^{\aff,k}\:, 
\quad\forall k \in \hat K\:,
\end{equation*}
which is bounded on $\hat K$ since, 
from~\eqref{eq:lemma_dx_to_zero_lambda_cvge},
\eqref{eq:lemma_dx_to_zero_tblak_cvge}, and the
definition~\eqref{eq:tblq_tblaq_def} of $\{\tilde{\lambda}^{\aff,+}\}$,
both 
$\{\tilde{\Lambda}_{\hat Q}^{\aff,k+1}\}$ and $\{\Dbl_{\hat Q}^{\aff,k}\}$ 
are bounded, and from~\eqref{eq:lemma_dx_to_zero_lambda_cvge}, 
$\{\bslambda^k_{\hat Q}\}$ is bounded away 
from zero on $\hat K$.
Therefore, $\{\tilde{\lambda}_i^{\prime,k+1}\}$ is bounded above
on $\hat K$ for $i\in \hat Q$ as well, proving that
$\{{\alpha}_\prim^k\}$ is bounded away from zero on $\hat K$,
i.e., that there exists $\underline{\alpha}>0$ such that 
$\alpha_\prim^k>\underline{\alpha}$, for all $k\in \hat K$, as claimed.
Without loss of generality, choose $\underline{\alpha}$ in $(0,1)$.

Finally, we show that $\{f(\bx^k)\}\to -\infty$ as $k\to\infty$ on $\hat K$, 
which contradicts boundedness of $\{\bx^k\}$ (Lemma~\ref{lemma:bdd_seq_x}).
For all $k\in \hat K$, since $\Dbxak\neq \bzero$ 
(by~\eqref{eq:lemma_dx_to_zero_Dbxak_notzero}) and 
$\alpha_\prim^k\in(\underline{\alpha}, 1]$,
Proposition~\ref{proposition:desc_mpc} implies that $\{f(\bx^k)\}$ 
is monotonically decreasing and that, for all $k\in\hat K$,
\begin{equation*}
f(\bx^k+\alpha_\prim^k\Dbxak) < f(\bx^k+\underline{\alpha}\Dbxak)\:.
\end{equation*}
Expanding the right-hand side yields
\begin{align*}
 f(\bx^k+\underline{\alpha}\Dbxak)
= f(\bx^k) + \underline{\alpha}\nabla f(\bx^k)^T \Dbxak 
+ \frac{\underline\alpha^2}{2} (\Dbxak)^T H\Dbxak\\
= f(\bx^k) + \underline{\alpha} 
\left(\nabla f(\bx^k)^T \Dbxak + (\Dbxak)^T H\Dbxak\right) 
-\left(\underline\alpha - \frac{\underline\alpha^2}{2}\right)
(\Dbxak)^T H\Dbxak\:,
\end{align*}
where the sum of the last two terms tends to a strictly 
negative limit as $k\to\infty$, $k\in\hat K$.
Indeed, in view of~\eqref{eq:f'd+d'Hd<=0}, the second term is 
non-positive and (i) if $(\Delta\hat\bx^{\aff})^T H\Delta\hat\bx^{\aff}>0$,
since $\underline\alpha>\underline\alpha^2/2$, 
from~\eqref{eq:lemma_dx_to_zero_Dbxak_cvge} 
and~\eqref{eq:lemma_dx_to_zero_Dbxak_notzero},
the third
term tends to a negative limit, and (ii) if
$(\Delta\hat\bx^{\aff})^T H\Delta\hat\bx^{\aff}=0$ then the sum
of the last two terms
tends to $\underline\alpha\nabla f(\hat{\bx})^T\Delta\hat\bx^{\aff}$
which is also strictly negative in view of~\eqref{eq:f'dx<0}, 
since we either have $\hat W 
\succ\bzero$ 
(in the case that $\{(\bx^k,\bslambda^k)\}$ bounded away from $\cF^*$) 
or at least $[\hat W \,(A_{\hat{Q}})^T]$ full row rank (in 
the case that 
Assumption~\ref{assum:singleton_sol+strict_complementary} holds
and using the fact that $\cA(\hat\bx)\subseteq \hat{Q}$).
It follows that, for some $\delta>0$,
$f(\bx^k+\alpha_\prim^k\Dbxak)<f(\bx^k)-\delta$ for all $k\in\hat K$
large enough. 
Proposition~\ref{proposition:desc_mpc} (eq.~\eqref{eq:prop_desc_mpc})
then gives that 
$f(\bx^{k+1}):=f(\bx^k+\alpha_\prim^k\Dbxk)<f(\bx^k)-\frac{\omega}{2}\delta$ 
for all $k\in\hat K$ large enough, where $\omega>0$ is an algorithm
parameter. 
Since $\{f(\bx^k)\}$ is monotonically decreasing,
the proof is now complete.
\end{proof}

We now conclude the proof of Theorem~\ref{thm:convergence} via
a string of eight lemmas, each of which builds on the previous one.
First, on any subsequence, if $\{\Dbxak\}$ tends to zero,
then $\{\bx^k\}$
approaches stationary points.
(Here both $\{\tilde\bslambda^{\aff,k+1}\}$ and
$\{\tilde\bslambda^{k+1}\}$ are as defined in~\eqref{eq:tblq_tblaq_def}.)
\begin{lemma} 
Suppose that Assumption~\ref{assum:non_empty_bdd_sol} holds and
that $\{\bx^k\}$ converges to some limit point $\hat\bx$ on an 
infinite index set $K$.
If $\{\Dbxak\}$ converges to zero on $K$, then (i) $\hat\bx$ is stationary
and
\begin{equation}
\nabla f(\bx^k) - \left(A_{\cA(\hat\bx)}\right)^T\tilde\bslambda^{\aff,k+1}_{\cA(\hat\bx)}
 \to \bzero \:,
 \text{ as } k\to\infty\:,\,k\in K\:.
\label{eq:tlak_conv2}
\end{equation}
If, in addition, Assumption~\ref{assum:lin_ind_act} holds, then
(ii) $\{\tilde\bslambda^{\aff,k+1}\}$ and
$\{\tilde\bslambda^{k+1}\}$ 
converge on $K$ to $\hat\bslambda$, 
the unique multiplier associated with $\hat\bx$.
\label{lemma:stationary_multiplier}
\end{lemma}

\begin{proof}
Suppose $\{\bx^k\}\to\hat\bx$ on $K$ and $\{\Dbx^{\aff,k}\}\to \bzero$ on $K$. 
Let $\bs^k:=A\bx^k-\bb(>\bzero)$ for all $k\in K$ and $\hat{\bs}:=A\hat\bx-\bb(\geq\bzero)$, 
so that 
$\{\bs^k\}\to\hat{\bs}$ on $K$.
As a first step toward proving Claim~(i), we show that, for any
$i\not\in\cA(\hat\bx)$, 
$\{\tilde\lambda_i^{\aff,k+1}\}\to0$ on $K$. 
For $i\not\in\cA(\hat\bx)$, since $\hat{s}_i>0$, $\{s_i^k\}$ is bounded 
away from zero on $K$.  
Since it follows from~\eqref{eq:tblq_tblaq_def} 
and~\eqref{eq:tblaq} that, for all $k$,
\begin{equation*}
 \tilde{\lambda}_i^{\aff,k+1}=0\:,\, \forall i \not\in Q_k
 \quad\text{and}\quad
 \tilde{\lambda}_i^{\aff,k+1}=-{(s_i^k)}^{-1}\lambda_i^k \Delta s_i^{\aff,k}\:, \,\forall i\in Q_k\:,
\end{equation*}
and since $\{\lambda_i^k\}$ is bounded (by construction)
and $\Dbs^{\aff,k}=A\Dbx^{\aff,k}$ (by~\eqref{eq:normal_sys_CR_P}), we have
$\{\tilde\lambda_i^{\aff,k+1}\}\to0$ on $K$.
To complete the proof of Claim~(i), note that the first equation
of \eqref{eq:CR_KKT} (with $H$ replaced by $W$) yields
\begin{equation*}
\nabla f(\bx^k) - (A_{Q_k})^T\tilde\bslambda^{\aff,k+1}_{Q_k} = - W_k \Dbx^{\aff,k}\:.
\end{equation*}
Since (i) $\{\tilde\lambda^{\aff,k+1}_i\}\to0$ on $K$ for $i\not\in\cA(\hat\bx)$, 
(ii) $\{W_k\}$ is bounded (since $H\preceq W_k\preceq H+R$), 
(iii) $\{\Dbxak\}\to\bzero$ on $K$, 
and (iv) by definition~\eqref{eq:tblq_tblaq_def},
$\tilde{\lambda}_i^{\aff,+}=0$ for $i\in\Qc$, we conclude 
that~\eqref{eq:tlak_conv2} holds,
hence $\{\left(A_{\cA(\hat\bx)}\right)^T\tilde\bslambda^{\aff,k+1}_{\cA(\hat\bx)}\}$ 
converges (since $\nabla f(\bx^k)$ does) as $k\to\infty$, $k\in K$, to a 
point in the range of $\left(A_{\cA(\hat\bx)}\right)^T$, say 
$\left(A_{\cA(\hat\bx)}\right)^T \hat\bslambda_{\cA(\hat\bx)}$.
We get $\nabla f(\hat\bx) - \left(A_{\cA(\hat\bx)}\right)^T \hat\bslambda_{\cA(\hat\bx)}=\bzero$, 
proving Claim~(i).
Finally, Claim~(ii) follows from~\eqref{eq:tlak_conv2}, 
Assumption~\ref{assum:lin_ind_act}, and the fact
that for $i\not\in\cA(\hat\bx)$, 
$\{\tilde\lambda_i^{\aff,k+1}\}\to0$ as $k\to\infty$, $k\in K$,
noting that the same argument applies to $\{\tilde\bslambda^{k+1}\}$, 
using a modified version of~\eqref{eq:tlak_conv2}, with
$\tilde\bslambda$ replacing $\tilde\bslambda^\aff$,
obtained by starting from the first equation of~\eqref{eq:CR_KKT_dx} 
instead of that of~\eqref{eq:CR_KKT} 
and using the fact, proved next,
that $\{\tilde\lambda_i^{k+1}\}\to0$ on $K$ for all $i\not\in\cA(\hat{\bx})$.
From its definition in~\eqref{eq:tblq_tblaq_def} and the last equation 
in~\eqref{eq:normal_sys_CR}, we have that, 
for all $k$,
\begin{align*}
 \tilde{\lambda}_i^{k+1}=0\:,\quad &\forall i \not\in Q_k\:, \\
 \tilde{\lambda}_i^{k+1} = {(s_i^k)}^{-1}(-\lambda_i^k \Delta s_i^k +
\gamma_k\sigma_k\mu_{(Q_k)}^{(k)} -
\gamma_k\Delta s_i^{\aff,k} \Delta \lambda_i^{\aff,k})\:, \quad&\forall i\in Q_k\:.
\end{align*}
Since $\{\tilde\lambda_i^{\aff,k+1}\}$ converges (to zero) on $K$, 
$\{\Delta\lambda_i^{\aff,k}\}$ is bounded on $K$.  Furthermore, 
from its definition~\eqref{eq:mixing_para}--\eqref{eq:mixing_para_1}
(see also~\eqref{eq:search_dir_CR}), $\{\gamma_k\}$ is bounded and
$|\gamma_k\sigma_k\mu_{(Q_k)}^{(k)}|\leq\tau\|\Dbxak\|$ for all $k$.
Since $\Dbs^{\aff,k}=A\Dbx^{\aff,k}$ and $\Dbs^k=A\Dbx^k$, in
view of Lemma~\ref{lemma:dx_cvge},
it follows that, for $i\not\in\cA(\hat\bx)$, 
$\{\tilde\lambda_i^{k+1}\}\to0$ on $K$.
\end{proof}
Lemma~\ref{lemma:stationary_multiplier}, combined with
Proposition~\ref{prop:dx_to_zero} via a contradiction argument,
then implies that (on a subsequence), if $\{\bx^k\}$ does
not approach $\cF_P^*$, then $\{\Dbxk\}$ approaches zero.

\begin{lemma} 
Suppose that Assumptions~\ref{assum:non_empty_bdd_sol} 
and~\ref{assum:lin_ind_act} hold and that
$\{\bx^k\}$ is bounded away from $\cF_P^*$ on some infinite index set $K$.
Then $\{\Dbxk\}\to\bzero$ as $k\to\infty$, $k\in K$.
\label{lemma:dx_to_zero_2}
\end{lemma}
\begin{proof}
Proceeding by contradiction, let $K$ be an infinite index 
set such that $\{\bx^k\}$ is bounded away from $\cF_P^*$ on $K$ and
$\{\Dbxk\}\not\to\bzero$ as $k\to\infty$, $k\in K$. 
Then, in view of Proposition~\ref{prop:dx_to_zero} and boundedness of
$\{\bx^k\}$ (Lemma~\ref{lemma:bdd_seq_x}), there exist 
$\hat Q\subseteq\bm$, $\hat\bx\not\in \cF_P^*$, and
an infinite index set $\hat K\subseteq K$ such that 
$Q_k = \hat Q$ for all $k\in\hat K$ and
\begin{equation*}
\begin{alignedat}{2}
&\{\bx^{k}\} \to \hat\bx\:, \text{ as } k\to\infty,\: k\in \hat K\:,\\
&\{\Dbx^{\aff,k-1}\} \to \bzero\:, \text{ as } k\to\infty,\: k\in \hat K\:,\\
&\{[\tbl^{\aff,k}_{\hat Q}]_{-}\} \to \bzero\:, \text{ as } k\to\infty,\: k\in\hat K\:.
\end{alignedat}
\end{equation*}
On the other hand, from \eqref{eq:update}, \eqref{eq:search_dir_CR} 
and~\eqref{eq:mixing_para}--\eqref{eq:mixing_para_1},  
\begin{equation*}
\|\bx^k-\bx^{k-1}\|=\|\alpha_\prim^{k-1}\Dbx^{k-1}\|\leq \|\Dbx^{k-1}\| 
\leq (1+\tau) \|\Dbx^{\aff,k-1}\|\:,
\end{equation*}
which implies that $\{\bx^{k-1}\}\to\hat\bx$ as $k\to\infty$, $k\in\hat K$.
It then follows from Lemma~\ref{lemma:stationary_multiplier} that
$\hat\bx$ is stationary and that $[\tilde\bslambda^{\aff,k}_{\cA(\hat\bx)}]_+$
converges to the associated multiplier vector.  Hence the 
multipliers are non-negative,
contradicting the fact that $\hat\bx\not\in\cF_P^*$.
\end{proof}
A contradiction argument based on 
Lemmas~\ref{lemma:stationary_multiplier} and~\ref{lemma:dx_to_zero_2}
then shows that 
$\{\bx^k\}$ approaches the set of stationary points 
of~\eqref{eq:primal_cqp}.

\begin{lemma} 
Suppose Assumptions~\ref{assum:non_empty_bdd_sol} and~\ref{assum:lin_ind_act} hold.
Then the sequence $\{\bx^k\}$ approaches the set of stationary points 
of~\eqref{eq:primal_cqp}, i.e., there exists a sequence $\{\hat{\bx}^k\}$
of stationary points such that $\|\bx^k-\hat{\bx}^k\|$ goes to zero 
as $k\to\infty$.
\label{lemma:dx_to_stationary}
\end{lemma}

\begin{proof}
Proceeding by contradiction, suppose the claim does not hold,
i.e., (invoking Lemma~\ref{lemma:bdd_seq_x}) suppose $\{\bx^k\}$ converges 
to some non-stationary point $\hat\bx$ on some infinite index set $K$.
Then $\{\Dbxak\}$ does not converge to zero on $K$
(Lemma~\ref{lemma:stationary_multiplier}(i))
and nor does $\{\Dbxk\}$ (Lemma~\ref{lemma:dx_cvge}).
Since $\hat\bx$ is non-stationary, this is in contradiction with
Lemma~\ref{lemma:dx_to_zero_2}.
\end{proof}
The next technical result, proved in~\cite[Lemma~3.6]{TZ:94},
invokes analogues of Lemmas~\ref{lemma:bdd_seq_x},~\ref{lemma:dx_to_zero_2}
and~\ref{lemma:dx_to_stationary}.
\begin{lemma}
Suppose Assumptions~\ref{assum:non_empty_bdd_sol} 
and~\ref{assum:lin_ind_act} hold.
Suppose $\{\bx^k\}$ is bounded away from $\cF_P^*$. 
Let $\hat\bx$ and $\hat\bx^{\prime}$ be limit points of $\{\bx^k\}$
and let $\hat\bslambda$ and $\hat\bslambda^\prime$ be the associated 
KKT multipliers.  Then $\hat\bslambda = \hat\bslambda^\prime$.
\label{lemma:common_multiplier}
\end{lemma}
Convergence of $\{\bx^k\}$ to $\cF_P^*$ ensues, proving
Claim~(i) of Theorem~\ref{thm:convergence}.
\begin{lemma}
Suppose Assumptions~\ref{assum:non_empty_bdd_sol}
and~\ref{assum:lin_ind_act} hold.
Then $\{\bx^k\}$ converges to $\cF_P^*$.
\label{lemma:x_convergence}
\end{lemma}
\begin{proof}
We proceed by contradiction. 
Thus suppose $\{\bx^k\}$ does not converge 
to $\cF_P^*$.  Then, since $\{\bx^k\}$ is bounded 
(Lemma~\ref{lemma:bdd_seq_x}), 
(by Proposition~\ref{proposition:desc_mpc}) $\{f(\bx^k)\}$ is 
a bounded, monotonically decreasing sequence, and
it has at least one limit point $\hat{\bx}$ that is not in $\cF_P^*$.
Hence, $f(\hat{\bx})=\inf_{k}f(\bx^k)$.
Then, by Lemmas~\ref{lemma:dx_to_zero_2} and~\ref{lemma:dx_cvge},
$\{\Dbx^k\}$ and $\{\Dbxak\}$ converge to zero 
as $k\to\infty$.
It follows from Lemmas~\ref{lemma:stationary_multiplier} 
and~\ref{lemma:common_multiplier} that all limit points 
of $\{\bx^k\}$ are stationary, and that 
both $\{\tblak\}$ and $\{\tblk\}$ converge to $\hat\bslambda$, the common KKT multiplier
vector associated to all limit points of $\{\bx^k\}$.
Since 
$\hat\bx\not\in\cF_P^*$,
there exists $i_0$ such that $\hat\lambda_{i_0}<0$,
so that, for some $\hat{k}>0$,
\begin{equation}
\tilde{\lambda}_{i_0}^{\aff,k+1}<0\:\text{ and } \tilde{\lambda}_{i_0}^{k+1}<0\:, 
\quad \forall k>\hat{k}\:,
\label{eq:tilde_lambdas}
\end{equation}
which, in view of Step~8 of the algorithm, implies that
$i_0\in Q_k$ for all $k>\hat k$.
Then \eqref{eq:tblaq} gives
\begin{equation*}
\Delta s_{i_0}^{\aff,k} = -(\lambda_{i_0}^k)^{-1} 
s_{i_0}^k\tilde{\lambda}_{i_0}^{\aff,k+1} \:,\quad\forall k>\hat k\:,
\end{equation*}
where $s_{i_0}^k>0$, $\lambda_{i_0}^k>0$ by construction. 
Thus, in view of~\eqref{eq:tilde_lambdas}, 
$\Delta s_{i_0}^{\aff,k}>0$ for all $k>\hat k$.
On the other hand, the last equation of \eqref{eq:normal_sys_CR} gives
\begin{equation}
\Delta s_{i_0}^k = (\lambda_{i_0}^k)^{-1} 
(-s_{i_0}^k\tilde{\lambda}_{i_0}^{k+1} + \gamma_k \sigma_k\mu_{(Q_k)}^k
- \gamma_k \Delta s_{i_0}^{\aff,k} \Delta \lambda_{i_0}^{\aff,k}),
\quad\forall k > \hat k \:,
\label{eq:Delta_s}
\end{equation}
where $\gamma_k\geq0$, $\sigma_k\geq0$, and $\mu_{(Q_k)}^k\geq0$ by 
construction. 
Further, for $k>\hat k$, $\Delta\lambda_{i_0}^{\aff,k}<0$ 
since $\lambda_{i_0}^k>0$ and $\tilde{\lambda}_{i_0}^{\aff,k+1}(=\lambda_{i_0}^k+\Delta\lambda_{i_0}^{\aff,k})<0$.
It follows that all terms in \eqref{eq:Delta_s} are 
non-negative and the first term is positive, 
so that $\Delta s_{i_0}^k>0$ for all $k>k^\prime$. 
Moreover, for all $k>\hat k$, we have $s_{i_0}^{k+1} = s_{i_0}^{k} + \alpha_\prim^k \Delta s_{i_0}^k > s_{i_0}^{k} > 0$, where $\alpha_\prim^k>0$ since $\bs^k > \bzero$. 
Since $\{\bs^k\}$ is bounded (Lemma~\ref{lemma:bdd_seq_x}), we then 
conclude that $\{s_{i_0}^k\}\to \hat{s}_{i_0}>0$ 
so that $\hat{s}_{i_0}\hat{\lambda}_{i_0}<0$, in contradiction with 
the stationarity of limit points.  
\end{proof}
Under strict complementarity, the next lemma then 
establishes appropriate convergence of the multipliers, setting the
stage for the proof of part (ii) of Theorem~\ref{thm:convergence}
in the following lemma.

\begin{lemma}
Suppose Assumptions~\ref{assum:non_empty_bdd_sol}
to~\ref{assum:singleton_sol+strict_complementary} hold
and let $(\bx^*,\bslambda^*)$ be the unique primal-dual solution.
Then, 
given any infinite index set $K$ such that $\{\Dbxak\}_{k\in K}\to\bzero$,
it holds that 
$\{\bslambda^{k+1}\}_{k\in K}\to\bsxi^*$, 
where $\xi_i^*:=\min\{\lambda^*_i,\lambda^{\max}\}$, for all $i\in\bm$.
\label{lemma:lambda_convergence_1}
\end{lemma}

\begin{proof}
Lemma~\ref{lemma:x_convergence} guarantees that $\{\bx^k\}\to\bx^*$.
Let $K$ be an infinite index set such that $\{\Dbxak\}_{k\in K}\to\bzero$.
Then, in view of Lemma~\ref{lemma:stationary_multiplier}(ii),
$\{\tbl^{\aff,k+1}\}_{k\in K}\to\bslambda^*\geq\bzero$.  
Accordingly,
$\{\chi_k\}=\{\|\Dbxak\|^\nu+\|[\tbl^{\aff,k+1}_{Q_k}]_-\|^\nu\}\to0$ 
as $k\to\infty$, $k\in K$.  Hence, in view 
of~\eqref{eq:update_lambda_Q} and~\eqref{eq:update_lambda_notQ},
the proof will be complete if we show that 
$\{\thickbreve{\bslambda}^{k+1}\}_{k\in K}\to\bslambda^*$, 
where
\begin{equation*}
\thickbreve{\lambda}^{k+1}_i :=
 \lambda^k_i + \alpha_\dual^{k}\Delta\lambda^k_i\:,\quad i\in Q_k,
 \quand
\thickbreve{\lambda}^{k+1}_i := \mu^{k+1}_{(Q_k)}/s_i^{k+1}\:,\quad i\in\Qc_k\:,
\end{equation*}
or equivalently, 
$\{\tbl^{k+1}-\thickbreve{\bslambda}^{k+1}\}_{k\in K}\to\bzero$,
which we do now.

For every $Q\subseteq\bm$, define the index set $K(Q):=\{k\in K:Q_k=Q\}$,
and let $\cQ:=\{Q\subseteq\bm:|K(Q)|=\infty\}$.  
We first show that for all $Q\in\cQ$, 
$\{\tbl^{k+1}_{Q}-\thickbreve{\bslambda}^{k+1}_Q\}_{k\in K(Q)}\to\bzero$.
For $Q\in\cQ$, 
the definition~\eqref{eq:tblq_tblaq_def} of $\tilde\bslambda^+$ yields
\begin{equation*}
\|\tbl^{k+1}_{Q}-\thickbreve{\bslambda}^{k+1}_Q\|= (1-\alpha_\dual^k)\|\Dblk_{Q}\|,
\quad k\in K(Q)\:.
\end{equation*}
Since boundedness of $\{\bslambda^k_Q\}$ (by construction) and of
$\{\tbl^{k+1}_Q\}_{k\in K}$ (=$\{\bslambda^k_Q+\Dblk_{Q})\}_{k\in K}$)
(by Lemma~\ref{lemma:stationary_multiplier}(ii))
implies 
boundedness of $\{\Dblk_{Q}\}_{k\in K}$, we only need 
$\{\alpha_\dual^k\}_{k\in K(Q)}\to 1$ in order to guarantee that 
$\|\tbl^{k+1}_{Q}-\thickbreve{\bslambda}^{k+1}_Q\|\to0$
on $K(Q)$.  Now, 
$\{\Dbxak\}_{k\in K}\to\bzero$ implies that $\{\Dbsak\}_{k\in K}\to\bzero$,
and from Lemma~\ref{lemma:dx_cvge} that $\{\Dbxk\}_{k\in K}\to\bzero$, 
implying that $\{\Dbsk\}_{k\in K}\to\bzero$; and $\{\bx^k\}\to\bx^*$
yields
$\{\bs^k\}\to \bs^*:=A\bx^*-\bb$,
so $s_i^k+\Delta s_i^k>0$ for all $i\in\cA(\bx^*)^\c$, $k\in K$ large
enough.
Moreover, 
Assumption~\ref{assum:singleton_sol+strict_complementary}
gives $\lambda_i^*>0$ for all $i\in\cA(\bx^*)$
so that, for sufficiently large $k\in K$, 
$\tilde{\lambda}_i^{k+1}>0$ for all $i\in\cA(\bx^*)$,
and Condition~\ref{cond:working_set_GC}(ii) implies that
$\cA(\bx^*)\subseteq Q$, so Lemma~\ref{lemma:step_size_bound} applies,
with $\cA:=\cA(\bx^*)$.  It follows that
$\{\bar{\alpha}_\dual^k\}_{k\in K}\to 1$, 
since all terms on the right-hand side 
of~\eqref{eq:step_size_bound_dual} converge to one on $K$.
Thus, from the definition of $\alpha_\dual^k$ in \eqref{eq:step_size} and the fact that $\{\Dbxk\}_{k\in K}\to\bzero$,
we have $\{\alpha_\dual^k\}_{k\in K}\to 1$ indeed, establishing
that $\{\tbl^{k+1}_{Q}-\thickbreve{\bslambda}^{k+1}_Q\}_{k\in K(Q)}\to\bzero$.

It remains to show that, for all $Q\in\cQ$,
$\{\tbl^{k+1}_{\Qc}-\thickbreve{\bslambda}^{k+1}_{\Qc}\}_{k\in K(Q)}\to\bzero$.
To show this, we first note that, since $\{\chi_k\}_{k\in K}\to 0$, it
follows from~\eqref{eq:update_lambda_Q} and~\eqref{eq:update_lambda_notQ} 
and the fact established above that 
$\{\thickbreve{\bslambda}^{k+1}_Q\}_{K(Q)}\to\bslambda_Q^*$
that, for all $Q\in\cQ$, 
\begin{equation}
\{\bslambda^{k+1}_{Q}\}\to\bsxi^*_{Q}, \quad k\to\infty,~k\in K(Q).
\label{eq:lambda->xi}
\end{equation}
Next, from \eqref{eq:update_lambda_Q}, \eqref{eq:update_lambda_notQ}, 
and the definition~\eqref{eq:tblq_tblaq_def} of $\tilde\bslambda^+$, 
we have, for $Q\in\cQ$ and sufficiently large $k\in K(Q)$,
\begin{equation}
|\tilde\lambda_i^{k+1} - \thickbreve{\lambda}_i^{k+1}| = \thickbreve{\lambda}_i^{k+1}
= \frac{\mu_{(Q)}^{k+1}}{s_i^{k+1}},
\quad i\in \Qc\:.
\label{eq:lambda_not_barQ_cvgce}
\end{equation}
Clearly, since $\cA(\bx^*)\subseteq Q$, we have $s_i^*>0$ for $i\in \Qc$.
Hence, since $\{\Dbsk\}_{k\in K}\to\bzero$, $\{s_i^{k+1}\}$ is bounded 
away from zero on $K$ for $i\in \Qc$.
When $Q$ is empty, the right-hand side of \eqref{eq:lambda_not_barQ_cvgce} is 
set to zero (see definition~\eqref{eq:mu_Q} of $\mu_{(Q)}$).
When $Q$ is not empty, since $\xi_i^*=0$ whenever $\lambda_i^*=0$,
\eqref{eq:lambda->xi} and complementary slackness gives
\begin{equation*}
\left\{\mu_{(Q)}^{k+1}\right\}
=\left\{\frac{(\bs^{k+1}_{Q})^T\bslambda^{k+1}_{Q}}
{|Q|}\right\}\to \left\{\frac{(\bs^*_Q)^T \bsxi^*_Q}
{|Q|}\right\}=0, \quad k\in K(Q) \:,
\end{equation*}
and it follows from~\eqref{eq:lambda_not_barQ_cvgce} that
$\{\tbl^{k+1}_{\Qc}-\thickbreve{\bslambda}^{k+1}_{\Qc}\}_{k\in K(Q)}\to\bzero$,
completing the proof.
\end{proof}
Claim~(ii) of Theorem~\ref{thm:convergence} can now be proved.
\begin{lemma}
\label{lemma:lambda_convergence}
Suppose Assumptions~\ref{assum:non_empty_bdd_sol}
to~\ref{assum:singleton_sol+strict_complementary} hold
and let $(\bx^*,\bslambda^*)$ be the unique primal-dual solution.
Then 
$\{\tblk\}\to\bslambda^*$ and $\{\bslambda^{k}\}\to\bsxi^*$, with $\xi_i^*:=\min\{\lambda^*_i,\lambda^{\max}\}$ for all $i\in\bm$.
\end{lemma}
\begin{proof}
Again, Lemma~\ref{lemma:x_convergence} guarantees that $\{\bx^k\}\to\bx^*$
and $\{\bs^k\}\to\bs^*:=A\bx^*-\bb$.
Note that if $\{\Dbxak\}\to\bzero$, the claims are immediate consequences of Lemmas~\ref{lemma:stationary_multiplier} and~\ref{lemma:lambda_convergence_1}.
We now prove by contradiction that $\{\Dbxak\}\to\bzero$.
Thus, suppose that for some infinite
index set $K$, $\inf_{k\in K}\|\Dbxak\|>0$.
Then, Lemma~\ref{lemma:dx_cvge} gives $\inf_{k\in K}\|\Dbxk\|>0$.
It follows from Proposition~\ref{prop:dx_to_zero} that, on some infinite index 
set $K^\prime\subseteq K$, $\{\Dbx^{\aff,k-1}\}\to \bzero$ 
and $\{[\tblak]_-\}\to\bzero$. 
Since $Q_k$ is selected from a finite set and $\{W_k\}$ is bounded, 
we can assume without loss of generality that $Q_k=Q$ on $K^\prime$ 
for some $Q \subseteq \bm$, and that $\{W_k\}\to W^*\succeq H$ on $K^\prime$.
%
Further, 
%
from Lemma~\ref{lemma:lambda_convergence_1}, 
$\{\bslambda^k\}_{k\in K^\prime}\to\bsxi^*$.
Therefore, $\{J(W_k,A_{Q_k},\bs_{Q_k},\bslambda_{Q_k})\}_{k\in K^\prime}
\to J(W^*,A_{Q},\bs_Q^*,\bsxi_Q^*)$, and in view of
Assumptions~\ref{assum:lin_ind_act} 
and~\ref{assum:singleton_sol+strict_complementary} and
Lemma~\ref{lemma:nonsingular_J}, $J(W^*,A_{Q},\bs_Q^*,\bsxi_Q^*)$ is
non-singular (since $(\bx^*,\bslambda^*)$ is optimal).
It follows from~\eqref{eq:CR_KKT}, with $W$ substituted for $H$,
that $\{\Dbxak\}\to\bzero$ on $K^{\prime}$, 
a contradiction, proving that
$\{\Dbxak\}\to\bzero$.
\end{proof}
%
%
Claim~(iv) of Theorem~\ref{thm:convergence} follows as well.

\begin{lemma}
Suppose Assumptions~\ref{assum:non_empty_bdd_sol}
and~\ref{assum:lin_ind_act} hold and $\varepsilon>0$.
Then Algorithm~\ref{algo:CR-MPC} terminates (in Step~1) after 
finitely many iterations.
\label{lemma:termination}
\end{lemma}
\begin{proof}
If $\{(\bx^k,\bslambda^k)\}$ has a limit point in $\cF^*$, then $\inf_k\{E_k\}=0$,
proving the claim.
Thus, suppose that $\{(\bx^k,\bslambda^k)\}$ is bounded away from $\cF^*$.
In view of Lemmas~\ref{lemma:bdd_seq_x} and~\ref{lemma:x_convergence},
$\{\bx^k\}$ has a limit point $\bx^*\in\cF_P^*$.
Assumption~\ref{assum:lin_ind_act} then implies that there exists
a unique KKT multiplier vector $\bslambda^*\geq\bzero$ associated to $\bx^*$.
If $(\bx^*,\bslambda^*)\in\cF^*$ is a limit point of $\{(\bx^k, \tblk)\}$,
which also implies that
$\inf_k\{E(\bx^k, \tblk)\}=0$,
then in view of the stopping criterion, the claim
again follows.
Thus, further suppose that there is an infinite index set $K$
such that $\{\bx^k\}_{k\in K}\to\bx^*$,
but $\inf_{k\in K}\|\tblk-\bslambda^*\|>0$.
It then follows from Lemma~\ref{lemma:stationary_multiplier}(ii)
that $\{\Dbx^{\aff,k-1}\}_{k\in K}\not\to\bzero$, and from
Lemma~\ref{lemma:dx_cvge} that $\{\Dbx^{k-1}\}_{k\in K}\not\to\bzero$.
Proposition~\ref{prop:dx_to_zero} and Lemma~\ref{lemma:dx_cvge} then imply that
$\{\Dbx^{\aff,k-2}\}_{k\in K^\prime}\to\bzero$ and
$\{\Dbx^{k-2}\}_{k\in K^\prime}\to\bzero$
for some infinite index set $K^\prime\subseteq K$.
Next, from Lemmas~\ref{lemma:bdd_seq_x} and~\ref{lemma:x_convergence},
we have $\{\bx^{k-2}\}_{k\in K^{\prime\prime}}\to\bx^{**}\in\cF_P^*$
for some infinite index
set $K^{\prime\prime}\subseteq K^\prime$, and in view
of Lemma~\ref{lemma:stationary_multiplier}(ii)
$\{\tilde\bslambda^{k-1}\}_{k\in K^{\prime\prime}}\to\bslambda^{**}$,
where $\bslambda^{**}$ is
the KKT multiplier associated to $\bx^{**}$.
Since $\alpha_\prim^{k}\in[0,1]$ for all $k$, we also have
$\{\bx^{k-1}\}_{k\in K^{\prime\prime}}
=\{\bx^{k-2}+\alpha_\prim^{k-2}\Dbx^{k-2}\}_{k\in K^{\prime\prime}}\to\bx^{**}$,
i.e.,
$\{(\bx^{k-1},\tilde\bslambda^{k-1})\}_{k\in K^{\prime\prime}}\to(\bx^{**},\bslambda^{**})\in\cF^*$,
completing the proof. 
\end{proof}

\medskip
{\noindent\bf Proof of Theorem~\ref{thm:convergence}.}
Claim~(i) was proved in Lemma~\ref{lemma:x_convergence} and
Claim~(ii) in Lemma~\ref{lemma:lambda_convergence},
Claim~(iii) is a direct consequence of Condition~\ref{cond:working_set_GC}(ii),
and Claim~(iv) was proved in Lemma~\ref{lemma:termination}.

\medskip

{\noindent\bf Proof of Corollary~\ref{cor:exact_working_set}.}
From Theorem~\ref{thm:convergence}, 
$\{(\bx^k,\bslambda^k)\}\to(\bx^*,\bslambda^*)$, i.e., $\{E_k\}\to0$.
It follows that (i) in view of Proposition~\ref{prop:rule_cond_GC} 
and Condition~\ref{cond:working_set_GC}(ii) $Q_k\supseteq\cA(x^*)$ 
for all $k$ large enough, and (ii) in view of {\rulename}, 
$\{\delta_k\}\to0$, so that $Q_k$ eventually excludes all indexes that 
are not in $\cA(\bx^*)$. \qedsym


\section{Proof of Theorem~\ref{thm:q-quad}}
\label{appendix:LQC}

Parts of this proof are adapted from~\cite{JungThesis, TZ:94, WNTO-2012}.
Throughout, we assume that 
Assumption~\ref{assum:singleton_sol+strict_complementary} 
holds (so that Assumption~\ref{assum:non_empty_bdd_sol} also holds),
that $\varepsilon=0$ {and that the iteration never stops},
and that $\lambda^*_i<\lambda^{\max}$ for all $i$. 

Newton's method plays the central role in the local
analysis.  The following lemma
is standard or readily proved; see, e.g.,~\cite[Proposition 3.10]{TZ:94}.

\begin{lemma}
\label{coro:Newton_conv_2}
Let $\Phi\colon\bbR^{n}\to\bbR^{n}$ be twice continuously differentiable 
and let $\bt^*\in\bbR^n$ such that $\Phi(\bt^*)=\bzero$.
Suppose there exists $\rho>0$ such that $\frac{\p\Phi}{\p\bt}(\bt)$ is
non-singular for all $\bt\in B(\bt^*,\rho)$.
Define $\Delta^{\rm N}\bt$ to be the Newton increment at $\bt$, i.e.,
$\Delta^{\rm N}\bt=-\left(\frac{\p\Phi}{\p\bt}(\bt)\right)^{-1}\Phi(\bt)$.
Then, given any $c>0$, there exists $c^*>0$ such that, 
for all $\bt\in B(\bt^*,\rho)$, if $\bt^+\in\bbR^n$ 
satisfies 
\begin{equation}
\min\{|t_i^+-t_i^*|,|t_i^+-(t_i+(\Delta^{\rm N} t)_i)|\}\leq c \, 
\max\{\|\Delta^{\rm N}\bt\|^2, \|\bt-\bt^*\|^2\},\quad i=1,\ldots,n\:,
\label{eq:Newton_conv_assum_2}
\end{equation}
then
\begin{equation*}
\|\bt^+-\bt^*\|\leq c^* \|\bt-\bt^*\|^2\:.
\end{equation*}
\end{lemma}

For convenience, 
define $\bz:=(\bx, \bslambda)$ (as well as $\bz^*:=(\bx^*,\bslambda^*)$, etc.).
For $\bz\in \cF^o:=\{\bz:\bx\in\cF_P^o,\,\bslambda>\bzero\}$, define
\begin{equation*}
\varrho(\bz):=\min\left\{1,\frac{E(\bx,\bslambda)}{\bar E}\right\}
\quand W(\bz):=H+\varrho(\bz)R\:. \quad 
\end{equation*}
The gist of the remainder of this appendix is to 
apply Lemma~\ref{coro:Newton_conv_2} to 
\begin{equation*}
\Phi_Q(\bz):=\left[\begin{array}{c}
H\bx-(A_Q)^T\bslambda_Q+\bc\\
\Lambda_Q(A_Q\bx-\bb_Q)
\end{array}\right]\:, \quad Q\subseteq\bm.
\end{equation*}
(Note that $\Phi_Q(\bz^*)=\bzero$.)
Let $\bz_Q:=(\bx, \bslambda_Q)$, then the step taken on the $Q$ components along the search direction generated by the Algorithm~\ref{algo:CR-MPC} is 
analogously given by
$\thickbreve{\bz}^+_Q : = (\bx^+,\thickbreve\bslambda_Q^+)$
with $\thickbreve\bslambda_Q^+:=\bslambda_Q + \alpha_\dual\Dbl_Q$.
The first major step of the proof is achieved by
Proposition~\ref{prop:MPC_and_Newton_bound} below, where the focus is on
$\thickbreve{\bz}^+_Q$ rather than on $\bz^+$.
Thus we compare $\thickbreve{\bz}^+_Q$, with $Q\in\cQ^*$
to the $Q$ components of the (unregularized) Newton step,
i.e., $\bz_Q+(\Delta^{\rm N}\bz)_Q$.
Define
\begin{equation*}
\mathscr{A}:=\left[ \begin{array}{cc}
\alpha_\prim I_n & \bzero \\
\bzero & \alpha_\dual I_{|Q|}
\end{array} \right]\:,\quand
\alpha:=\min\{\alpha_\prim,\alpha_\dual\}\:.
\end{equation*}
The difference between the CR-MPC iteration and the Newton iteration
can be written as
\begin{equation}
\begin{alignedat}{2}
&\|\thickbreve{\bz}^+_Q-(\bz_Q+(\Delta^{\rm N}\bz)_Q)\|\\
&\leq \|\thickbreve{\bz}^+_Q-(\bz_Q+\Dbz_Q)\|+ \|\Dbz_Q-\Dbza_Q\| +
\|\Dbza_Q - \Dbz^0_Q\| + \|\Dbz^0_Q-(\Delta^{\rm N}\bz)_Q\| \\
&= \|(I-\mathscr{A})\Dbz_Q\|+ \gamma\|\Dbzc_Q\| +  \|\Dbza_Q - \Dbz^0_Q\| +
\|\Dbz^0_Q-(\Delta^{\rm N}\bz)_Q\|\\
&\leq (1-\alpha) \|\Dbz_Q\|+ \|\Dbzc_Q\| + \|\Dbza_Q - \Dbz^0_Q\| + \|\Dbz^0_Q-(\Delta^{\rm N}\bz)_Q\|\:,
\end{alignedat}
\label{eq:MPC_and_Newton}
\end{equation}
where $\Dbz_Q:=(\Dbx,\Dbl_Q)$, $\Dbza_Q:=(\Dbx^{\aff},\Dbl^\aff_Q)$, $\Dbz^\cor_Q:=(\Dbx^{\cor},\Dbl^\cor_Q)$, and $\Dbz^0_Q$ is the (constraint-reduced) affine-scaling direction for the original (unregularized) system (so $\Delta^{\rm N}\bz=\Dbz^0_\bm$).

Let
\[
J_a(W,A,\bs,\bslambda) := 
\left[\begin{array}{cc}    
W & -A^T \\    
\Lambda A   & S 
\end{array}\right]\:.
\]
The following readily proved lemma will be of help.  (For details, 
see 
Lemmas B.15 and B.16 in~\cite{JungThesis}; also Lemmas~13 
and~1 in~\cite{TAW-06})

\begin{lemma}
Let $\bs,\bslambda\in\bbR^m$ and $Q\subseteq\bm$ be arbitrary and let $W$ 
be symmetric, with $W\succeq H$.
Then (i) $J_a(W,A_Q,\bs_Q,\bslambda_Q)$ is non-singular
if and only if $J(W,A_Q,\bs_Q,\bslambda_Q)$ is, and
(ii) 
if $\cA(\bx^*)\subseteq Q$, then $J(W,A_Q,\bs_Q^*,\bslambda_Q^*)$ 
is non-singular
(and so is $J_a(W,A_Q,\bs_Q^*,\bslambda_Q^*)$).
\label{lemma:nonsingular_Ja}
\end{lemma}

\noindent
With $\bs:=A\bx-\bb$,
$J_a(H,A_Q,\bs_Q,\bslambda_Q)$, the system matrix for the
(constraint-reduced) original (unregularized) 
``augmented'' system, is the Jacobian of
$\Phi_Q(\bz)$, i.e.,
\begin{equation*}
J_a(H,A_Q,\bs_Q,\bslambda_Q) \Dbz^0_Q = -\Phi_Q(\bz)\:,
\end{equation*}
and its regularized version $J_a(W(\bz),A_Q,\bs_Q,\bslambda_Q)$
satisfies (among other systems solved by Algorithm~\ref{algo:CR-MPC})
\[
J_a(W(\bz),A_Q,\bs_Q,\bslambda_Q) \Dbza_Q = -\Phi_Q(\bz)\:.
\]
Next, we verify that $J_a(W(\bz),A_Q,\bs_Q,\bslambda_Q)$ is 
non-singular near $\bz^*$ (so that $\Dbz^0_Q$ and $\Dbza_Q$
in~\eqref{eq:MPC_and_Newton} are well defined)
and establish other useful local properties.  
For convenience, we define
\begin{equation*}
\cQ^*:=\{Q\subseteq\bm:\cA(\bx^*)\subseteq Q\}\:.
\end{equation*}
and
\[
\tbs^+ := \bs + \Dbs,  \qquad \tilde{\bs}^{\aff,+} := \bs + \Dbsa.
\]

\begin{lemma}

Let $\epsilon^*:=\min\{1,\min_{i\in\bm}(\lambda_i^*+s_i^*)\}$.
There exist $\rho^*>0$ and $r>0$, such that, 
for all $\bz\in \cF^o\cap B(\bz^*,\rho^*)$
and all $Q\in\cQ^*$, the following hold:
\begin{enumerate}[label=(\roman*)]
\item $\|J_a(W(\bz),A_Q,\bslambda_Q,\bs_Q)^{-1}\|\leq r$,
\item $\max\{\|\Dbza_Q\|, \|\Dbz_Q\|, \|\Dbsa_Q\|, \|\Dbs_Q\|\}<\epsilon^*/4$,
\item $\min\{\lambda_i, \tilde{\lambda}_i^{\aff,+} ,  \tilde{\lambda}_i^+ \}>\epsilon^*/2,\,\forall i\in\cA(\bx^*)$,\\
$\max\{ \lambda_i ,  \tilde{\lambda}_i^{\aff,+} ,  \tilde{\lambda}_i^+ \}<\epsilon^*/2,\,\forall i\in\bm\setminus\cA(\bx^*)$,\\
$\max\{ s_i ,  \tilde{s}_i^{\aff,+} ,  \tilde{s}_i^+ \}<\epsilon^*/2,\,\forall i\in\cA(\bx^*)$,\\
$\min\{ s_i ,  \tilde{s}_i^{\aff,+} ,  \tilde{s}_i^+ \}>\epsilon^*/2,\,\forall i\in\bm\setminus\cA(\bx^*)$.

\item $\tilde\lambda_i^+ < \lambda^{\max},\,\forall i\in\bm$.
\end{enumerate}
\label{lemma:rho_star}
\end{lemma}

\begin{proof}

Claim (i) follows from Lemma~\ref{lemma:nonsingular_Ja}, continuity 
of $J_a(W(\bz),A_Q,\bslambda_Q,\bs_Q)$ (and the fact that $W(\bz^*)=H$).
Claims (ii) and (iv) follow from Claim (i), Lemma~\ref{lemma:nonsingular_Ja}, 
continuity of the right-hand sides of \eqref{eq:CR_KKT} 
and~\eqref{eq:CR_KKT_dx}, which are zero at the solution,
definition~\eqref{eq:tblq_tblaq_def} of $\tilde\bslambda^+$,
and our assumption that $\lambda^*_i < \lambda^{\max}$ for all $i\in\bm$.
Claim (iii) is true due to strict complementary slackness, the 
definition of $\epsilon^*$, and Claim (ii).
\end{proof}

In preparation for Proposition~\ref{prop:MPC_and_Newton_bound}, 
Lemmas~\ref{lemma:cor_dir_bound}--\ref{lemma:aff_and_Newton_bound}
provide bounds on the four terms in the last line of \eqref{eq:MPC_and_Newton}. 
The $\rho^*$ used in these lemmas comes 
from Lemma~\ref{lemma:rho_star}.
The proofs of Lemmas~\ref{lemma:cor_dir_bound}, 
\ref{lemma:min_step_size_bound}, and~\ref{lemma:aff_and_Newton_bound}
are omitted, as they are very similar to those of Lemmas~A.9 and A.10
in the supplementary materials of \cite{WNTO-2012} (where an MPC 
algorithm for linear optimization problems is considered) 
and of Lemma~B.19 in \cite{JungThesis} (also Lemma~16 in\cite{TAW-06}).

\begin{lemma}
\label{lemma:cor_dir_bound}
There exists a constant $c_1>0$ such that, for all $\bz\in \cF^o\cap B(\bz^*,\rho^*)$, and for all $Q\in\cQ^*$,
\begin{equation*}
\|\Dbzc_Q\| \leq c_1 \|\Dbza_Q\|^2\:.
\end{equation*}
\end{lemma}
Note that an upper bound on the magnitude of the MPC search 
direction $\Dbz_Q$ can be obtained by using Lemma~\ref{lemma:cor_dir_bound} 
and Lemma~\ref{lemma:rho_star}(ii), viz.
\begin{equation}
\|\Dbz_Q\| \leq \|\Dbza_Q\| + \|\Dbzc_Q\| 
\leq \|\Dbza_Q\|+c_1\|\Dbza_Q\|^2
\leq \left(1+c_1\frac{\epsilon^*}{4}\right)\|\Dbza_Q\|\:.
\label{eq:MPC_dir_bound}
\end{equation}
This bound is used in the proofs of 
Lemma~\ref{lemma:min_step_size_bound} and Proposition~\ref{prop:MPC_and_Newton_bound}.

\begin{lemma}
\label{lemma:min_step_size_bound}
There exists a constant $c_2>0$ such that, 
for all $\bz\in \cF^o\cap B(\bz^*,\rho^*)$, and for all $Q\in\cQ^*$,
\begin{equation*}
|1-\alpha| \leq c_2 \|\Dbza_Q\|\:.
\end{equation*}
\end{lemma}

\begin{lemma}
\label{lemma:reg_err_bound}
There exists a constant $c_3>0$ such that, for all $\bz\in \cF^o\cap B(\bz^*,\rho^*)$ and all $Q\in\cQ^*$,
\begin{equation*}
\|\Dbza_Q - \Dbz^0_Q\| \leq c_3 \|\bz-\bz^*\|^2.
\end{equation*}
\end{lemma}
\begin{proof}
We have
\begin{equation*}
\Dbza_Q - \Dbz^0_Q
= - (J_a(W(\bz),A_Q,\bs_Q,\bslambda_Q)^{-1} 
    - J_a(H,A_Q,\bs_Q,\bslambda_Q)^{-1}) \Phi_Q(\bz)
\end{equation*}
so that there exist $c_{31}>0$ 
such that,
for all $\bz\in \cF^o\cap B(\bz^*,\rho^*)$ and all $Q\in\cQ^*$,
\begin{equation*}
\|\Dbza_Q - \Dbz^0_Q\|
\leq c_{31} \|W(\bz)-H\| \|\bz-\bz^*\|\:,
\end{equation*}
where the second inequality follows from Lemma~\ref{lemma:rho_star}(i).
Since $W(\bz)-H=\varrho(\bz)R$, $|\varrho(\bz)|\leq c_{32} |E(\bz)|$, and
$|E(\bz)|\leq c_{33}\|\bz-\bz^*\|$, for some $c_{32}>0$ and $c_{33}>0$, 
the proof is complete.
\end{proof}

\begin{lemma} 
There exists a constant $c_4>0$ such that, for all $\bz\in \cF^o\cap B(\bz^*,\rho^*)$, and for all $Q\in\cQ^*$,
\begin{equation*}
\|\Dbz^0_Q-(\Delta^{\rm N}\bz)_Q\|\leq c_4 \|\bz-\bz^*\| \|(\Delta^{\rm N}\bz)_Q\|\:.
\end{equation*}
\label{lemma:aff_and_Newton_bound}
\end{lemma}
\noindent
With Lemmas~\ref{lemma:cor_dir_bound}--\ref{lemma:aff_and_Newton_bound}
in hand, 
we return to inequality~\eqref{eq:MPC_and_Newton}.

\begin{proposition}
\label{prop:MPC_and_Newton_bound}
There exists a constant $c_5>0$ such that, for all $\bz\in \cF^o\cap B(\bz^*,\rho^*)$, and for all $Q\in\cQ^*$,
\begin{equation}
\|\thickbreve{\bz}^+_Q-(\bz_Q+(\Delta^{\rm N}\bz)_Q)\| \leq c_5 \max\{\|\Delta^{\rm N}\bz\|^2,\|\bz-\bz^*\|^2\}\:.
\label{eq:Lemma 21}
\end{equation}
\end{proposition}
\begin{proof}
Let $\bz\in \cF^o\cap B(\bz^*,\rho^*)$ and $Q\in\cQ^*$.
It follows from \eqref{eq:MPC_and_Newton}, 
Lemmas~\ref{lemma:cor_dir_bound}--\ref{lemma:aff_and_Newton_bound},
and \eqref{eq:MPC_dir_bound} that
\begin{equation*}
\begin{alignedat}{2}
\|\thickbreve{\bz}^+_Q-(\bz_Q+(\Delta^{\rm N}\bz)_Q)\|
&\leq(1-\alpha) \|\Dbz_Q\|+ \|\Dbzc_Q\| + \|\Dbza_Q - \Dbz^0_Q\| + \|\Dbz^0_Q-(\Delta^{\rm N}\bz)_Q\|\\
&\leq c_2\|\Dbza_Q\| \|\Dbz_Q\| + c_1\|\Dbza_Q\|^2 + c_3 \|\bz-\bz^*\|^2 + c_4 \|\bz-\bz^*\| \|(\Delta^{\rm N}\bz)_Q\|\\
&\leq \left(c_2\left(1+c_1\frac{\epsilon^*}{4}\right) + c_1\right)\|\Dbza_Q\|^2 + c_3 \|\bz-\bz^*\|^2 + c_4 \|\bz-\bz^*\| \|(\Delta^{\rm N}\bz)_Q\|\:.
\end{alignedat}
\end{equation*}
Also, by Lemmas~\ref{lemma:reg_err_bound} and \ref{lemma:aff_and_Newton_bound}, we have
\begin{equation}
\begin{alignedat}{2}
\|\Dbza_Q\| &\leq 
\|\Dbza_Q-\Dbz^0_Q\| + \|\Dbz^0_Q-(\Delta^{\rm N}\bz)_Q\|
+ \|(\Delta^{\rm N}\bz)_Q\| \\
&\leq 
c_3 \|\bz-\bz^*\|^2 + c_4\|\bz-\bz^*\|\|(\Delta^{\rm N}\bz)_Q\|
+ \|(\Delta^{\rm N}\bz)_Q\| \:.
\end{alignedat}
\label{eq:AS_Newton_bound}
\end{equation}
The claim follows (in view of boundedness of $\cF^o\cap B(\bz^*,\rho^*)$).
\end{proof}

With Proposition~\ref{prop:MPC_and_Newton_bound} established, 
we proceed to the second major step of the proof of 
Theorem~\ref{thm:q-quad}: to show that~\eqref{eq:Lemma 21} 
still holds when $\bz^+$
is substituted for $\thickbreve{\bz}^+_Q$.

\medskip
{\noindent\bf Proof of Theorem~\ref{thm:q-quad}.}
Again, let $\rho^*$ be as given in Lemma~\ref{lemma:rho_star}.
Let $\bz\in \cF^o\cap B(\bz^*,\rho^*)$ and $Q\in\cQ^*$.
Let $\rho:=\rho^*$, $\bt:=\bz$, and $\bt^*:=\bz^*$. 
Then the desired q-quadratic convergence is a direct consequence of Lemma~\ref{coro:Newton_conv_2}, provided that the condition \eqref{eq:Newton_conv_assum_2} is satisfied.
Hence, we now show that there exists some constant $c>0$ such that, 
for each $i\in\bm$,
\begin{equation}
\min\{|z_i^+-z_i^*|,|z_i^+-(z_i+(\Delta^{\rm N} z)_i)|\}\leq c \,\max\{\|\Delta^{\rm N}\bz\|^2,\|\bz-\bz^*\|^2\}\:.
\label{eq:MPC_conv_assum}
\end{equation}
As per Proposition~\ref{prop:MPC_and_Newton_bound}, \eqref{eq:MPC_conv_assum}
holds for $i\in Q$ with $z_i^+$ replaced with $\thickbreve z_i^+$.
In particular,
\eqref{eq:MPC_conv_assum} holds for the $\bx^+$ components of $\bz^+$.
It remains to show that \eqref{eq:MPC_conv_assum} holds for 
the $\bslambda^+$ components of $\bz^+$.
Firstly, for all $i\in\cA(\bx^*)$, we show that $\lambda_i^+=\thickbreve{\lambda}_i^+$, thus \eqref{eq:MPC_conv_assum} holds for all $\lambda^+_i$ such that $i\in\cA(\bx^*)$ by Proposition~\ref{prop:MPC_and_Newton_bound}.
From the fact that $\bslambda>\bzero$ ($\bz\in\cF^o$) and Lemma~\ref{lemma:rho_star}(ii), 
and since $\nu\geq 2$, it follows that
\begin{equation}
\chi:=\|\Dbxa\|^\nu + \|[\tilde{\bslambda}^{\aff,+}_Q]_-\|^\nu 
\leq \|\Dbxa\|^\nu + \|\Dbla_Q\|^\nu
\leq 2\left(\frac{\epsilon^*}{4}\right)^\nu
\leq \frac{\epsilon^*}{2}\:,
\label{eq:conv_thm_in_A_1}
\end{equation}
so that $\min\{\chi,\underline\lambda\}\leq\epsilon^*/2$.
Also, from Lemma~\ref{lemma:rho_star}(iii) and the fact that $\thickbreve{\bslambda}_Q^+$ is a convex combination of $\bslambda_Q$ and $\tbl_Q^+$, we have, for all $i\in\cA(\bx^*)$,
\begin{equation}
\frac{\epsilon^*}{2}< \min\{\lambda_i,\tilde{\lambda}_i^+\}\leq \thickbreve{\lambda}_i^+\:.
\label{eq:conv_thm_in_A_2}
\end{equation}
Hence, from \eqref{eq:conv_thm_in_A_1}, \eqref{eq:conv_thm_in_A_2}, 
Lemma~\ref{lemma:rho_star}(iv), and \eqref{eq:update_lambda_Q}, we 
conclude that $\lambda_i^+=\thickbreve{\lambda}_i^+$ for all $i\in\cA(\bx^*)$. 
Secondly,
we prove that there exists $d_1>0$ such that
\begin{equation}
\|\bslambda_{Q\setminus\cA(\bx^*)}^+\|
= \|\bslambda_{Q\setminus\cA(\bx^*)}^+ 
- \bslambda_{Q\setminus\cA(\bx^*)}^*\|
\leq d_1\max\{\|\Delta^{\rm N}\bz\|^2,\|\bz-\bz^*\|^2\} 
\quad\forall i\in Q\setminus\cA(\bx^*)\:,
\label{eq:conv_thm_in_Q}
\end{equation}
thus establishing~\eqref{eq:MPC_conv_assum} for $\lambda_i^+$ with
$i\in Q\setminus\cA(\bx^*)$.
For $i\in Q\setminus\cA(\bx^*)$, we know from \eqref{eq:update_lambda_Q} that, 
either $\lambda_i^+=\min\{\lambda^{\max}, \thickbreve{\lambda}_i^+\}$, or 
$\lambda_i^+ = \min\{\underline{\lambda},\|\Dbxa\|^\nu+\|[\tilde{\bslambda}^{\aff,+}_Q]_-\|^\nu\}$.
In the former case,
we have
\begin{equation*}
\begin{alignedat}{2}
|\lambda_i^+|\leq
|\thickbreve{\lambda}_i^+|=
|\thickbreve{\lambda}_i^+-\lambda_i^*|
&\leq 
|\thickbreve{\lambda}_i^+-(\lambda_i+(\Delta^{\rm N}\lambda)_i)| +
|(\lambda_i+(\Delta^{\rm N}\lambda)_i)-\lambda_i^*|\\
&\leq d_2\max\{\|\Delta^{\rm N}\bz\|^2,\|\bz-\bz^*\|^2\} + d_3\|\bz-\bz^*\|^2\:,
\end{alignedat}
\end{equation*}
for some $d_2>0$, $d_3>0$. 
Here the last inequality follows from 
Proposition~\ref{prop:MPC_and_Newton_bound} and the quadratic 
rate of the Newton step given in 
Lemma~\ref{coro:Newton_conv_2}.
In the latter case, since $\bslambda>\bzero$, we obtain 
\begin{equation}
|\lambda_i^+| \leq \|\Dbxa\|^\nu + \|[[\tilde{\bslambda}^{\aff,+}_Q]_-\|^\nu
\leq \|\Dbxa\|^\nu + \|\Dbla_Q\|^\nu = \|\Dbza_Q\|^\nu
\leq d_4\max\{\|\Delta^{\rm N}\bz\|^2,\|\bz-\bz^*\|^2\}\:,
\label{eq:conv_thm_in_Q_2}
\end{equation}
for some $d_4>0$. 
Here the equality is from the definition of $\Dbza$ and the last 
inequality follows from $\nu\geq2$, \eqref{eq:AS_Newton_bound},
and boundedness of $\cF^o\cap B(\bz^*,\rho^*)$.
Hence, we have established \eqref{eq:conv_thm_in_Q}. 
%
Thirdly and finally, consider the case that $i\in \Qc$.
Since $\cA(\bx^*)\subseteq Q$, $\bslambda_{\Qc}^*=\bzero$ and 
it follows from \eqref{eq:update_lambda_notQ} that,
either $\lambda_i^+=\min\{\lambda^{\max}, {\mu_{(Q)}^+}/{s_i^+}\}$, or $\lambda_i^+ = \min\{\underline{\lambda},\|\Dbxa\|^\nu+\|[\tbla_Q]_-\|^\nu\}$.
In the latter case, the bound in \eqref{eq:conv_thm_in_Q_2} follows.
In the former case, we have
\begin{equation*}
|\lambda_i^+ - \lambda_i^*| = |\lambda_i^+| 
\leq {\mu_{(Q)}^+}/{s_i^+} \:.
\end{equation*}
By definition, $s_i^+:=s_i+\alpha_\prim \Delta s_i$ is a convex combination of $s_i$ and $\tilde{s}_i^+$. 
Thus, Lemma~\ref{lemma:rho_star}(iii) gives that $s_i^+\geq\min\{s_i,\tilde{s}_i^+\}>\epsilon^*/2$. 
Then using the definition of $\mu_{(Q)}^+$ (see Step 10 of 
Algorithm~\ref{algo:CR-MPC}) 
leads to
\begin{equation*}
|\lambda_i^+ - \lambda_i^*| \leq 
\begin{cases}
\frac{2}{\epsilon^*|Q|} \left( 
 (\bs_{\cA(\bx^*)}^+)^T  (\bslambda_{\cA(\bx^*)}^+)
+ (\bs_{Q\setminus\cA(\bx^*)}^+)^T (\bslambda_{Q\setminus\cA(\bx^*)}^+)\right)
, & \text{if } |Q|\neq0\\
0\:, & \text{otherwise}
\end{cases}
\:.
\end{equation*}
Since $\bz\in B(\bz^*,\rho^*)$, $\bslambda_{\cA(\bx^*)}^+$ and $\bs_{Q\setminus\cA(\bx^*)}^+$ are bounded by Lemma~\ref{lemma:rho_star}(ii). 
Also, by definition, $\bs_{\cA(\bx^*)}^*=\bzero$.
Thus
there exist $d_5>0$ and $d_6>0$ such that
\begin{equation*}
|\lambda_i^+ - \lambda_i^*| 
\leq  d_5\|\bs_{\cA(\bx^*)}^+-\bs_{\cA(\bx^*)}^*\|
+ d_6 \|\bslambda_{Q\setminus\cA(\bx^*)}^+\|\:.
\end{equation*}
Having already established that
the second term is bounded by the right-hand side
of~\eqref{eq:conv_thm_in_Q}, and we are left to prove that 
the first term 
also is.
By definition,
\begin{equation*}
\|\bs_{\cA(\bx^*)}^+ - \bs_{\cA(\bx^*)}^*\| 
= \|A_{\cA(\bx^*)}\bx^+ - A_{\cA(\bx^*)}\bx^*\| 
\leq \|A\bx^+ - A\bx^*\|
\leq \|A\| \|\thickbreve{\bz}_Q^+ - \bz_Q^*\|\:.
\end{equation*}
Applying Proposition~\ref{prop:MPC_and_Newton_bound} and 
Lemma~\ref{coro:Newton_conv_2},
we get
\begin{equation*}
\begin{alignedat}{2}
\|\bs_{\cA(\bx^*)}^+ - \bs_{\cA(\bx^*)}^*\| 
&\leq \|A\| \|\thickbreve{\bz}_Q^+ - (\bz_Q+(\Delta^{\rm N}\bz)_Q)\| + \|A\| \|(\bz_Q+(\Delta^{\rm N}\bz)_Q) - \bz_Q^*\|\\
&\leq d_7 \max\{\|\Delta^{\rm N}\bz\|^2,\|\bz-\bz^*\|^2\} + d_8 \|\bz - \bz^*\|^2\:,
\end{alignedat}
\end{equation*}
for some $d_7>0$, $d_8>0$.
Hence, we established \eqref{eq:MPC_conv_assum} for all $i\in\bm$, 
thus proving the q-quadratic convergence rate. \qedsym

\bibliographystyle{spmpsci_NoUrlDoi}
\bibliography{PFPN_ref}

\end{document}